\documentclass[11pt,a4paper]{amsart}

\usepackage{graphicx}
\usepackage{enumerate}
\usepackage{amsmath, amssymb, amsthm, amscd}
\usepackage{mathtools}
\usepackage[usenames, dvipsnames]{xcolor}
\usepackage{caption}
\usepackage{subcaption}
\usepackage{multirow}
\usepackage{extarrows}
\usepackage{stmaryrd}
\usepackage{tikz-cd}
\usepackage[pdfencoding=auto, psdextra, hidelinks]{hyperref}

\theoremstyle{plain}
\newtheorem{theorem}{Theorem}[section]
\newtheorem{lemma}[theorem]{Lemma}
\newtheorem{corollary}[theorem]{Corollary}
\newtheorem{proposition}[theorem]{Proposition}
\newtheorem{hypothesis}[theorem]{Hypothesis}

\theoremstyle{definition}
\newtheorem{definition}[theorem]{Definition}

\theoremstyle{remark}
\newtheorem{remark}[theorem]{Remark}

\numberwithin{equation}{section}

\def\softd{{\leavevmode\setbox1=\hbox{d}%
              \hbox to 1.05\wd1{d\kern-0.4ex{\char039}\hss}}}

\def\bold#1{\mbox{\boldmath $#1$}}
\newcommand{\uu}[1]{\bold{#1}}

\newcommand{\ldblbrace}{\{\mskip-5mu\{}
\newcommand{\rdblbrace}{\}\mskip-5mu\}}

\newcommand{\R}{\mathbb{R}}
\newcommand{\N}{\mathbb{N}}
\newcommand{\Cinf}{C_c^\infty}
\newcommand{\tdom}{(0,T)}
\newcommand{\odom}{\tdom\times\Omega}
\newcommand{\dom}{\lbrack 0,T\rbrack\times\Omega}

\newcommand{\bu}{{\uu{u}}}
\newcommand{\bv}{\uu{v}}
\newcommand{\bw}{\uu{w}}

\newcommand{\bm}{{\uu{m}}}

\newcommand{\bU}{{\uu{U}}}

\newcommand{\du}{\delta\bu}
\newcommand{\F}{\mathcal{F}} 
\newcommand{\D}{\mathcal{D}} 
\newcommand{\M}{\mathcal{M}} 
\newcommand{\K}{\mathcal{K}} 
\newcommand{\E}{\mathcal{E}}
\newcommand{\Pro}{\mathcal{P}} 
\newcommand{\veps}{\varepsilon}
\newcommand{\vrho}{\varrho}
\newcommand{\vth}{\vartheta}
\newcommand{\uphi}{\uu{\varphi}}
\newcommand{\epso}{{ 0,\varepsilon}}
\newcommand{\half}{\frac{1}{2}}
\newcommand{\dx}{\mathrm{d}x}
\newcommand{\dt}{\mathrm{d}t}

\newcommand{\psig}{\psi_\gamma}
\newcommand{\pig}{\Pi_\gamma}
\newcommand{\V}{\mathcal{V}_{t,x}}
\newcommand{\U}{\mathcal{U}_{t,x}}
\newcommand{\mcv}{\mathcal{V}}
\newcommand{\mcu}{\mathcal{U}}
\newcommand{\tvrho}{\tilde{\vrho}}
\newcommand{\tbm}{\widetilde{\bm}}
\newcommand{\tbv}{\tilde{\bv}}
\newcommand{\X}{\mathcal{X}}
\newcommand{\m}{{(m)}}
\newcommand{\peps}{\lbrack\Pro_\veps\rbrack}
\newcommand{\pepsh}{\lbrack\Pro_\veps^{h_\T}\rbrack}    
\newcommand{\pepso}{\lbrack\Pro_0\rbrack}
\newcommand{\pepsho}{\lbrack\Pro^{h_\T}_0\rbrack}

\newcommand{\T}{\mathcal{T}}
\newcommand{\nuk}{\uu{\nu}_{\sigma,K}}
\newcommand{\Lt}{\mathcal{L}_{\mathcal{T}}}
\newcommand{\He}{H_{\E}}
\newcommand{\Heo}{H_{\E,0}}
\newcommand{\flx}{F_{\sigma,K}}
\newcommand{\absk}{\abs{K}}
\newcommand{\abssig}{\abs{\sigma}}
\newcommand{\dsig}{D_{\sigma}}
\newcommand{\divup}{\mathrm{div}_{\T}^{up}}
\newcommand{\divt}{\mathrm{div}_{\T}}
\newcommand{\gradt}{\nabla_{\T}}

\newcommand{\gradd}{\nabla_{\E}}
\newcommand{\rk}{\vrho_K}
\newcommand{\uk}{\bu_K}

\newcommand{\sk}{\sigma,K}

\newcommand{\wsk}{w_{\sk}}
\newcommand{\sink}{\sigma\in\E(K)}
\newcommand{\delt}{\delta t}

\newcommand{\norm}[1]{\left\lVert#1\right\rVert}
\newcommand{\abs}[1]{\left\lvert#1\right\rvert}
\newcommand{\weakstar}{\overset{\ast}{\rightharpoonup}}

\newcommand{\Div}{\mathrm{div}_x}
\newcommand{\grad}{\nabla_x}
\newcommand{\Dt}{\partial_t}
\DeclareMathOperator*{\essup}{ess\,sup}

\title[AP Scheme for Compressible Euler Equations]{ASYMPTOTIC PRESERVING FINITE VOLUME METHOD FOR THE COMPRESSIBLE EULER EQUATIONS: ANALYSIS VIA DISSIPATIVE MEASURE-VALUED SOLUTIONS} 

\author[Arun]{K.R.\ Arun}
\address{School of Mathematics, Indian Institute of Science Education
  and Research Thiruvananthapuram, Thiruvananthapuram 695551, India} 
\email{arun@iisertvm.ac.in}
\thanks{K.\,R.\,A. acknowledges the support from Science and Engineering Research Board, Department of Science \& Technology, Government of India through grant CRG/2021/004078.}

\author[Krishnamurthy]{A.\ Krishnamurthy}
\address{School of Mathematics, Indian Institute of Science Education
  and Research Thiruvananthapuram, Thiruvananthapuram 695551, India} 
\email{amoghk0720@iisertvm.ac.in}

\author[Luk\'a\v{c}ov\'a-Medvi\softd ov\'a]{M.\ Luk\'a\v{c}ov\'a-Medvi\softd ov\'a}
\address{Institut f{\"u}r Mathematik, Johannes Gutenberg-Universit{\"a}t Mainz, Staudingerweg 9, D-55128 Mainz, Germany}
\email{lukacova@mathematik.uni-mainz.de}
\thanks{The work of M.L. was funded by the Deutsche Forschungsgemeinschaft (DFG, German Research
Foundation) - under SPP 2410 Hyperbolic Balance Laws: Complexity, Scales and Randomness.  M.L. is grateful to the Gutenberg Research College and Mainz Institute of Multiscale Modelling for
supporting her research.}

\date{\today}



\begin{document}

\begin{abstract}
    We propose and analyze a new asymptotic preserving (AP) finite volume scheme for the multidimensional compressible barotropic Euler equations to simulate low Mach number flows. The proposed scheme uses a stabilized upwind numerical flux, with the stabilization term being proportional to the stiff pressure gradient, and we prove its conditional energy stability and consistency. Utilizing the concept of dissipative measure-valued (DMV) solutions, we rigorously illustrate the AP properties of the scheme for well-prepared initial data. In particular, we prove that the numerical solutions will converge weakly to a DMV solution of the compressible Euler equations as the mesh parameter vanishes, while the Mach number is fixed. The DMV solutions then converge to a classical solution of the incompressible Euler system as the Mach number goes to zero. Conversely, we show that if the mesh parameter is kept fixed, we obtain an energy stable and consistent finite-volume scheme approximating the incompressible Euler equations as the Mach number goes to zero. The numerical solutions generated by this scheme then converge weakly to a DMV solution of the incompressible Euler system as the mesh parameter vanishes. Invoking the weak-strong uniqueness principle, we conclude that the DMV solution and classical solution of the incompressible Euler system coincide, proving the AP property of the scheme. We also present an extensive numerical case study in order to illustrate the theoretical convergences, wherein we utilize the techniques of $\mathcal{K}$-convergence.
\end{abstract}

\maketitle
\section{Introduction}
\label{sec:intro}

Several laws of continuum physics can be formulated in terms of the conservation of quantities such as mass, momentum or energy. Hyperbolic conservation laws are the mathematical formulation of these conservation principles. An iconic example of this class of partial differential equations, undoubtedly, is the system of compressible Euler equations of gas dynamics. Solutions to the compressible Euler system contain waves that propagate at different speeds, exhibiting the so-called `multiscale' phenomenon. In order to quantify this disparity between the speeds, one appropriately rescales the compressible Euler equations and introduces the Mach number $\veps$, which is defined as the ratio of a reference convection velocity to a reference sound velocity; see \cite{KM82}. It is well documented in literature that as $\veps\to 0$, the solutions to the compressible Euler equations converge to the solutions of the incompressible system and the limit is termed as the zero or low Mach number limit; see \cite{KM82, LM98, Maj84, Sco94} for a detailed exposition of the same. The zero Mach number limit is classified as a singular limit, as the hyperbolic compressible Euler system changes into the mixed hyperbolic-elliptic incompressible Euler system. In the aforementioned references, the zero Mach number limit is performed for smooth or strong solutions of the compressible system. On the other hand, it is well known that the solutions to the compressible Euler equations develop singularities in finite time even if the initial data are smooth, necessitating the notions of weak or distributional solutions. Using the techniques of convex integration, the non-uniqueness of entropy weak solutions to the multidimensional compressible Euler system was shown by Chiodaroli et al.\ \cite{CLK15}, DeLellis and Sz\'ekelyhidi \cite{LS10}. 

In the present work, we adopt the framework of dissipative measure-valued (DMV) solutions, which was first introduced by Feireisl et al.\ for the Navier-Stokes system in \cite{FGS+16} and for the Euler system in \cite{BF18a, FGJ19}. The concept of measure-valued solutions for general hyperbolic conservation laws was advocated by DiPerna in \cite{DiP85}. Here, the Young measures are used to describe the weak limits of osciallating and concentrating sequences. Contrary to weak solutions, the DMV solutions to the compressible Euler system exist globally in time, as long as the initial data has finite energy. Furthermore, they can be realized as suitable limits of numerical schemes as well as other approximation problems; see the monographs \cite{FLM+21a, MNR+96}. A key advantage of working in this framework is the validity of the so-called weak-strong uniqueness principle, which asserts that a strong/classical solution and a DMV solution coincide as long as they emanate from the same initial data; cf.\ \cite{BF18a, FGJ19, FGS+16, FLM+21a, GSW15, Wei18} for further details. Similar results have also been proven for the DMV solutions of the incompressible Euler system; cf.\ \cite{BLS11} and the survey papers \cite{GWW23, Wei18}. 

The study of singular limits using DMV solutions has been a subject of numerous works; see \cite{BF18b, BM21, Cha22, FKK+20, FKM19, FT19, NT20} and the references therein. Notably, Ferieisl et al.\ in \cite{FKM19} showed that DMV solutions to the compressible Euler system will converge to a classical solution of the incompressible Euler system as the Mach number $\veps\to 0$, provided the latter exists, the initial datum is well-prepared and the domain is bounded with periodic boundaries. The existence of a smooth solution to the incompressible system is guaranteed by the results of Kato \cite{Kat72, KL84}, provided the initial data has appropriate regularity.

As stated earlier, DMV solutions to the compressible Euler system can be identified as suitable limits of numerical solutions. As a matter of fact, any structure preserving, entropy stable and consistent numerical scheme approximating the compressible Euler system will generate a DMV solution; see \cite{FLM+21a}. In recent years, there has been a considerable amount of literature dedicated to proving the convergence of numerical schemes to DMV solutions of the target system, cf.\ \cite{AA24, FLM20b, FLM+19, FLM+21a, FLM+21b, FLM20a} and the references therein. The overall strategy across all the above mentioned references can be explained as follows. First step is to show that the numerical solutions generated by the scheme are uniformly bounded in suitable spaces, preserve the positivity of density and internal energy and satisfy a discrete variant of the entropy inequality. For the second step, one needs to prove that the numerical solution satisfies the weak formulation of the target system modulo some consistency errors that vanish as the mesh parameter vanishes. Refining the mesh and using the apriori bounds then leads to the weak convergence of the numerical solutions to a DMV solution of the target system. Since the convergence achieved is only weak , representation of the limiting solution is a subtle issue. This is were the novel concept of $\K$-convergence applies, first introduced by Balder in \cite{Bal00} and further expanded upon by Feireisl et al.\ in \cite{FLM+21a, FLM+21b, FLM20a}. Briefly, $\K$-convergence turns weakly convergent sequences of individual numerical solutions to strongly convergent sequences of empirical averages (C\'{e}saro averages). This can be seen as a generalization of Koml\'{o}s' theorem to Young measures; cf.\ \cite{Kom67}.

On the other hand, performing the zero Mach number limit at the discrete level is a challenging task due to the singular nature of the limit, in the sense that the system of equations changes its nature. Due to the stiff character of the problem, classical time-stepping methods are proven to be ineffective to compute the numerical solutions. The CFL conditions associated to explicit methods impose extremely stringent conditions on the time-step to ensure stability, which mainly arises due to the large disparity between the flow and sound velocities; cf.\ \cite{ADS21, AS20, GV99, Kle95}. Also, even though fully implicit schemes may offer unconditional stability, said stability is obtained at a significant increase to the computational costs which is rather undesirable. Hence, one prefers to work with the semi-implicit or the implicit-explicit (IMEX) methods and one of the early attempts in this direction is available in \cite{PM05}. However, the time-stepping alone is not sufficient and one needs a much more robust platform in order to accurately capture the flow at low Mach numbers. 

One such framework is the so-called `Asymptotic Preserving' (AP) methodology, which was first introduced in the context of kinetic models of diffusive transport \cite{Jin99}. The overall principle is as follows. Consider a singularly perturbed problem $\peps$, where $\veps>0$ denotes the perturbation parameter, e.g.\ the Mach number in the present case. Assume that the solutions of $\peps$ converge to the solutions of a well posed problem $\pepso$ as $\veps\to 0$, where the limit problem $\pepso$ is termed as the singular limit. Let $\pepsh$ denote a numerical scheme for $\peps$, where $h_\T>0$ denotes the discretization parameter, e.g.\ the mesh size. Then, the scheme $\pepsh$ is said to be AP if the following hold.
\begin{enumerate} 
    \item For every $h_\T>0$, the scheme $\pepsh$ converges to a numerical scheme $\pepsho$ as $\veps\to 0$, where $\pepsho$ is a consistent and stable approximation of the limit system $\pepso$. 
    \item The stability constraints on $h_\T$ are independent of $\veps$.
\end{enumerate}
The idea of an AP scheme can also be illustrated by the commutativity of the diagram in Figure \ref{fig:AP-prop}. 
\begin{figure}[htpb]
    \centering
        \begin{tikzcd}
            \pepsh \arrow[dd, "\varepsilon\to 0"] \arrow[rr, "h_\T\to 0"] &  & \peps \arrow[dd, "\varepsilon\to 0"] \\
            &  &                                                                      \\
            \pepsho \arrow[rr, "h_\T\to 0"] &  & \pepso                        
        \end{tikzcd}
        \caption{AP Property of a scheme.}
        \label{fig:AP-prop}
\end{figure}

Therefore, if a scheme is AP, then the low Mach number limit will be respected at the discrete level as well. Hence, the AP methodology is a natural choice for designing numerical schemes to capture the low Mach number flows. IMEX or semi-implicit time discretizations are most commonly used in order to design AP schemes; see \cite{AGK23, AKS22, AS20, BAL+14, BRS18, HJL12, HLS21, NBA+14} for results related to the Euler system.

The goal of the present work is to analyze the low Mach number limit of the cell-centered, semi-implicit finite volume scheme introduced in \cite{AA24} to approximate the compressible Euler equations, cf.\ \eqref{eqn:disc-mss-bal}-\eqref{eqn:disc-mom-bal}, via DMV solutions. In particular, we rigorously show that the proposed scheme is indeed AP and satisfies the commutative diagram given in Figure \ref{fig:AP-prop}. The energy stability of the scheme was achieved in \cite{AA24} by introducing a shifted velocity in the convective fluxes of the mass and momenta, provided a CFL-like condition holds. This shifted velocity was shown to be proportional to the stiff pressure gradient and further, a consistency analysis was carried out in order to show that the solutions generated by the scheme satisfy the weak formulation of the compressible system modulo the consistency error terms that vanish as the mesh size $h_\T$ goes to 0. It was also proven that as $h_\T\to 0$, the solutions generated by the scheme will converge to a DMV solution of the compressible Euler equations. In the present work, we extend these results and perform the low Mach number limit for the solutions generated by the scheme. Firstly, note that for a fixed $\veps>0$, the scheme will generate a DMV solution of the compressible system as $h_\T\to 0$. Then, provided that the initial data are well prepared, we can conclude that as $\veps\to 0$, the DMV solutions will converge to a classical solution of the incompressible system in accordance with \cite{FKM19}. Secondly, we prove that on a fixed mesh with size $h_\T>0$, as $\veps\to 0$, we obtain a scheme which is an approximation of the incompressible Euler equations, hereafter referred to as the limit scheme. Further, we rigorously show that the limit scheme is indeed energy stable and a consistent approximation of the incompressible system which generates a DMV solution as $h_\T\to 0$. Finally, using the weak-strong uniqueness principle, cf.\ \cite{GWW23, Wei18}, we conclude that the DMV solution coincides with the classical solution of the incompressible system, which in turn proves the AP property of the scheme. We also present an extensive numerical case study in order to illustrate the theoretical convergence results.

The rest of this manuscript is organized as follows. In Section \ref{sec:comp-euler}, we introduce the non-dimensional form of the compressible Euler system parameterized by the Mach number $\veps$ and recall a few energy identities satisfied by classical solutions of the same. Further, we also define the notions of a weak solution and a DMV solution of the compressible system and recall the weak-strong uniqueness principle. In Section \ref{sec:incomp-euler}, we present the incompressible Euler system and analogous to Section \ref{sec:comp-euler}, recall the definitions of a weak solution, a DMV solution and the statement of the weak-strong uniqueness principle. Next, we present a brief overview of the low Mach number limit for the compressible system in Section \ref{sec:incomp-cont-limit} after introducing the notion of `well-prepared' initial data. The semi-implicit scheme along with its stability and consistency analysis is detailed in Section \ref{sec:fv-scheme}. In Section \ref{sec:lim-analys-scheme}, we perform the analysis of the limits of the scheme while taking the the limits $h_\T\to 0$ and $\veps\to 0$ in an iterative manner. We consider both the possible cases, first $h_\T\to 0$ and successively $\veps\to 0$, and vice-versa, and thoroughly describe the evolution of the solutions generated by the scheme as we perform said limits. Finally, in Section \ref{sec:AP-prop-scheme}, we deduce the AP property of the scheme and present the main result of this paper. Convergence results using $\K$-convergence are recalled in Section \ref{sec:K-con}. In order to illustrate the claimed convergences, a detailed numerical case study is presented in Section \ref{sec:num-exp}. Lastly, the paper is concluded with a few remarks in Section \ref{sec:conc}.

\section{Compressible Euler System}
\label{sec:comp-euler}
We consider the following initial boundary value problem for the compressible barotropic Euler system, which is parameterized by the Mach number $\veps$, in $\odom$ where $\Omega\subset\R^d$ is bounded with $d=2$ or 3:
\begin{subequations}
\begin{align}
    &\Dt\vrho_\veps+\Div\bigl(\vrho_\veps\bu_\veps\bigr)\,=\,0, \label{eqn:mss-bal-baro} \\
    &\Dt\bigl(\vrho_\veps\bu_\veps\bigr)+\Div\bigl(\vrho_\veps\bu_\veps\otimes\bu_\veps\bigr)+\frac{1}{\veps^2}\grad p_\veps\,=\,0, \label{eqn:mom-bal-baro} \\
    &\vrho_\veps(0,\cdot)=\vrho_{0,\veps}\,,\,\bu_\veps(0,\cdot)=\bu_{0,\veps}\,,\, \bu_{\veps} = 0\text{ on }(0,T)\times\partial\Omega.\label{eqn:ini-data-baro}
\end{align}
\end{subequations}
The variables $\vrho_\veps$ and $\bu_\veps$ denote the density and the velocity of the fluid respectively. The pressure $p_\veps = p(\vrho_\veps)$ is assumed to follow the isentropic equation of state $p(\vrho) = \vrho^\gamma$ where $\gamma> 1$ is the ratio of the specific heats, also known as the adiabatic constant.

For the sake of simplicity, we impose space-periodic boundary conditions. Hence, we identify $\Omega$ with the $d$-dimensional torus $\mathbb{T}^d$ throughout this paper and the system \eqref{eqn:mss-bal-baro}-\eqref{eqn:ini-data-baro} will be denoted by $\peps$.

\subsection{Apriori Energy Estimates}
\label{subsec:apriori-est}

We begin by recalling a few apriori energy estimates satisfied by classical solutions of \eqref{eqn:mss-bal-baro}-\eqref{eqn:mom-bal-baro}. To this end, we start by defining the so-called pressure potential or internal energy per unit volume 
\begin{equation}
\label{eqn:pres-pot}
    \psig(\vrho) \coloneqq \frac{\vrho^\gamma}{\gamma - 1}, \gamma > 1.
\end{equation}

\begin{proposition}
\label{prop:apriori-enrg-est}

Classical solutions of \eqref{eqn:mss-bal-baro}-\eqref{eqn:mom-bal-baro} satisfy 

\begin{enumerate}
    \item a renormalization identity: 
    \begin{equation}
    \label{eqn:renorm-idt}
        \Dt\psig(\vrho_\veps) + \Div(\psig(\vrho_\veps)\bu_\veps) + p_\veps\Div(\bu_\veps) = 0. 
    \end{equation}

    \item a positive renormalization identity: 
    \begin{equation}
    \label{eqn:pos-renorm-idt}
        \Dt\pig(\vrho_\veps) + \Div((\psig(\vrho_\veps) - \psig^\prime(1)\vrho_\veps)\bu_\veps) + p_\veps\Div(\bu_\veps) = 0, 
    \end{equation}
    where $\pig(\vrho) \coloneqq \psig(\vrho) - \psig(1) - \psig^\prime(1)(\vrho - 1)$ is the relative internal energy, which is an affine approximation of $\psig$ with respect to the constant state $\vrho = 1$.

    \item a kinetic energy identity:
    \begin{equation}
    \label{eqn:ke-idt}
        \Dt\biggl(\half\vrho_\veps\abs{\bu_\veps}^2\biggr)+\Div\biggl(\half\vrho_\veps\abs{\bu_\veps}^2\bu_\veps\biggr)+\frac{1}{\veps^2}\grad p_\veps\cdot\bu_\veps = 0.
    \end{equation}

    \item the total energy identity:
    \begin{equation}
    \label{eqn:tot-enrg-idt}
        \Dt\biggl(\half\vrho_\veps\abs{\bu_\veps}^2 + \frac{1}{\veps^2}\pig(\vrho_\veps)\biggr) + \Div\biggl(\biggl(\half\vrho_\veps\abs{\bu_\veps}^2 
        +\frac{1}{\veps^2}\biggl(\psig(\vrho_\veps) - \psig^\prime(1)\vrho_\veps + p_\veps\biggr)\biggr)\bu_\veps\biggr) = 0.
    \end{equation}
\end{enumerate}
\end{proposition}

\subsection{Admissible Weak Solutions}
\label{subsec:baro-wk-form}

Let $\bm_\veps = \vrho_\veps\bu_\veps$ denote the momentum. Then, a pair $(\vrho_\veps,\bm_\veps)$, $\vrho_\veps\in C_{weak}(\lbrack 0,T\rbrack; L^\gamma(\Omega))$ with $\vrho\geq 0$ a.e.\ in $\odom$ and $\bm_\veps\in C_{weak}(\lbrack 0,T\rbrack;L^{\frac{2\gamma}{\gamma+1}}(\Omega;\R^d))$, is an admissible weak solution of the system $\peps$ if the following identities hold:
\begin{equation}
\label{eqn:mss-wk-baro}
\biggl\lbrack\int_\Omega\vrho_\veps\varphi\,\dx\biggr\rbrack_{t=0}^{t=\tau}\,=\,\int_0^\tau\int_\Omega\Bigl\lbrack\vrho_\veps\Dt\varphi+\bm_\veps\cdot\grad\varphi\Bigr\rbrack\,\dx\,\dt,
\end{equation}
for any $\tau\in\lbrack 0,T\rbrack$ and $\varphi\in\Cinf\bigl(\lbrack 0,T)\times\Omega\bigr)$;
\begin{equation}
\label{eqn:mom-wk-baro}
\biggl\lbrack\int_\Omega\bm_\veps\cdot\uu{\varphi}\,\dx\biggr\rbrack_{t=0}^{t=\tau}\,=\,\int_0^\tau\int_\Omega\biggl\lbrack\bm_\veps\cdot\Dt\uu{\varphi}+\biggl(\frac{\bm_\veps\otimes\bm_\veps}{\vrho_\veps}\biggr)\colon\grad\uu{\varphi}+\frac{1}{\veps^2}p_\veps\,\Div\uu{\varphi}\biggr\rbrack\,\dx\,\dt,
\end{equation}
for any $\tau\in\lbrack 0,T\rbrack$ and $\uphi\in\Cinf\bigl(\lbrack 0,T)\times\Omega;\R^d\bigr)$;
\begin{equation}
\label{eqn:baro-entropy}
\int_\Omega\biggl\lbrack\half\frac{\abs{\bm_\veps}^2}{\vrho_\veps} + \frac{1}{\veps^2}\psig(\vrho_\veps)\biggr\rbrack(t,\cdot)\,\dx\leq \int_\Omega\biggl\lbrack\half\frac{\abs{\bm_\epso}^2}{\vrho_\epso} + \frac{1}{\veps^2}\psig(\vrho_\epso)\biggr\rbrack\,\dx,
\end{equation}
for a.e. $t\in (0,T)$.

\subsection{Dissipative Measure-Valued Solutions}
\label{subsec:dmv-soln-baro}

The uniqueness of entropy weak solutions fails to hold for the multidimensional barotropic Euler equations, seemingly due to the lack of `compactness' of entropy weak solutions, in the sense that a bounded sequence of solutions may develop oscillations and/or concentrations. Though measure-valued solutions were first introduced by DiPerna in \cite{DiP85} in 1985, the framework of DMV solutions is rather recent and was first established by Feireisl et al.\ in \cite{FGS+16} for the Navier-Stokes system and was further extended to the Euler system in \cite{BF18a, FGJ19}. A key reason why one might prefer working with DMV solutions rather than entropy weak solutions is the fact that DMV solutions are obtained as limits of weakly convergent approximations. We now recall the definition of a DMV solution to the compressible barotropic Euler system following \cite{FLM+21a}. We begin by defining the phase space of the system which reads
\begin{equation}
\label{eqn:phs-spc-baro}
    \F_{comp}=\bigl\{\lbrack \tvrho,\tbm\,\rbrack\colon\tvrho \geq 0,\tbm\in\R^d\bigr\} \subset{\R^{d+1}},
\end{equation}
and we let $\Pro(\F_{comp})$ denote the space of all proability measures on $\F_{comp}$.

\begin{definition}[Dissipative Measure-Valued Solution]
\label{def:dmv-baro}

A parameterized family of probability measures $\mcv^\veps=\lbrace\V^\veps\rbrace_{(t,x)\in\odom}\in L_{weak-*}^\infty\bigl(\odom,\Pro(\F_{comp})\bigr)$ is called a DMV solution of $\peps$ with initial data $\lbrack \vrho_{0,\veps},\bm_{0,\veps}\rbrack$ if the following hold:

\begin{itemize}
\item \textbf{Energy inequality} - the integral inequality 
\begin{equation}
\label{eqn:enrg-ineq-baro}
\int_\Omega\biggl\langle \V^\veps; \half\frac{\abs{\tbm}^2}{\tvrho}+\frac{1}{\veps^2}\psig(\tvrho)\biggr\rangle\,\dx\,+\,\int_\Omega\,\mathrm{d}\mathfrak{C}^\veps_{cd}(t)\leq\int_\Omega\biggl\lbrack\half\frac{\abs{\bm_{0,\veps}}^2}{\vrho_{0,\veps}}+\frac{1}{\veps^2}\psig(\vrho_{0,\veps})\biggr\rbrack\,\dx
\end{equation}
holds for a.e.\ $t\in\lbrack 0,T\rbrack$ with the so-called energy concentration defect measure
\begin{equation}
\label{eqn:enrg-con-def}
    \mathfrak{C}^\veps_{cd}\in L^\infty\bigl(0,T;\M^+(\overline{\Omega})\bigr);
\end{equation}

\item \textbf{Equation of continuity} - the integral identity 
\begin{equation}
\label{eqn:dmv-mss-bal-baro}
\biggl\lbrack\int_\Omega\bigl\langle\V^{\veps};\tvrho\bigr\rangle\varphi\,\dx\biggr\rbrack_{t=0}^{t=\tau}\,=\,\int_0^\tau\int_\Omega\biggl\lbrack\bigl\langle\V^\veps;\tvrho\bigr\rangle\Dt\varphi+\bigl\langle\V^\veps;\tbm\bigr\rangle\cdot\grad\varphi\biggr\rbrack\,\dx\,\dt
\end{equation}
holds for any $\tau\in\lbrack 0,T\rbrack$ and $\varphi\in\Cinf(\lbrack 0,T)\times\Omega)$;

\item \textbf{Momentum equation} - the integral identity
\begin{equation}
\label{eqn:dmv-mom-bal-bro}
\begin{split}
    &\biggl\lbrack\int_\Omega\langle\V^\veps;\tbm\rangle\cdot\uu{\varphi}\,\dx\biggr\rbrack_{t=0}^{t=\tau}\,\\
    &=\,\int_0^\tau\int_\Omega\biggl\lbrack\langle\V^\veps;\tbm\rangle\cdot\Dt\uu{\varphi}+\biggl\langle\V^\veps;\frac{\tbm\otimes\tbm}{\tvrho}\biggr\rangle\colon\grad\uu{\varphi}+\frac{1}{\veps^2}\langle\V;p(\tvrho)\rangle\Div\uu{\varphi}\biggr\rbrack\,\dx\,\dt \\
    &+\int_0^\tau\int_\Omega\grad\uu{\varphi}\colon\mathrm{d}\mathfrak{R}^\veps_{cd}(t)\,\dt
\end{split}
\end{equation}
holds for any $\tau\in\lbrack 0,T\rbrack$ and $ \uphi\in\Cinf(\lbrack 0,T)\times\Omega;\R^d)$ with the so-called Reynolds concentration defect measure
\begin{equation}
\label{eqn:reyn-conc-def}
    \mathfrak{R}^\veps_{cd}\in L^\infty\bigl(0,T;\M\bigl(\overline{\Omega};\R^{d\times d}\bigr)\bigr);
\end{equation}

\item \textbf{Defect compatibility condition} - there exists a constant $c>0$ such that 
\begin{equation}
\label{eqn:def-cond}
    \abs{\mathfrak{R}^\veps_{cd}(t)}(\overline{\Omega})\leq c\D^\veps(t)
\end{equation}
holds for a.e.\ $t\in (0,T)$, where $\D^\veps(t) = \int_\Omega\mathrm{d}\mathfrak{C}^\veps_{cd}(t)>0\in L^\infty(0,T)$ is called the dissipation defect. 
\end{itemize}
\end{definition}

\begin{remark}
\label{rem:def-cond}
    The formulation given in \cite{FLM+21a} has, in place of \eqref{eqn:def-cond}, a compatibility condition of the form 
    \[
    \underline{d}\mathfrak{C}^\veps_{cd}\leq\mathrm{tr}(\mathfrak{R}^\veps_{cd})\leq\overline{d}\mathfrak{C}^\veps_{cd},
    \]
    where $0<\underline{d}\leq\overline{d}$. However, we would like to note that as long as we are generating a DMV solution from a `suitable' sequence of solutions, say a sequence of admissible weak solutions or a sequence of numerical solutions generated using some structure-preserving, entropy stable scheme that are consistent with the Euler system, both compatibility conditions are simultaneously satisfied. Further, as a direct consequence of the energy inequality \eqref{eqn:enrg-ineq-baro}, we get 
    \[
    \V\lbrack\lbrace\tvrho>0\rbrace\cup\lbrace\tvrho =0, \tbm = 0\rbrace\rbrack = 1\text{ for a.e.\ }(t,x)\in\odom.
    \]
\end{remark}

\subsection{Weak-Strong Uniqueness Principle}
\label{subsec:wk-strng-uni}
One of the main reasons why the framework of DMV solutions is useful is the validity of the so-called weak-strong uniqueness principle, which asserts that any strong solution and DMV solution of the barotropic Euler system emanating from the same initial data coincide as long as the former exists. We refer the interested reader to \cite{BF18a, FGJ19, FGS+16, FLM+21a, GSW15, Wei18} for more details. The key tool in proving this principle is the relative energy functional, which is defined as
\begin{equation}
\label{eqn:rel-ent}
    E_{rel}(\vrho_\veps, \bm_\veps \,\vert\, r,\bU)(t) = \int_\Omega \biggl\lbrack \half \vrho_\veps\abs{\frac{\bm_\veps}{\vrho_\veps} - \bU}^2 +\frac{1}{\veps^2}\bigl(\psig(\vrho_\veps) - \psig(r) - \psig^{\prime}(r)(\vrho_\veps - r)\bigr)\biggr\rbrack(t,\cdot)\,\dx.
\end{equation}
where $r>0,\bU$ are test functions mimicking a strong solution of the barotropic Euler system.

\begin{remark}
\label{rem:rel-ent}
Though \eqref{eqn:rel-ent} is not a metric as it is not symmetric, the following still hold:
\[
    E_{rel}(\vrho_\veps, \bm_\veps \,\vert\, r,\bU)\geq 0\text{ and }E_{rel}(\vrho_\veps, \bm_\veps \,\vert\, r,\bU) = 0 \iff \vrho_\veps = r\,,\,\bm_\veps=r\bU.
\]
\end{remark}
Analogously, one also defines a measure-valued variant of the relative energy functional, which reads 
\begin{equation}
\label{eqn:rel-ent-mv}
    E_{rel}(\mcv^\veps\,\vert\,r, \bU)(t) = \int_\Omega \biggl\langle\V^\veps; \half\tvrho\abs{\frac{\tbm}{\tvrho} - \bU}^2 + \frac{1}{\veps^2}(\psig(\tvrho) - \psig(r) - \psig^\prime(r)(\tvrho - r)\biggr\rangle\,\dx,
\end{equation}

We now recall the statement of the weak-strong uniqueness principle from \cite{FLM+21a}.
\begin{theorem}[Weak-Strong Uniqueness Principle]
\label{thm:wk-str-uniq}
    Let a parameterized family of probability measures $\mathcal{V^\veps}=\bigl\{\V^\veps\bigr\}_{(t,x)\in\odom}$ be a DMV solution of the barotropic Euler system in the sense of Definition \ref{def:dmv-baro}, with initial data $\lbrack\vrho_{0,\veps},\bm_{0,\veps}\rbrack$ and associated concentration defect measures $\mathfrak{C}^\veps_{cd}$ and $\mathfrak{R}^\veps_{cd}$.

    Let $r_\veps,\bU_\veps$ be a strong solution of the barotropic Euler system belonging to the class 
    \begin{equation}
    \label{eqn:baro-wk-strg-fun-cls}
        \begin{split}
            &r_\veps\in W^{1,\infty}(\odom),\,\inf_{(t,x)\in\odom}r_\veps(t,x)>0, \\
            &\bU_\veps\in W^{1,\infty}(\odom;\R^d),
        \end{split}
    \end{equation}
    such that $r_\veps(0,\cdot)=\vrho_{0,\veps}$ and $r_\veps(0,\cdot)\bU_\veps(0,\cdot)=\bm_{0, \veps}$.

    Then
    \[
        \V^\veps=\delta_{\lbrack r_\veps(t,x),r_\veps\textbf{\textit{U}}_\veps(t,x)\rbrack}\text{ for a.e. }(t,x)\in\odom,
    \]
    and
    \[
        \mathfrak{C}^\veps_{cd}=\mathfrak{R}^\veps_{cd}=0.
    \]
\end{theorem}

In the following lemma, we show that an admissible weak solution (as defined in Subsection \ref{subsec:baro-wk-form}) and a DMV solution (in the sense of Definition \ref{def:dmv-baro}) of the compressible Euler equations satisfy an additional energy estimate, given in terms of the relative internal energy $\pig$. 

\begin{lemma}[Energy Estimates for Weak and DMV Solutions]
\label{lem:eng-est-pig}
Consider the compressible Euler system \eqref{eqn:mss-bal-baro}-\eqref{eqn:mom-bal-baro} with the initial and boundary conditions given by \eqref{eqn:ini-data-baro}.
    \begin{enumerate}
        \item Let $(\vrho_\veps,\bu_\veps)$ be an admissible weak solution of the compressible Euler system as defined in Subsection \ref{subsec:baro-wk-form}. Then, $(\vrho_\veps,\bu_\veps)$ also satisfies the following energy inequality:
            \begin{equation}
            \label{eqn:baro-ent-pig}
                \int_\Omega\biggl\lbrack\half\frac{\abs{\bm_\veps}^2}{\vrho_\veps} + \frac{1}{\veps^2}\pig(\vrho_\veps)\biggr\rbrack(t,\cdot)\,\dx\leq \int_\Omega\biggl\lbrack\half\frac{\abs{\bm_\epso}^2}{\vrho_\epso} + \frac{1}{\veps^2}\pig(\vrho_\epso)\biggr\rbrack\,\dx.
            \end{equation}
        
        \item Let $\mcv^\veps = \lbrace\V^\veps\rbrace_{(t,x)\in\odom}$ be a DMV solution of the compressible Euler system, in the sense of Definition \ref{def:dmv-baro}, with concentration defects $\mathfrak{C}_{cd}$ and $\mathfrak{R}_{cd}$. Then, the following energy estimate holds:
            \begin{equation}
            \label{eqn:dmv-baro-ent-pig}
                \int_\Omega\biggl\langle \V^\veps; \half\frac{\abs{\tbm}^2}{\tvrho}+\frac{1}{\veps^2}\pig(\tvrho)\biggr\rangle\,\dx\,+\,\int_\Omega\,\mathrm{d}\mathfrak{C}^\veps_{cd}(t)\leq\int_\Omega\biggl\lbrack\half\frac{\abs{\bm_{0,\veps}}^2}{\vrho_{0,\veps}}+\frac{1}{\veps^2}\pig(\vrho_{0,\veps})\biggr\rbrack\,\dx.
            \end{equation}
    \end{enumerate}
\end{lemma}

\begin{proof}
    For a strong solution $(r_\veps,\bU_\veps)$ of the compressible Euler equations belonging to the class \eqref{eqn:baro-wk-strg-fun-cls}, one can obtain the following relative energy inequalities, see \cite{FLM+21a} for a detailed exposition of the same. 
    \begin{enumerate}
        \item $E_{rel}(\vrho_\veps,\bu_\veps\,\vert\, r_\veps,\bU_\veps)(t)\leq E_{rel}(\vrho_\veps, \bu_\veps\,\vert\, r_\veps, \bU_\veps)(0)$ where $(\vrho_\veps,\bu_\veps)$ is an admissible weak solution of the compressible Euler system. 
        \item  $E_{rel}(\mcv^\veps\,\vert\,r_\veps, \bU_\veps)(t) + \D^\veps(t) \leq E_{rel}(\mcv^\veps\,\vert\,r_\veps, \bU_\veps)(0)$ where $\mcv^\veps = \lbrace\V^\veps\rbrace_{(t,x)\in\odom}$ is a DMV solution of the compressible Euler system.
    \end{enumerate}
    Then, noting that $(r_\veps, \bU_\veps) = (1,0)$ is a strong solution of the system \eqref{eqn:mss-bal-baro}-\eqref{eqn:ini-data-baro}, one immediately obtains \eqref{eqn:baro-ent-pig} and \eqref{eqn:dmv-baro-ent-pig} by setting $r_\veps = 1$ and $\bU_\veps = 0$ in the above relative energy inequalities respectively.
\end{proof}

\begin{remark}
\label{rem:pig-bound}
    The additional entropy estimate presented in the above lemma is quite important, as it allows us to deduce the necessary apriori bounds on the density via the bounds on the relative internal energy $\pig$, enabling us to prove the convergence of $\vrho_\veps\to 1$ as $\veps\to 0$. 
\end{remark}

\subsection{Velocity Stabilization}
\label{subsec:vel-stab}

In order to enforce energy stability at a numerical level, we adopt the formalism introduced in \cite{CDV17,DVB17,DVB20,GVV13,PV16} wherein a stabilization term in the form of a shifted velocity is introduced in the convective fluxes of the mass and momenta, to yield the following modified system:
\begin{subequations}
\begin{align}
    &\Dt\vrho_\veps+\Div(\vrho_\veps(\bu_\veps-\du_\veps))=0 \label{eqn:mod-mss-bal-baro},\\
    &\Dt(\vrho_\veps\bu_\veps)+\Div(\vrho_\veps\bu_\veps\otimes(\bu_\veps-\du_\veps))+\frac{1}{\veps^2}\grad p_\veps=0. \label{eqn:mod-mom-bal-baro}
\end{align}
\end{subequations}

Analogous to Proposition \ref{prop:apriori-enrg-est}, we can show that the solutions to the modified system satisfy the following a priori energy estimates. 
\begin{proposition}[A priori energy estimates for the modified system]
    \label{prop:mod-enrg-est}
    Classical solutions of \eqref{eqn:mod-mss-bal-baro}-\eqref{eqn:mod-mom-bal-baro} satisfy
    \begin{enumerate}
        \item a renormalization identity:
        \begin{equation}
            \label{eqn:mod-renorm-idt}
            \Dt\psig(\vrho_\veps)+\Div(\psig(\vrho_\veps)(\bu_\veps-\du_\veps))+p_\veps\,\Div(\bu_\veps-\du_\veps)\,=\,0;
        \end{equation}

        \item a positive renormalization identity:
        \begin{equation}
        \label{eqn:mod-pos-renom}
            \Dt\pig(\vrho_\veps) + \Div((\psig(\vrho_\veps) - \psig^\prime(1)\vrho_\veps)(\bu_\veps-\du_\veps)) + p_\veps\Div(\bu_\veps-\du_\veps) = 0
        \end{equation}

        \item a kinetic energy identity: 
        \begin{equation}
        \label{eqn:mod-ke-idt}
            \Dt\biggl(\half\vrho_\veps\abs{\bu_\veps}^2\biggr)+\Div\biggl(\half\vrho_\veps\abs{\bu_\veps}^2(\bu_\veps-\du_\veps)\biggr)+\frac{1}{\veps^2}\grad p_\veps\cdot(\bu_\veps - \du_\veps) = -\frac{1}{\veps^2}\grad p_\veps\cdot\du_\veps .    
        \end{equation}
        
        \item the energy balance:
        \begin{equation}
        \label{eqn:mod-ent-idt}
        \begin{split}
            \Dt\biggl(\half\vrho_\veps\abs{\bu_\veps}^2 + \frac{1}{\veps^2}\pig(\vrho_\veps)\biggr) + \Div\biggl(\biggl(\half\vrho_\veps\abs{\bu_\veps}^2 +&\frac{1}{\veps^2}\biggl(\psig(\vrho_\veps) - \psig^\prime(1)\vrho_\veps + p_\veps\biggr)\biggr)(\bu_\veps-\du_\veps)\biggr) \\
            &= -\frac{1}{\veps^2}\grad p_\veps\cdot\du_\veps.
        \end{split}
        \end{equation}
    \end{enumerate}
\end{proposition}

Therefore, at least at the continuous level, formally setting $\du_\veps\,=\,\eta\grad p_\veps\,,\,\eta>0$, gives us
\begin{equation}
    \label{eqn:mod-ent-ineq}
    \begin{split}
    \Dt\biggl(\half\vrho_\veps\abs{\bu_\veps}^2 + \frac{1}{\veps^2}\pig(\vrho_\veps)\biggr) + \Div\biggl(\biggl(\half\vrho_\veps\abs{\bu_\veps}^2 +&\frac{1}{\veps^2}\biggl(\psig(\vrho_\veps) - \psig^\prime(1)\vrho_\veps + p_\veps\biggr)\biggr)(\bu_\veps-\du_\veps)\biggr)\\ 
    &= -\frac{\eta}{\veps^2}\abs{\grad p_\veps}^2\leq 0.
    \end{split}
\end{equation}
We thus observe that introducing a shift in the velocity indeed has a stabilizing effect, allowing us to obtain the desired energy inequality, cf.\ \eqref{eqn:baro-entropy}, which we replicate at the discrete level.

\section{Incompressible Euler System}
\label{sec:incomp-euler}

The incompressible Euler equations in $\odom$ can be formally viewed as the asymptotic limit of \eqref{eqn:mss-bal-baro}-\eqref{eqn:mom-bal-baro} as $\veps\to 0$; see \cite{KM82, Maj84} for more details. The system is given by 
\begin{subequations}
    \begin{align}
        \Div\bv = 0, \label{eqn:div-free-incomp}\\
        \Dt\bv +\Div(\bv\otimes\bv) + \grad\pi  = 0, \label{eqn:mom-bal-incomp} \\
        \bv(0,\cdot) = \bv_0,\quad \Div\bv_0 = 0. \label{eqn:incomp-ini-data}
    \end{align}
\end{subequations}
Here, $\pi$ denotes the incompressible pressure and it is the formal limit of $(p(\vrho_\veps)-1)/\veps^2$. Analogously, we denote system \eqref{eqn:div-free-incomp}-\eqref{eqn:mom-bal-incomp} by $\pepso$ and with some abuse of notation, $\displaystyle\lim_{\veps\to 0}\peps = \pepso$

We now recall from Kato \cite{Kat72, KL84}, see also the monograph by Majda \cite{Maj84}, the fundamental result which asserts the existence and uniqueness of a classical solution to $\pepso$. 

\begin{theorem}
\label{thm:incomp-soln-exis}
    Let $\bv_0\in H^k(\Omega;\R^d)$, $k>d/2+1$, with $\Div\bv_0 = 0$. Then, there exists a $T_{max}>0$ such that \eqref{eqn:div-free-incomp}-\eqref{eqn:mom-bal-incomp} admits a classical solution $\bv\in C^1(\lbrack 0, T_{max}\rbrack\times\Omega)$, which is unique in the class 
    \[
    \bv\in C(\lbrack 0, T_{max}\rbrack; H^k(\Omega;\R^d))\cap C^1(\lbrack 0, T_{max}\rbrack; H^{k-1}(\Omega;\R^d)).
    \]
    The incompressible pressure $\pi$ can be recovered from the elliptic problem 
    \[
    -\Delta_x\pi = \Div(\Div(\bv\otimes\bv)).
    \]
    Furthermore, if $d=2$, then the existence is global, i.e.\ $T_{max} = \infty$.
\end{theorem}

\subsection{Energy Admissible Weak Solutions}
\label{subsec:incomp-wk-soln}

A function $\bv\in L^2_{loc}(\odom;\R^d)$ with $\bv(0,\cdot) = \bv_0$ is a weak solution of $\pepso$ if 
\begin{equation}
\label{eqn:div-free-wk}
    \int_\Omega \bv(t,\cdot)\cdot\grad\varphi\,\dx = 0,
\end{equation}
for a.e.\ $t\in (0,T)$ and $\varphi\in \Cinf(\Omega)$;
\begin{equation}
\label{eqn:mom-bal-incomp-wk}
\int_0^T\int_\Omega\lbrack \bv\cdot\Dt\uphi + (\bv\otimes\bv)\colon\grad\uphi\rbrack\,\dx\,\dt, + \int_\Omega \bv_0\cdot\uphi(0,\cdot)\,\dx = 0,
\end{equation}
for any $\uphi\in\Cinf(\lbrack 0,T)\times\Omega;\R^d)$ with $\Div\uphi = 0$.

Further, the weak solution $\bv$ is termed energy admissible if 
\begin{equation}
\label{eqn:ent-cond-incomp}
   \half\int_\Omega\abs{\bv(t,\cdot)}^2\,\dx \leq \half\int_\Omega\abs{\bv_0}^2\,\dx,
\end{equation}
for a.e.\ $t\in (0,T)$.

\subsection{Dissipative Measure-Valued Solutions}
\label{subsec:dmv-incomp}

In the case of the incompressible system, the phase space $\F_{incomp} = \R^d$ and analogous to \eqref{eqn:phs-spc-baro}, we denote by $\tbv$ the generic identity element of $\F_{incomp}$. We follow \cite{GWW23} in defining the notion of a DMV solution to the incompressible system.

\begin{definition}[Dissipative measure-valued solution]
\label{def:dmv-incomp}
    A family of probability measures $\mcu = \lbrace\U\rbrace_{(t,x)\in\odom}\in L^\infty_{weak-*}(\odom;\Pro(\F_{incomp}))$ is a DMV solution of $\pepso$ with initial data $\bv_0$ if the following hold:
    \begin{itemize}
        \item \textbf{Energy inequality} - the integral inequality 
        \begin{equation}
        \label{eqn:enrg-ineq-incomp}
            \half\int_{\Omega}\langle\U; \abs{\tbv}^2\rangle\,\dx + \int_{\Omega}\mathrm{d}\mathfrak{D}_{cd}(t) \leq \half\int_{\Omega}\abs{\bv_0}^2\,\dx 
        \end{equation}
        holds for a.e.\ $t\in (0,T)$ with $\mathfrak{D}_{cd}\in L^\infty(0,T;\M^+(\overline{\Omega}))$;

        \item \textbf{Divergence-free condition} - the integral identity 
        \begin{equation}
        \label{eqn:div-free-dmv}
            \int_\Omega\langle\U;\tbv\rangle\cdot\grad\varphi\,\dx = 0
        \end{equation}
        holds for a.e.\ $t\in (0,T)$ and $\varphi\in\Cinf(\Omega)$;

        \item \textbf{Momentum equation} - the integral identity
        \begin{equation}
        \label{eqn:incomp-mom-dmv}
            \int_0^T\int_\Omega\langle\lbrack\U; \tbv\rangle\cdot\Dt\uphi + \langle\U; \tbv\otimes\tbv\rangle\colon\grad\uphi\rbrack\,\dx\,\dt + \int_0^T\int_\Omega\grad\uphi\colon\mathrm{d}\mathfrak{M}_{cd}(t)\,\dt + \int_\Omega\bv_0\cdot\uphi(0,\cdot)\,\dx = 0
        \end{equation}
        holds for all $\uphi\in\Cinf(\lbrack 0,T)\times\Omega;\R^d)$ with $\Div\uphi = 0$, with $\mathfrak{M}_{cd}\in L^\infty\bigl(0,T;\M\bigl(\overline{\Omega};\R^{d\times d}\bigr)\bigr)$;

        \item \textbf{Defect compatibility condition} - there exists $c>0$ such that 
        \begin{equation}
        \label{eqn:incomp-dmv-defect}
            \abs{\mathfrak{M}_{cd}(t)}(\overline{\Omega})\leq c\D(t)
        \end{equation}
        holds for a.e.\ $t\in(0,T)$, where $\D(t) = \int_\Omega \mathrm{d}\mathfrak{D}_{cd}(t)>0\in L^\infty(0,T)$ is the dissipation defect.
    \end{itemize}
\end{definition}

\begin{remark}
\label{rem:incomp-cd-form}
    There is a slight difference between the formulation presented in Definition \ref{def:dmv-incomp} and the one available in \cite{GWW23}. However, both are consistent with each other due to the fact that any DMV solution to the incompressible Euler system is generated by a sequence of energy admissible weak solutions $\lbrace \bv_n\rbrace_{n\in\N}$ of the incompressible system; see \cite[Theorem 1]{SW12}. We can then write
        \begin{align*}
            &\mathfrak{D}_{cd}(t) = \overline{\half\abs{\bv}^2}(t,\cdot) - \half\langle\U;\abs{\tbv}^2\rangle, \\
            &\mathfrak{M}_{cd}(t) = \overline{\bv\otimes\bv}(t,\cdot) - \langle\U; \tbv\otimes\tbv\rangle.
        \end{align*}
    where $\overline{\half\abs{\bv}^2}$ is the weak-star limit of $\bigl\lbrace\half\abs{\bv_n}^2\bigr\rbrace_{n\in\N}$ in $L^\infty(0,T;\M^{+}(\overline{\Omega}))$ and $\overline{\bv\otimes\bv}$ is the weak-star limit of $\lbrace\bv_n\otimes\bv_n\rbrace_{n\in\N}$ in $L^\infty\bigl(0,T;\M\bigl(\overline{\Omega};\R^{d\times d}\bigr)\bigr)$. Further, we can also observe that 
    \[
    \overline{\half\abs{\bv}^2}(t,\cdot) = \half\mathrm{tr}\bigl(\overline{\bv\otimes\bv}(t,\cdot)\bigr);\, \langle\U;\abs{\tbv}^2\rangle = \mathrm{tr}\bigl(\langle\U;\tbv\otimes\tbv\rangle\bigr)
    \]
    and as such, we obtain $\mathfrak{D}_{cd}(t) = \half\mathrm{tr}\bigl(\mathfrak{M}_{cd}(t)\bigr)$. One can also represent these defect measures in terms of the recession functions associated to $\abs{\bv}^2$ and $\bv\otimes\bv$, see \cite{AB97, SW12} for further details. 
\end{remark}

\begin{remark}
\label{rem:comp-cd-form}
    In the case of the compressible Euler system, there are DMV solutions which are not generated by a sequence of admissible weak solutions; see \cite{CFK+17, GW21}. However, any sequence of energy stable and consistent approximations to the compressible Euler system will always generate a compressible DMV solution. Examples of structure-preserving, energy stable schemes that are consistent with the compressible Euler system are presented in \cite{FLM+21a}. In this case, one obtains a similar representation for the defect measures as presented in Remark \ref{rem:incomp-cd-form}.
\end{remark}

\subsection{Weak-Strong Uniqueness}
\label{subsec:wk-strong-incomp}

Analogous to Theorem \ref{thm:wk-str-uniq}, one has a similar result for the incompressible system, which is proven using the tool of relative energy. In case of the incompressible system, the relative energy written in measure-valued form reads
\[
E_{rel}(\mcu\,\vert\,\bU)(t) = \half\int_\Omega\langle\U;\abs{\tbv - \bU}^2\rangle\,\dx.
\]
Here, $\bU$ is a test function that will be later replaced by a classical solution of the incompressible system.

We now state the weak-strong uniqueness principle for the incompressible system; see \cite{GWW23, Wei18} for details. 

\begin{theorem}[Weak-Strong Uniqueness Principle]
\label{thm:wk-str-incomp}
    Let $\mcu = \lbrace\U\rbrace_{(t,x)\in\odom}$ be a DMV solution of the incompressible Euler system with concentration defect measures $\mathfrak{D}_{cd}$ and $\mathfrak{M}_{cd}$ and initial data $\bv_0$ in the sense of Definition \ref{def:dmv-incomp}. Let $\bU\in C^1(\dom;\R^d)$ be a classical solution of \eqref{eqn:div-free-incomp}-\eqref{eqn:mom-bal-incomp} such that $\bU(0,\cdot) = \bv_0$. 
    
    Then 
    \[
    \U = \delta_{\textbf{U}(t,x)}
    \]
    for a.e. $(t,x)\in\odom$ and 
    \[
    \mathfrak{D}_{cd} = \mathfrak{M}_{cd} = 0.
    \]
\end{theorem}

\section{Incompressible Limit in the Continuous Case}
\label{sec:incomp-cont-limit}
In this section, we give an overview of the limit $\veps\to 0$ of $\peps$. We begin by defining the notion of `well-prepared' initial data.
\begin{definition}[Well-prepared Initial Data] 
\label{def:well-prep}
    The family of initial data $\lbrace(\vrho_{\epso}, \bu_{\epso})\rbrace_{\veps>0}$ of the compressible Euler equations is said to be well prepared if $\vrho_\epso > 0$ a.e.\ in $\Omega$ for every $\veps>0$ and if the following hold:
    \begin{equation}
    \label{eqn:well-prep}
        \begin{split}
            &\frac{1}{\veps^2}\norm{\vrho_\epso - 1}_{L^\infty(\Omega)}\leq C, \\
            &\bu_{\epso}\to \bv_0\text{ in }L^2(\Omega;\R^d)\text{ as }\veps\to 0\,,\,\Div\bv_0 = 0,
        \end{split}
    \end{equation}
    where $C>0$ is a constant independent of $\veps$.
\end{definition}

\subsection{Asymptotic Limit of Weak Solutions}
\label{subsec:asymp-lim-wk}

In what follows, we assume the existence of a family of admissible weak solutions $\lbrace(\vrho_\veps, \bu_\veps)\rbrace_{\veps > 0}$ to the compressible Euler equations, as defined in Section \ref{subsec:baro-wk-form}, emanating from the family of well-prepared initial data $\lbrace(\vrho_\epso,\bu_\epso)\rbrace_{\veps>0}$. As we assume that the weak solutions are admissible, they satisfy
\[
\int_\Omega\biggl\lbrack\half\vrho_\veps\lvert\bu_\veps\rvert^2 + \frac{1}{\veps^2}\pig(\vrho_\veps)\biggr\rbrack(t,\cdot)\,\dx \leq \int_\Omega\biggl\lbrack\half\vrho_\epso\lvert\bu_\epso\rvert^2 + \frac{1}{\veps^2}\pig(\vrho_\epso)\biggr\rbrack\,\dx,
\]
cf.\ Lemma \ref{lem:eng-est-pig}. The initial data being well-prepared, along with some estimates on $\pig$ (cf.\ \cite{Lio96}), will give a uniform bound with respect to $\veps$ for the term on the right hand side of the above inequality. As a consequence of the above energy estimate, one can show that $\vrho_\veps\to 1$ in $L^\infty(0,T;L^r(\Omega))$ for any $r\in\lbrack 1, \min\lbrace 2,\gamma\rbrace\rbrack$. Further, one also obtains the existence of $\bv\in L^\infty(0,T;L^2(\Omega;\R^d))$ such that $\bu_\veps\weakstar \bv$ in $L^\infty(0,T;L^2(\Omega;\R^d))$ and $\bv$ is a weak solution of $\pepso$ with $\bv(0,\cdot) = \bv_0$; see \cite{HLS21, KM82, LM98, Sco94} and the references therein for further details.

\subsection{Asymptotic Limit of DMV solutions}
\label{subsec:asymp-lim-DMV}

In what follows, we assume the existence of a family of DMV solutions $\lbrace\mcv^\veps\rbrace_{\veps>0}$ to the compressible Euler equations, in the sense of Definition \ref{def:dmv-baro}, emanating from the family of well-prepared initial data $\lbrace(\vrho_\epso, \bu_\epso)\rbrace_{\veps>0}$. As a consequence, the DMV solutions will also satisfy the estimate \eqref{eqn:dmv-baro-ent-pig}, cf.\ Lemma \ref{lem:eng-est-pig}.

Feireisl et al.\ in \cite{FKM19} proved that as $\veps\to 0$, the DMV solutions $\mcv^\veps$ will converge to a classical solution $\bv$ of the incompressible Euler system, provided the hypothesis of Theorem \ref{thm:incomp-soln-exis} holds. We briefly recall the arguments here for the sake of completion.

Due to Theorem \ref{thm:incomp-soln-exis}, we have the existence of a classical solution $\bv$ to the incompressible Euler equations and we note that $r = 1$ and $\bU=\bv$ satisfy the regularity requirements of the test functions in \eqref{eqn:rel-ent-mv}. As a consequence, one obtains the following relative energy inequality 
\begin{equation}
\label{eqn:rel-ent-ineq}
    \begin{split}
        \Bigl\lbrack E_{rel}(\mcv^\veps\vert 1,\bv)\Bigr\rbrack_{t=0}^{t=\tau} + \D^\veps(\tau) &\leq \int_0^\tau\int_\Omega\biggl\lbrack\langle\V^\veps; \tvrho\bv - \tbm\rangle\cdot\Dt\bv + \biggl\langle\V^\veps; (\tvrho\bv - \tbm)\otimes\frac{\tbm}{\tvrho}\biggr\rangle\colon\grad\bv\biggr\rbrack\,\dx\,\dt \\
        & - \int_0^\tau\int_\Omega\grad\bv\colon\mathrm{d}\mathfrak{R}^\veps_{cd}(t),
    \end{split}
\end{equation}
see \cite{BF18a, FKM19, FLM+21a} for more details. 

Using some energy estimates arising because of $\mcv^\veps$ being a DMV solution of the compressible system, while also noting that $\bv$ is a classical solution of the incompressible system, yields
\[
E_{rel}(\mcv^\veps\,\vert\,1,\bv)(\tau) + \D^\veps(\tau)\leq C\int_0^\tau\Bigl\lbrack E_{rel}(\mcv^\veps\,\vert\,1,\bv)(t) + \D^\veps(t)\Bigr\rbrack\,\dt.
\]
where $C>0$. Application of Gronwall's Lemma leads to the following result, the statement of which we recall from \cite{FKM19}.

\begin{theorem}[Asymptotic Limit of DMV Solutions]
\label{thm:asymp-lim-dmv}
    Let $\lbrace\mcv^\veps\rbrace_{\veps>0}$ be a family of DMV solutions to the compressible Euler system \eqref{eqn:mss-bal-baro}-\eqref{eqn:mom-bal-baro} in the sense of Definition \ref{def:dmv-baro}, emanating from the family of well-prepared initial data $\lbrace(\vrho_\epso, \bu_\epso)\rbrace_{\veps>0}$. Let $\bv_0\in H^k(\Omega;\R^d)$ and $\Div\bv_0 = 0$, $k>d/2+1$ and suppose $T<T_{max}$, where $T_{max}$ denotes the life span of the classical solution $\bv$ to the incompressible Euler system \eqref{eqn:div-free-incomp}-\eqref{eqn:mom-bal-incomp} with initial data $\bv_0$.

    Then 
    \[
    \D^\veps\to 0\text{ in }L^\infty(0,T)\text{ as }\veps\to 0
    \]
    and
    \[
    \essup_{t\in (0,T)}E_{rel}(\mcv^\veps\,\vert\,1,\bv)(t)\to 0\text{ as }\veps\to 0.
    \]
\end{theorem}

\section{Finite Volume Scheme}
\label{sec:fv-scheme}
In this section, we present a collocated finite volume scheme in order to approximate the compressible Euler system and give a summary of the results related to the energy stability and consistency of the scheme; cf.\ \cite{AGK23, AA24}. 

\subsection{Mesh and Unknowns}
\label{subsec:mesh-and-unkn}

We introduce a tessellation $\T$ of $\Omega\subset\R^d$, known as the primal mesh, consisting of possibly non-uniform closed rectangles ($d=2$) or closed cuboids ($d=3$) such that $\overline{\Omega}=\cup_{K\in\T}K$, where $K$ is called a control volume. The $d$-dimensional Lebesgue measure of $K$ will be denoted by $\absk$ and by $x_K$, we denote the cell-centre of $K$. The set of all edges ($d=2$) or faces ($d=3$) (hereafter commonly referred to as edges) of all control volumes $K\in\T$ is denoted by $\E$ and by $\abssig$, we denote the ($d-1$)- dimensional Lebesgue measure of $\sigma\in\E$. We assume that the mesh is a structured mesh, i.e.\ given any $\sigma\in\E$, $\sigma$ will be orthogonal to one of the standard basis vectors $\mathbf{e}^{(k)}$ of $\R^d$, $k = 1,\dots,d$. By $\E_{int}, \, \E_{ext}$ and $\E(K)$, we denote the subsets of all internal edges, external edges, i.e.\ edges lying on $\partial\Omega$, and the edges of a control volume $K\in\T$, respectively. Any two cells $K,L\in\T$ either have no common edges, or share exactly one common edge denoted by $\sigma=K\vert L$. For $\sigma\in\E(K)$, we denote by $x_\sigma$ the point lying on the edge $\sigma$ such that the line joining $x_K$ and $x_\sigma$ is orthogonal to the edge $\sigma$. For $\sigma= K\vert L$, we set $d_\sigma = \abs{x_K - x_L}$ and if $\sigma \in\E_{ext}\cap \E(K)$, we set $d_\sigma = \abs{x_K - x_\sigma}$. As we consider a structured mesh, it follows that $x_L = x_K \pm d_\sigma \mathbf{e}^{(k)}$ for some $1\leq k\leq d$ and $\sigma = K\vert L$. For $\sigma\in\E(K)$, $\nuk$ denotes the unit normal to $\sigma$ pointing outwards from $K$. By $a\lesssim b$ we mean $a\leq cb$ for a constant $c>0$, independent of the mesh parameters.

The mesh size $h_\T$ is defined by 
\begin{equation}
\label{eqn:mesh-size}
    h_\T=\sup\lbrace h_K\colon K\in\T\rbrace,
\end{equation}
where $h_K = \text{diam}(K)$.

Though the scheme to be introduced is collocated or cell-centered, it is convenient to introduce the dual mesh $\D_\T$ corresponding to the primal mesh $\T$ for the sake of analysis. For each $\sigma\in\E_{int}$, $\sigma=K\vert L$, we associate a dual cell $D_\sigma=D_{K,\sigma}\cup D_{L,\sigma}$ where $D_{K,\sigma}$ (resp.\ $D_{L,\sigma})$ is built by half of $K$ (resp.\ $L$). If $\sigma\in\E_{ext}\cap\E(K)$, we define $D_\sigma=D_{K,\sigma}$ and we set $\D_\T=\lbrace D_\sigma\rbrace_{\sigma\in\E}$ as the collection of all dual cells. It follows that $\abs{\dsig} = d_\sigma\abs{\sigma}$ for any $\sigma\in\E$. See Figure \ref{fig:dual-grid} for a representation of the dual grid.

\begin{figure}[htpb]
    \centering
    \begin{tikzpicture}
        \fill [green!20!white] (2.5,0) rectangle (5,4);
        \fill [red!20!white] (5,0) rectangle (7.5,4);
        \draw[black, thick] (0,0) rectangle (10,4);
        \draw[green!70!black, thick] (2.5, 0) -- (5,0);
        \draw[green!70!black, thick] (2.5, 4) -- (5,4);
        \draw[green!70!black, thick] (2.5,0) -- (2.5,4);
        \draw[red, thick] (5,0) -- (7.5,0);
        \draw[red, thick] (5,4) -- (7.5,4);
        \draw[red, thick] (7.5,0) -- (7.5,4);
        \draw[blue, thick] (5,0) -- (5,4);
        \draw[->, RedViolet, thick] (5,2) -- (6,2);
        \draw[<->, black] (2.5, 4.4) -- (7.5,4.4);
        \node[fill=white] at (5,4.4) {$D_{\sigma}$};
        \draw (5.5,2) node[anchor = south]{\color{RedViolet}{\footnotesize $\nuk$}};
        \filldraw [black] (2.5,2) circle (2pt) node[anchor = north east]{$x_K$};
        \filldraw [black] (5,2) circle (2pt) node[anchor = north east]{$x_\sigma$};
        \filldraw [black] (7.5,2) circle (2pt) node[anchor = north west]{$x_L$};
        \draw (1.25,2) node[anchor = south]{$K$};
        \draw (8.75,2) node[anchor = south]{$L$};
        \draw (5,0) node[anchor = north]{{\color{blue}$\sigma = K\vert L$}};
        \draw (3.75,1) node[anchor = north]{$D_{K,\sigma}$};
        \draw (6.25,1) node[anchor = north]{$D_{L,\sigma}$};
    \end{tikzpicture}
    \caption{Dual grid.}
    \label{fig:dual-grid}
\end{figure}
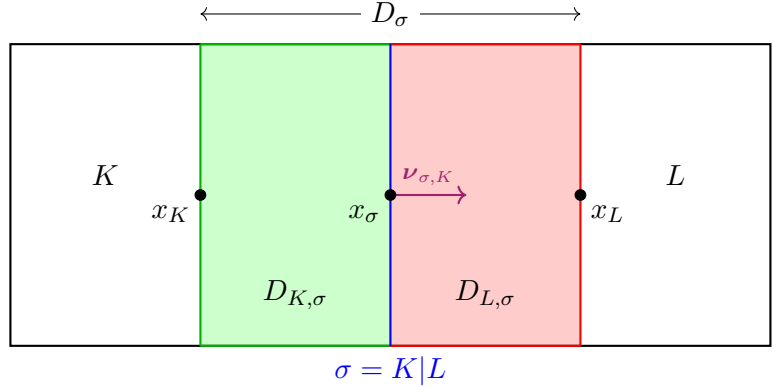
We impose the following regularity conditions on the mesh: 
\begin{enumerate}
    \item There exist constants $0<\underline{\vth}\leq\overline{\vth}$ such that for any $\sigma\in\E(K)\subset\E$, 
    \begin{equation}
    \label{eqn:mesh-regul}
        \underline{\vth}\absk\leq\abs{D_{\sigma}}\leq\overline{\vth}\absk
    \end{equation}
    where $\abs{D_{\sigma}}$ denotes the $d$-dimensional Lebesgue measure of $D_\sigma$.

    \item There exists $\alpha\in (0,1)$ such that 
    \begin{equation}
    \label{eqn:non-flat-cond}
        \alpha h_\T\leq\inf\lbrace{d_\sigma\colon \sigma\in\E}\rbrace
    \end{equation}
\end{enumerate}

\subsection{Discrete Function Spaces}
\label{subsec:func-space}

By $\Lt(\Omega)\subset L^\infty(\Omega)$, we denote the space of all scalar-valued functions constant on each cell $K\in\T$ with the associated projection map $\Pi_\T\colon L^1(\Omega)\to\Lt(\Omega)$ defined as
\begin{equation}
\label{eqn:proj-op-prim}
    \begin{split}
    &\Pi_\T q=\sum_{K\in\T}\bigl(\Pi_T q\bigr)_K\X_K, \\
    &\bigl(\Pi_\T q\bigr)_K=\frac{1}{\absk}\int_K q\,\dx,
    \end{split}
\end{equation}
where $\X_K$ denotes the indicator function of $K$. Analogously, we define $\Lt(\Omega;\R^d)$, the space of all vector-valued piecewise constant functions with the projection operator being defined componentwise. Further, the following stability estimate holds : For any $1\leq p\leq \infty$ and $q\in L^p(\Omega)$,
\begin{equation}
\label{eqn:stab-proj-op}
    \norm{\Pi_\T q}_{L^p(\Omega)}\leq \norm{q}_{L^p(\Omega)},
\end{equation}
see \cite{GMN19} for details.

We define $\He(\Omega)\subset L^\infty(\Omega)$ (resp.\ $\He(\Omega;\R^d)$) as the space of all scalar-valued (resp.\ vector-valued) functions constant on each dual cell $D_\sigma\in\D_\T$. Further, we also define $\Heo(\Omega)\subset\He(\Omega)$ (resp.\ $\Heo(\Omega, \R^d)\subset\He(\Omega;\R^d)$) as the space of piecewise constant functions vanishing on all dual cells $D_\sigma$ corresponding to an external edge $\sigma\in\E_{ext}$.

For $q\in\Lt(\Omega)$, $q_K$ denotes the constant value of $q$ on the cell $K\in\T$ and for $z\in\He(\Omega)$, $z_\sigma$ denotes the constant value of $z$ on the dual cell $D_\sigma\in\D_\T$. Further, if we have $\sigma = K\vert L\in\E(K)$, we define the average value of $q$ across $\sigma$ as 
\begin{equation}
\label{eqn:avg-val-op}
    \ldblbrace q\rdblbrace_{\sigma}=\frac{q_K+q_L}{2},
\end{equation}
and it is defined componentwise if $\uu{q} = (q_1,\cdots,q_d)\in\Lt(\Omega;\R^d)$.
\subsection{Discrete Operators and Discrete Mass Flux}
\label{subsec:disc-op-flx}

We introduce the various discrete differential operators on the primal mesh that are used to discretize the compressible system and we also define their counterparts on the dual mesh which will be helpful in the further analysis. To this end, we begin by defining the upwind mass flux. 

\begin{definition}[Discrete Upwind Mass Flux]
\label{def:mss-flx}
    For each $K\in\T$ and $\sigma = K\vert L\in\E(K)$, the mass flux $\flx\colon\Lt(\Omega)\times\Lt(\Omega;\R^d)\to\R$ is defined as 
    \begin{equation}
    \label{eqn:mss-flx}
        \flx(q,\bw)=\abssig\bigl\lbrace q_K(\wsk)^+ + q_L(\wsk)^-\bigr\rbrace=\abssig\bigl\lbrace \flx^+ +\flx^-\bigr\rbrace,
    \end{equation}
    where we define 
    \[
    \wsk = 
    \begin{dcases}
        \ldblbrace\bw\rdblbrace_\sigma\cdot\nuk, &\text{if }\sigma = K\vert L\in\E(K), \\
        0, &\text{if }\sigma\in\E(K)\cap\E_{ext}.
    \end{dcases}
    \]
    The positive and negative parts of $\wsk$ will be specified later once we give a description of the numerical scheme.
\end{definition}

\begin{definition}[Discrete Divergence and Gradient (Primal mesh)]
\label{def:disc-grad-div}
    In the case of a collocated grid, we follow \cite{HLC20} in defining the discrete divergence and gradient operators so that the classical grad-div duality still holds at the discrete level.

    The discrete divergence operator $\divt\colon\Lt(\Omega;\R^d)\to\Lt(\Omega)$ is defined as 
    \begin{equation}
    \label{eqn:disc-div}
        (\divt \bw)_K=\frac{1}{\absk}\sum_{\sink}\abssig\wsk,
    \end{equation}

    while the discrete gradient operator $\gradt\colon\Lt(\Omega)\to\Lt(\Omega;\R^d)$ reads
    \begin{equation}
    \label{eqn:disc-grad}
        (\gradt q)_K=\frac{1}{\absk}\Bigl\lbrack\,\sum_{\substack{\sink \\ \sigma=K\vert L}}\abssig\ldblbrace q\rdblbrace_{\sigma}\,\nuk+\sum_{\sigma\in\E_{ext}\cap\E(K)}\abssig q_K\,\nuk\Bigr\rbrack.
    \end{equation}
\end{definition}

\begin{definition}[Discrete Gradient (Dual Mesh)]
\label{def:disc-grad-div-dual-mesh}
    We also define a discrete gradient operator on the dual mesh as $\gradd\colon\Lt(\Omega)\to\Heo(\Omega;\R^d)$ with the constant value on each dual cell $D_\sigma\in\D_\T$ being defined as
    \begin{align}
        &(\gradd q)_{\sigma} = \frac{\abssig}{\abs{\dsig}}(q_L - q_K)\nuk \label{eqn:dual-grad}.
    \end{align}
    Note that due to periodic boundary conditions, the dual gradient vanishes on the external edges.
\end{definition}

We now report the following lemma from \cite{HLC20}. 
\begin{lemma}[Discrete Grad-Div Duality (Primal Mesh)]
\label{lem:disc-grad-div-primal}
    The discrete gradient operator defined by \eqref{eqn:disc-grad} satisfies
    \begin{equation}
    \label{eqn:disc-grad2}
        (\gradt q)_K = \frac{1}{\absk}\sum_{\substack{\sink \\ \sigma = K\vert L}}\abssig\frac{q_L - q_K}{2}\nuk.
    \end{equation}
    Further, the discrete analogue of the classical grad-div duality holds, i.e.\ for any $q\in\Lt(\Omega)$ and $\bw\in\Lt(\Omega;\R^d)$ 
    \begin{equation}
    \label{eqn:disc-grad-div-primal}
        \sum_{K\in\T}\absk q_K(\divt\bw)_K + \sum_{K\in\T}\absk(\gradt q)_K\cdot\bw_K=0.
    \end{equation}
\end{lemma}

\begin{definition}[Upwind Divergence Operator]
\label{def:up-div-op}
    We define an upwind divergence operator $\divup\colon\Lt(\Omega)\times\Lt(\Omega;\R^d)\to\Lt(\Omega)$ in order to discretize the convective terms in the mass and momentum balances which reads
    \begin{equation}
    \label{eqn:up-div-op}
            (\divup(q,\bw))_K=\frac{1}{\absk}\sum_{\substack{\sink \\ \sigma=K\vert L}}\flx(q,\bw).
    \end{equation}
    with $\flx(q,\bw)$ as defined in Definition \ref{def:mss-flx}.
\end{definition}

Another operator which is required to perform the consistency analysis of the scheme is the reconstruction operator, which transforms a piecewise constant function on the primal mesh into a piecewise constant function on the dual mesh. We now precisely define the operator and also recall a few key results related to the same.

\begin{definition}[Reconstruction Operator]
    \label{def:recon-op}
    Given a primal mesh $\T$ of $\Omega$ and a corresponding dual mesh $\D_\T$, a reconstruction operator is a map $\mathcal{R}_\T\colon\Lt(\Omega)\to\He(\Omega)$ defined as 
    \begin{equation}
        \label{eqn:recon-op}
        \mathcal{R}_\T q =\sum_{\sigma\in\E} \hat{q}_\sigma\X_{\dsig},
    \end{equation}
    where $\hat{q}_\sigma = \mu_\sigma q_K+(1-\mu_\sigma)q_L$ with $0\leq\mu_\sigma\leq 1$ for $\sigma=K\vert L\in\E_{int}$ and $\hat{q}_\sigma = q_K$ if $\sigma\in\E_{ext}\cap\E(K)$. 
\end{definition}

\begin{lemma}[Stability of the Reconstruction Operator]
    \label{lem:stab-recon-op}
    For any $1\leq p<\infty$, there exists $c\geq 0$ depending only on $p$ such that for any $q\in\Lt(\Omega)$,
    \begin{equation}
        \label{eqn:stab-recon-op}
        \norm{\mathcal{R}_\T q}_{L^p(\Omega)}\leq c\norm{q}_{L^p(\Omega)}.
    \end{equation}
\end{lemma}
\begin{proof}
    See \cite{HLC20}.
\end{proof}

\begin{lemma}[Weak Convergence of Dual Mesh Reconstructions]
    \label{lem:wk-cong-mesh-recon}
    Let $(\T^\m)_{m\in\N}$ be a sequence of space discretizations such that the mesh size $h_{\T^\m}\to 0$ as $m\to\infty$.
    Let $1\leq p<\infty$ and let $(q^\m)_{m\in\N}$ be a uniformly bounded sequence in $L^p(\Omega)$ such that $q^\m\in\mathcal{L}_{\T^\m}(\Omega)$ for each $m\in\N$.
    Let $\mathcal{R}_{\T^\m}$ denote the reconstruction operator on the mesh $\T^\m$. The following hold:
    \begin{enumerate}
        \item $\lVert\mathcal{R}_{\T^\m}q^\m - q^\m\rVert_{L^p(\Omega)}\to 0$ as $m\to\infty$.
        \item If there exists $q\in L^p(\Omega)$ such that $q^\m\rightharpoonup q$ in $L^p(\Omega)$, then $\mathcal{R}_{\T^\m}q^\m\rightharpoonup q$ in $L^p(\Omega)$.
        \item If there exists $\mu\in\M(\overline{\Omega})$ such that $q^\m\weakstar \mu$ in $\M(\overline{\Omega})$, then $\mathcal{R}_{\T^\m}q^\m\weakstar \mu$ in $\M(\overline{\Omega})$.
    \end{enumerate}
\end{lemma}
\begin{proof}
    The proofs of statements (1) and (2) is given in \cite[Lemma 1.16]{AA24}. To prove statement (3), let $\psi\in C(\overline{\Omega})$. Then,
    \begin{align*}
        \int_\Omega\mathcal{R}_{\T^\m}q^\m\psi\,\dx = \int_\Omega (\mathcal{R}_{\T^\m}q^\m - q^\m)\psi\,\dx + \int_\Omega q^\m\psi\,\dx = I_1 + I_2.
    \end{align*}
    Note that, $I_2 \to \int_\Omega\psi\mathrm{d}\mu$ as $m\to\infty$ whereas 
    \[
    \abs{I_1}\leq\lVert\mathcal{R}_{\T^\m}q^\m - q^\m\rVert_{L^1}\norm{\psi}_\infty\to 0\text{ as }m\to\infty
    \]
    as a consequence of (1) and this proves the claim.
\end{proof}

We also require the interpolates and discrete time derivatives of smooth functions in the forthcoming analysis which we now define. Let $\varphi\in\Cinf(\lbrack 0,T)\times\Omega)$. By $\varphi_{\T,\delt}$, we denote its interpolate for the space-time discretization ($\T,\delt$) defined by
\begin{equation}
    \label{eqn:spc-time-ipol}
    \varphi_{\T,\delt}(t,x)=\sum_{n=0}^{N-1}\varphi^n_K\X_{K}(x)\X_{[t_n,t_{n+1})}(t),
\end{equation}
where $\varphi^n_K=(\Pi_\T\varphi(t_n,x))_K=\frac{1}{\absk}\int_K\varphi(t_n,x)\,\dx$. If $\uu{\varphi}\in\Cinf(\lbrack 0,T)\times\Omega;\R^d)$, we define the interpolate componentwise.

We now define a discrete time derivative operator $\eth_t$
\begin{equation}
    \label{eqn:disc-time-der}
    \eth_t\varphi_{\T,\delt}(t,x)=\sum_{n=0}^{N-1}\sum_{K\in\T}\frac{\varphi^{n+1}_K-\varphi^n_K}{\delt}\X_K(x)\X_{[t_n,t_{n+1})}(t).
\end{equation}
The regularity of $\varphi$ ensures that the quantities $\varphi_{\T,\delt}, \eth_t\varphi_{\T,\delt}$ converge uniformly to $\varphi,\Dt\varphi$ respectively as $h_\T\to 0$.

We also recall a few basic interpolation inequalities that are used in the numerical analysis and we refer the interested reader to \cite{BLM+23, EGH00, FLM+19, FLM+21a, GMN19} for more details. Assuming that $\varphi\in\Cinf(\Omega)$, $\uu{\varphi}\in\Cinf(\Omega;\R^d)$ and $1\leq p\leq\infty$, the following hold:
\begin{equation}
    \label{eqn:op-est}
    \begin{split}
    &\norm{\Pi_\T\varphi}_{L^p(\Omega)}\lesssim\norm{\varphi}_{\infty} , \ \norm{\varphi-\Pi_\T\varphi}_{L^p(\Omega)}\lesssim h_\T\norm{\varphi}_{\infty}, \\
    & \norm{\gradd\Pi_\T\varphi}_{L^p(\Omega;\R^d)}\lesssim\norm{\varphi}_{\infty} , \ \norm{\grad\varphi-\gradd\Pi_\T\varphi}_{L^p(\Omega;\R^d)}\lesssim h_\T\norm{\varphi}_\infty ,\\
    &\norm{\grad\varphi - \gradt\Pi_\T\varphi}_{L^p(\Omega;\R^d)}\lesssim h_\T\norm{\varphi}_\infty, \ \norm{\Div\uu{\varphi}-\divt\Pi_\T\uu{\varphi}}_{L^p(\Omega)}\lesssim h_\T\norm{\varphi}_\infty.
    \end{split}
\end{equation}

\subsection{The Scheme}
\label{subsec:scheme}

Let us consider a discretization $0=t^0<t^1<\cdots<t^N=T$ of the time interval $\lbrack 0,T\rbrack$ and let $\delt=t^{n+1}-t^n$ for $n=0,1,\dots,N-1$ be the constant timestep. We suppose that the initial data is well prepared in the sense of Definition \ref{def:well-prep} and we also further assume that $\bv_0\in H^k(\Omega;\R^d)$, with $k>d/2+1$ and $T<T_{max}$; cf.\ Theorem \ref{thm:incomp-soln-exis}.

In what follows, we omit the dependence on $\veps$ for the discrete unknowns unless specified otherwise. The scheme is initialized by considering the initial approximations to be the projections of the initial data $\vrho_\epso$ and $\bu_\epso$ respectively, i.e.\
\begin{equation}
    \label{eqn:disc-ini-dat}
    \begin{split}
        &\rk^0=(\Pi_\T\vrho_\epso)_K=\frac{1}{\absk}\int_K\vrho_\epso\,\dx, \\
        &\uk^0=(\Pi_\T\bu_\epso)_K=\frac{1}{\absk}\int_K\bu_\epso\,\dx.
    \end{split}
\end{equation}
Further, we also suppose that the initial total energy is bounded uniformly with respect to $\veps$. That is, there exists $C>0$, independent of $\veps$, such that

\begin{equation}
\label{eqn:ini-enrg-bound}
    \begin{split}
        &\sum_{K\in\T}\absk\biggl(\half\rk^0\lvert\uk^0\rvert^2+\frac{1}{\veps^2}\psig(\rk^0)\biggr)\leq C, \\
        &\sum_{K\in\T}\absk\biggl(\half\rk^0\lvert\uk^0\rvert^2+\frac{1}{\veps^2}\pig(\rk^0)\biggr)\leq C.
    \end{split}
\end{equation}

We now consider the following fully discrete scheme approximating $\peps$, which we denote by $\pepsh$. For $0\leq n\leq N-1$:
\begin{subequations}
    \begin{align}
        &\frac{1}{\delt}(\rk^{n+1}-\rk^n)+(\divup(\vrho^{n+1},\bw^n))_K=0, \ \forall\,K\in\T, \label{eqn:disc-mss-bal}\\
        &\frac{1}{\delt}(\rk^{n+1}\uk^{n+1}-\rk^n\uk^n)+(\divup(\vrho^{n+1}\bu^n,\bw^n))_K+\frac{1}{\veps^2}(\gradt p^{n+1})_K= 0, \ \forall\,K\in\T, \label{eqn:disc-mom-bal}
    \end{align}
\end{subequations}
where $\bw^n = \bu^n - \delta\bu^{n+1}$ denotes the stabilized velocity, with the stabilization term $\delta\bu^{n+1}_K = \frac{\eta\delt}{\veps^2}(\gradt p^{n+1})_K$, $\eta>0$ to be defined later, being introduced in order to recover a discrete variant of the admissibility criteria \eqref{eqn:baro-entropy}.

The positive and negative parts of the term $\wsk$ appearing in the mass flux $\flx$ are defined so as to bring about an upwind-bias and maintain the signs split in $\flx$.
\begin{equation}
    \label{eqn:pos-neg-stab-velo}
    \begin{split}
    &(w_{\sk}^n)^{+}=u_{\sk}^{n,+}-\delta u^{n+1,-}_{\sk}\geq 0, \\
    &(w_{\sk}^n)^{-}=u_{\sk}^{n,-}-\delta u^{n+1,+}_{\sk}\leq 0,
    \end{split}
\end{equation}
wherein $a^{\pm}=\dfrac{a\pm\abs{a}}{2}$ denotes the standard positive and negative halves.

As a consequence of the initial data being well-prepared, we obtain the following estimate on the initial density $\vrho^0_K$.
\begin{proposition}
\label{prop:ini-den-bound}
    Suppose that the initial data $\lbrace(\vrho_{\epso}, \bu_{\epso})\rbrace_{\veps>0}$ are well-prepared in the sense of Definition \ref{def:well-prep}. Then, there exists a constant $C>0$ independent of $\veps$ such that
    \begin{equation}
    \label{eqn:ini-den-bound}
        \frac{1}{\veps^2}\max_{K\in\T}\abs{\rk^0 - 1}\leq C.
    \end{equation}
\end{proposition}
\begin{proof}
    For each $K\in\T$, \eqref{eqn:disc-ini-dat} along with Definition \ref{def:well-prep} yields the existence of a $C>0$ such that
    \[
    \abs{\rk^0 - 1} \leq \frac{1}{\absk}\int_K\abs{\vrho_{\epso} - 1}\,\dx\leq\norm{\vrho_{\epso} - 1}_{L^\infty(\Omega)}\leq C\veps^2,
    \]
    which in turn proves \eqref{eqn:ini-den-bound}.
\end{proof}

The proposed scheme is nonlinear in nature, mainly due to the presence of the stabilization term present in the mass balance \eqref{eqn:disc-mss-bal}. However, once we obtain the value of $\vrho^{n+1}$, the value of the velocity $\bu^{n+1}$ can be computed explicitly from \eqref{eqn:disc-mom-bal}. As such, we first prove that a solution indeed exists to the scheme. The existence is established using a standard degree argument, using the classical tools of topological degree theory in finite dimensions; see \cite{AGK23,Dei85,GMN19,Nir01,OCC06} for further details. We just recall the result for completeness.

\begin{theorem}[Existence of a Discrete Solution]
    \label{thm:exist-num-sol}
    Let $(\vrho^n,\bu^n)\in\Lt(\Omega)\times\Lt(\Omega;\R^{d})$ be such that $\vrho^n>0$ on $\Omega$. Then, there exists a solution $(\vrho^{n+1},\bu^{n+1})\in\Lt(\Omega)\times\Lt(\Omega;\R^{d})$ of $\pepsh$ satisfying $\vrho^{n+1}>0$ on $\Omega$.
\end{theorem}
\begin{proof}
    See \cite[Theorem 4.2]{AGK23}
\end{proof}

Since the existence of a discrete solution to the scheme is known, we define the `numerical solution' of the scheme as follows:

\begin{definition}[Numerical Solution of the Scheme]
    \label{def:scheme-soln}
    Let $\lbrace(\vrho^n,\bu^n)\in\Lt(\Omega)\times\Lt(\Omega;\R^d)\colon n=0,\dots,N-1\,;\,\vrho^n>0\rbrace$ be a family of discrete solutions to the scheme $\pepsh$. Then a numerical solution of the scheme corresponding to the space-time discretization $(\T,\delt)$ is defined to be the pair of piecewise constant functions 
    \begin{equation}
        \label{eqn:scheme-soln}
            \begin{split}
                &\vrho_{\T,\delt,\veps}(t,x)=\sum_{n=0}^{N-1}\sum_{K\in\T}\rk^n\X_K(x)\X_{[t_n,t_{n+1})}(t), \\
                &\bu_{\T,\delt,\veps}(t,x)=\sum_{n=0}^{N-1}\sum_{K\in\T}\uk^n\X_K(x)\X_{[t_n,t_{n+1})}(t).
            \end{split}
    \end{equation}
\end{definition}

Theorem \ref{thm:exist-num-sol} guarantees that the discrete density $\vrho^n>0$. However, we cannot claim nor prove the existence of suitable bounds on the density that ensure it being uniformly bounded below away from zero and being uniformly bounded above. The density staying in the non-degenerate region plays a key role in our analysis, and as such we make the following as our principal working hypothesis.

\begin{hypothesis}
\label{hyp:den-bound}
    There exist constants $0<\underline{\vrho}\leq\overline{\vrho}$ (independent of the mesh parameters and $\veps$) such that 
    \[
        0<\underline{\vrho}\leq\vrho_{\T,\delt}\leq\overline{\vrho}
    \]
\end{hypothesis}

Numerical solutions to $\pepsh$ satisfy a local in-time energy identity, which is a discrete variant of the admissibility criteria \eqref{eqn:baro-entropy}. We recall the result from \cite{AA24}. 

\begin{theorem}[Local In-Time Energy Inequality]
\label{thm:disc-loc-ent-ineq}
    Suppose that the following conditions hold:
    \begin{enumerate}
        \item $\eta>\dfrac{1}{\rk^{n+1}}$;
        
        \item $\delt\leq\dfrac{\absk\rk^{n+1}}{2\displaystyle\sum_{\substack{\sink \\ \sigma=K\vert L}}\abssig\vrho^{n+1}_L (-(w^n_{\sk})^{-})}$.
    \end{enumerate}
    Then, any numerical solution to $\pepsh$ satisfies the following local in-time energy inequality,
    \begin{equation}
    \label{eqn:disc-loc-ent-ineq-1}
        \sum_{K\in\T}\absk\biggl(\half\rk^{n+1}\abs{\uk^{n+1}}^2+\frac{1}{\veps^2}\psig(\rk^{n+1})\biggr)\leq \sum_{K\in\T}\absk\biggl(\half\rk^n\abs{\uk^n}^2+\frac{1}{\veps^2}\psig(\rk^n)\biggr)
    \end{equation}
    for each $0\leq n\leq N-1$.
\end{theorem}

\begin{proof}
    See Appendix \ref{subsec:app-proof-loc-ent}.
\end{proof}

One also obtains the following global energy estimate, see also \cite[Theorem 1.15]{AA24}

\begin{theorem}[Global Energy Estimate]
\label{thm:disc-glob-ent-est}
    There exists a constant $C>0$, independent of $\veps$, such that the discrete solutions $(\vrho^n,\bu^n)_{0\leq n\leq N}$ satisfy the following global energy inequality:
    \begin{equation}
    \label{eqn:disc-glob-ent-est-1}
        \sum_{K\in\T}\absk\biggl(\half\rk^n\abs{\uk^n}^2+\frac{1}{\veps^2}\psig(\rk^n)\biggr) + \frac{1}{\veps^4}\sum_{r=0}^{n-1}\delt^2\sum_{K\in\T}\abs{K}\biggl(\eta - \frac{1}{\vrho^{r+1}_K}\biggr)\abs{(\gradt p^{r+1})_K}^2\leq C,
    \end{equation}
\end{theorem}

\begin{proof}
    See Appendix \ref{subsec:app-proof-glob-ent}.
\end{proof}

\begin{corollary}
    Under the conditions (1) and (2) given in Theorem \ref{thm:disc-loc-ent-ineq}, one obtains
    \begin{enumerate}
        \item a local in-time entropy inequality 
        \begin{equation}
        \label{eqn:disc-loc-ent-ineq}
            \sum_{K\in\T}\absk\biggl(\half\rk^{n+1}\abs{\uk^{n+1}}^2+\frac{1}{\veps^2}\pig(\rk^{n+1})\biggr)\leq \sum_{K\in\T}\absk\biggl(\half\rk^n\abs{\uk^n}^2+\frac{1}{\veps^2}\pig(\rk^n)\biggr),
        \end{equation}

        \item a global entropy estimate 
        \begin{equation}
        \label{eqn:disc-glob-ent-est}
            \sum_{K\in\T}\absk\biggl(\half\rk^n\abs{\uk^n}^2+\frac{1}{\veps^2}\pig(\rk^n)\biggr) + \frac{1}{\veps^4}\sum_{r=0}^{n-1}\delt^2\sum_{K\in\T}\abs{K}\biggl(\eta - \frac{1}{\vrho^{r+1}_K}\biggr)\abs{(\gradt p^{r+1})_K}^2\leq C.
        \end{equation}
    \end{enumerate}
\end{corollary}

\begin{proof}
    See Appendix \ref{subsec:app-proof-loc-ent} and \ref{subsec:app-proof-glob-ent}.
\end{proof}

\begin{remark}
\label{rem:vel-pres-grad-est}
    The inequality \eqref{eqn:disc-glob-ent-est} along with Hypothesis \ref{hyp:den-bound} yields $\bu_{\T,\delt,\veps}\in L^\infty(0,T;L^2(\Omega;\R^d))$ and $\dfrac{\sqrt{\delt}}{\veps^2}\gradt p_{\T,\delt,\veps}\in L^2(\odom;\R^d))$, where $p_{\T,\delt,\veps}(t,x) = \sum_{n = 0}^{N-1} \sum_{K\in\T}p^n_K\X_K(x)\X_{\lbrack t_n,t_{n+1})}(t)$.
\end{remark}

We now state and give a proof of the weak consistency of the scheme $\pepsh$.

\begin{theorem}[Consistency of the Numerical Scheme]
\label{thm:cons-scheme}
Let $(\T,\delt)$ be a given space-time discretization of the domain $\dom$. Let $(\vrho_{\T,\delt,\veps},\bu_{\T,\delt,\veps})$ denote a numerical solution to $\pepsh$ in the sense of Definition \ref{def:scheme-soln} with the initial data given by \eqref{eqn:disc-ini-dat} and let $\veps>0$ be fixed. 
Then, the scheme is consistent with the weak formulation of $\peps$, i.e.\
\begin{equation}
    \label{eqn:mss-cons}
    \begin{split}
        -\int_{\Omega}\vrho_{\T,\delt,\veps}(0,x)\varphi(0,x)\,\dx\,=\,\int_0^T\int_\Omega\bigl\lbrack &\vrho_{\T,\delt,\veps}\Dt\varphi
       +\vrho_{\D,\delt,\veps}\bu_{\D,\delt,\veps}\cdot\grad\varphi\bigr\rbrack\,\dx\,\dt + \mathcal{S}^{mass}_{\T,\delt}
    \end{split}
\end{equation}   
for any $\varphi\in\Cinf(\lbrack 0,T)\times \Omega)$;
\begin{equation}
    \label{eqn:mom-cons}
    \begin{split}
    -\int_\Omega\vrho_{\T,\delt,\veps}(0,x)\bu_{\T,\delt,\veps}(0,x)\cdot\uu{\varphi}(0,x)\,\dx\,&=\,\int_0^T\int_\Omega\bigl\lbrack\vrho_{\T,\delt,\veps}\bu_{\T,\delt,\veps}\cdot\grad\uu{\varphi} 
    +\vrho_{\D,\delt,\veps}(\bu_{\D,\delt,\veps}\otimes\bu_{\D,\delt,\veps})\colon\grad\uu{\varphi}\\
    &+\frac{1}{\veps^2}p_{\T,\delt,\veps}(t+\delt, x)\Div\uu{\varphi}(t,x)\bigr)\bigr\rbrack\,\dx\,\dt + \mathcal{S}^{mom}_{\T,\delt}
    \end{split}
\end{equation}
for any $\uphi\in\Cinf(\lbrack 0, T)\times\Omega;\R^d)$.

Here, $\vrho_{\D,\delt,\veps}$ and $\bu_{\D,\delt,\veps}$ denote the reconstructions of $\vrho_{\T,\delt,\veps}$ and $\bu_{\T,\delt,\veps}$ on the dual mesh and the consistency errors $\mathcal{S}^{mass}_{\T,\delt}$ and $\mathcal{S}^{mom}_{\T,\delt}$ are such that
\begin{equation}
    \label{eqn:cons-errors}
    \begin{split}
        &\abs{\mathcal{S}^{mass}_{\T,\delt}} = o(h_\T,\sqrt{\delt}) \to 0\text{ as }h_\T\to 0,\\
        &\abs{\mathcal{S}^{mom}_{\T,\delt}} = o(h_\T,\sqrt{\delt}) \to 0\text{ as }h_\T\to 0.
    \end{split}
\end{equation}

\begin{proof}
    The proof of the consistency result is the same as in \cite[Theorem 1.17]{AA24}. However, an additional assumption of the discrete velocity being uniformly bounded is used there which we do not impose here. As such, we only give the explicit forms of the error terms $\mathcal{S}^{mass}_{\T,\delt}, \mathcal{S}^{mom}_{\T,\delt}$ and give the proof of estimating the same. The mass consistency error is given by 
    \begin{equation}
    \label{eqn:mss-cons-err}
        \begin{split}
            &\mathcal{S}^{mass}_{\T,\delt} = T_{mass}+\,\int_0^T\int_\Omega \vrho_{\T,\delt,\veps}(\eth_t\varphi_{\T,\delt}-\Dt\varphi)\,\dx\,\dt 
            +\int_{0}^T\int_\Omega\vrho_{\D,\delt,\veps}\bu_{\D,\delt,\veps}\cdot(\gradd\varphi_{\T,\delt}-\grad\varphi)\,\dx\,\dt \\
            &+\int_{\Omega}\vrho_{\T,\delt,\veps}(0,x)(\varphi_{\T,\delt}(0,x)-\varphi(0,x))\,\dx\  \\
            & = T_{mass}+R_1+R_2+R_3.
        \end{split}
    \end{equation}
    where 
    \[
    T_{mass} = -\sum_{n=0}^{N-1}\delt\sum_{K\in\T}\sum_{\substack{\sink \\ \sigma=K\vert L}}\abssig\Bigl(\rk^{n+1} \delta u^{n+1,-}_{\sk}+\vrho^{n+1}_L\delta u^{n+1,+}_{\sk}\Bigr)\varphi^{n}_K.
    \]

    The terms $R_1$ and $R_3$ are estimated readily due to the convergence of the interpolates to their continuous counterparts. To estimate $R_2$, we use Lemma \ref{lem:stab-recon-op} along with Remark \ref{rem:vel-pres-grad-est} and the estimates provided by \eqref{eqn:op-est} to get
    \[
    \abs{R_2}\lesssim\overline{\vrho}\norm{\bu_{\T,\delt,\veps}}_{L^2}\norm{\gradd\varphi_{\T,\delt} - \grad\varphi}_{L^2}\lesssim h_\T.
    \]
    $T_{mass}$, after rewriting, can be estimated using Remark \ref{rem:vel-pres-grad-est} and \eqref{eqn:op-est} as 
    \[
    \abs{T_{mass}}\lesssim \overline{\vrho}\eta\sqrt{\delt}\norm{\varphi}_{\infty}\norm{\frac{\sqrt{\delt}}{\veps^2}\gradt p_{\T,\delt,\veps}}_{L^2} \lesssim \sqrt{\delt}.
    \]
    As a consequence, we get $\mathcal{S}^{mass}_{\T,\delt} = o(h_\T,\sqrt{\delt})$.

    The momentum consistency error is given by 
    \begin{equation}
    \label{mom-cons-err}
        \begin{split}
            &\mathcal{S}^{mom}_{\T,\delt} = T_{mom}+\,\int_0^T\int_\Omega \vrho_{\T,\delt,\veps}\bu_{\T,\delt,\veps}\cdot(\eth_t\uphi_{\T,\delt}-\Dt\uphi)\,\dx\,\dt \\
            &+ \int_{0}^T\int_\Omega\vrho_{\D,\delt,\veps}(\bu_{\D,\delt,\veps}\otimes\bu_{\D,\delt,\veps})\colon(\gradd\uphi_{\T,\delt}-\grad\uphi)\,\dx\,\dt \\
            &+\frac{1}{\veps^2}\int_0^T\int_\Omega p_{\T,\delt,\veps}(t+\delt,x)(\divt\uphi_{\T,\delt}(t,x)-\Div\uphi(t,x))\,\dx\,\dt \\
            &+ \int_{\Omega}\vrho_{\T,\delt,\veps}(0,x)\bu_{\T,\delt,\veps}(0,x)\cdot(\uphi_{\T,\delt}(0,x)-\uphi(0,x))\,\dx
            = T_{mom}+S_1+S_2+S_3+S_4.
        \end{split}
    \end{equation}
    where 
    \[
    T_{mom} = -\sum_{n=0}^{N-1}\delt\sum_{K\in\T}\sum_{\substack{\sink \\ \sigma=K\vert L}}\abssig\Bigl(\rk^{n+1}\uk^n \delta u^{n+1,-}_{\sk}+\vrho^{n+1}_L\bu^n_L\delta u^{n+1,+}_{\sk}\Bigr)\cdot\uphi^{n}_K.
    \]
    The terms $S_1$ and $S_4$ again are estimated due to the convergence of the interpolates to their continuous counterparts. $S_2$ and $S_3$ can be estimated in the same manner as $R_2$. Finally, $T_{mom}$ is estimated as 
    \[
    \abs{T_{mom}}\lesssim \overline{\vrho}\eta\sqrt{\delt}\norm{\uphi}_{\infty}\norm{\bu_{\T,\delt,\veps}}_{L^2}\norm{\frac{\sqrt{\delt}}{\veps^2}\gradt p_{\T,\delt,\veps}}_{L^2}\lesssim\sqrt{\delt},
    \]
    which gives us $\mathcal{S}^{mom}_{\T,\delt} = o(h_\T,\sqrt{\delt})$.
\end{proof}

\end{theorem}

\section{Analysis of the Iterative Limits}
\label{sec:lim-analys-scheme}

The goal of this section is twofold:\ to analyze the two limits $\mathcal{I} = \lim_{\veps\to 0}\Bigl(\lim_{h_\T\to 0}\pepsh\Bigr)$ and $\mathcal{J} = \lim_{h_\T\to 0}\Bigl(\lim_{\veps\to 0}\pepsh\Bigr)$. We note that there is a slight abuse of notation, in the sense that $\lim_{h_\T\to 0}\pepsh$ means taking the limit $h_\T\to 0$ of the numerical solutions generated by the scheme for a fixed $\veps>0$. On the other hand, $\lim_{\veps\to 0}\pepsh$ means taking the limit $\veps\to 0$ of the numerical solutions generated by the scheme on a fixed mesh with size $h_\T$. For the sake of convenience and the simplicity of exposition, however, we stick to the same notations as above. Throughout this section, we assume that Hypothesis \ref{hyp:den-bound} and the conditions imposed by Theorem \ref{thm:disc-loc-ent-ineq} hold.

\subsection{Analysis of $\mathcal{I}$}
\label{subsec:first_limit}
We first consider $\lim_{h_\T\to 0}\pepsh$.  
Consider a family of space time discretizations $(\T^\m,\delt^\m)_{m\in\N}$ such that the step size $h_{\T^\m} = h^\m\to 0$ as $m\to\infty$. Let $\vrho_{\T^\m,\delt^\m,\veps} = \vrho^\m_\veps$, $\bu_{\T^\m,\delt^\m,\veps} = \bu^\m_{\veps}$ and set $\bm^\m_\veps = \vrho^\m_\veps\bu^\m_\veps$. Then, Hypothesis \ref{hyp:den-bound} along with Theorem \ref{thm:disc-glob-ent-est} ensures that $(\vrho^\m_\veps)_{m\in\N}\subset L^\infty(\odom)$ and $(\bm^\m_\veps)_{m\in\N}\subset L^\infty((0,T);L^2(\Omega;\R^d))$. The so-called `Fundamental Theorem of Young Measures', cf.\ \cite{JBal89, MNR+96, Ped97}, asserts that there exists a parameterized family of probability measures $\mcv^\veps = \lbrace \V^\veps\rbrace_{(t,x)\in\odom}\in L^\infty_{weak-*}(\odom;\Pro(\F_{comp}))$ such that 
\begin{align*}
    &\vrho^\m_\veps\weakstar\langle\V^\veps; \tvrho\rangle\text{ in }L^\infty(\odom), \\
    &\bm^\m_\veps\weakstar\langle\V^\veps; \tbm\rangle\text{ in }L^\infty((0,T);L^2(\Omega;\R^d)),
\end{align*}
possibly for a subsequence (not relabelled).

The consistency formulation, cf.\ Theorem \ref{thm:cons-scheme}, will now be satisfied by every pair $(\vrho^\m_\veps,\bm^\m_\veps)$ for each $m\in\N$. As a consequence, one can pass to the limit $m\to\infty$ in the formulation with the help of the above convergence and one obtains the following result. As before, we refer the reader to \cite[Theorem 1.18]{AA24} for the complete proof.

\begin{theorem}[Weak Convergence]
\label{thm:scheme-wk-con}
    Let $(\T^\m,\delt^\m)_{m\in\N}$ be a sequence of space-time discretizations such that $h^\m=h_{\T^\m}\to 0$ as $m\to\infty$ and let $\veps>0$ be fixed. Let $\lbrace(\vrho^\m_\veps,\bu^\m_\veps)\rbrace_{m\in\N}$, $\vrho^\m_\veps = \vrho_{\T^\m,\delt^\m, \veps}$ and $\bu^\m_\veps = \bu_{\T^\m,\delt^\m, \veps}$, be the sequence of numerical solutions generated by the scheme $\lbrack\Pro^{h^\m}_\veps\rbrack$. Let $\bm^\m_\veps = \vrho^\m_\veps\bu^\m_\veps$. Then, there exists a subsequence (not relabelled) such that
    \begin{align*}
        &\vrho^\m_\veps\weakstar\langle\V^\veps; \tvrho\rangle\text{ in }L^\infty(\odom), \\
        &\bm^\m_\veps\weakstar\langle\V^\veps;
        \tbm\rangle\text{ in }L^\infty(0,T; L^2(\Omega; \R^{d}))
    \end{align*}
    as $m\to\infty$, where $\mcv^\veps=\lbrace\V^\veps\rbrace_{(t,x)\in\odom}$ is a DMV solution of $\peps$ (in the sense of Definition \ref{def:dmv-baro}) with 
    \begin{equation}
    \label{eqn:conc-def}
        \begin{split}
            &\mathfrak{C}^\veps_{cd}(t) = \overline{\half\frac{\abs{\bm^2}}{\vrho}}(t,\cdot) - \biggl\langle\V^\veps; \half\frac{\abs{\tbm}^2}{\tvrho}\biggr\rangle, \\
            &\mathfrak{R}^\veps_{cd}(t) = \overline{\frac{\bm\otimes\bm}{\vrho}}(t,\cdot) - \biggl\langle\V^\veps; \frac{\tbm\otimes\tbm}{\tvrho}\biggr\rangle,  
        \end{split}
    \end{equation}
    where $\displaystyle\overline{\half\frac{\abs{\bm^2}}{\vrho}}$ is the weak-star limit of 
        $\displaystyle\biggl(\half\frac{\lvert{\bm^\m_\veps}\rvert^2}{\vrho^\m_\veps}\biggr)_{m\in\N}$ in $L^\infty(0,T;\M^{+}(\overline{\Omega}))$ and $\displaystyle\overline{\frac{\bm\otimes\bm}{\vrho}}$ is the weak-star limit of 
        $\displaystyle\biggl(\frac{\bm^\m_\veps\otimes\bm^\m_\veps}{\vrho^\m_\veps}\biggr)_{m\in\N}$ in $L^\infty(0,T;\M(\overline{\Omega};\R^{d\times d}))$ as $m\to\infty$.
\end{theorem}

\begin{remark}
    In \cite{AA24}, an additional assumption of the discrete velocity being uniformly bounded was used in order to prove the consistency of the scheme. The consequence of this assumption was that the concentration defects were $0$ when passing to the limit as given in \cite[Theorem 1.18]{AA24}. However, as the proof of Theorem \ref{thm:cons-scheme} shows, this assumption seems to be surplus and as such, we do not impose the same in the current work. This, however, gives us non-zero concentration defects as given in \eqref{eqn:conc-def}.
\end{remark}

Therefore, one has for a fixed $\veps>0$, $\lim_{h_\T\to 0}\pepsh = \mcv^\veps$, which is a DMV solution of $\peps$ with initial data $(\vrho_{\epso}, \bm_{\epso})$, where $\bm_\epso = \vrho_\epso\bu_\epso$.

Now, taking into account that the initial data are well-prepared (in the sense of Definition \ref{def:well-prep}) and as we also suppose that $\bv_0\in H^k(\Omega;\R^d)$ with $\Div\bv_0 = 0$, $k>d/2+1$, one immediately ascertains the following result using Lemma \ref{lem:eng-est-pig} and Theorem \ref{thm:asymp-lim-dmv}. 

\begin{theorem}
\label{thm:eps-h-lim}
    Let $(\T^\m,\delt^\m)_{m\in\N}$ be a sequence of space-time discretizations such that $h^\m=h_{\T^\m}\to 0$ as $m\to\infty$. Let $(\veps_k)_{k\in\N}$ be a decreasing sequence of positive numbers such that $\veps_k\to 0$ as $k\to\infty$. Let $\lbrace(\vrho^{m,k},\bu^{m,k})\rbrace_{m,k\in\N}$, $\vrho^{m,k} = \vrho_{\T^\m,\delt^\m, \veps_k}$ and $\bu^{m,k} = \bu_{\T^\m,\delt^\m, \veps_k}$, be the sequence of numerical solutions generated by the scheme $\lbrack\Pro^{h^\m}_{\veps_k}\rbrack$. Assume that the initial data $\lbrace(\vrho_{0, \veps_k},\bu_{0,\veps_k})\rbrace_{k\in\N}$ is well-prepared (in the sense of Definition \ref{def:well-prep}) and also assume that $\bv_0\in H^k(\Omega;\R^d)$ with $\Div\bv_0 = 0$, $k>d/2+1$. Finally, let $\bv\in C^1\bigl(\dom;\R^d\bigr)$ be the classical solution of the incompressible Euler system $\pepso$ emanating from $\bv_0$. Then $(\vrho^{m,k},\bu^{m,k})$ converges in the weak-* sense to a DMV solution $\mcv^k = \mcv^{\veps_k}$ of the compressible Euler system as $m\to\infty$ for each fixed $k$. Furthermore, the family of DMV solutions $\lbrace\mcv^k\rbrace_{k\in\N}$ converge to $\bv$ as $k\to\infty$, i.e.
    \[
    \essup_{t\in(0,T)}E_{rel}(\mcv^k\vert 1,\bv)(t)\to 0\text{ as }k\to\infty.
    \]In other words, we have
    \[
    \lim_{k\to\infty}\biggl(\lim_{m\to\infty}(\vrho^{m,k}, \bu^{m,k})\biggr) = (1,\bv).
    \]
\end{theorem}

\subsection{Analysis of $\mathcal{J}$}
\label{subsec:second-lim}

We begin by analyzing the numerical solutions of the scheme $\pepsh$ as we first perform the zero Mach number limit ($\veps\to 0$) at the discrete level. One of the key ingredients needed to pass to the limit is a uniform bound on the second order pressure $\pi^n_K = (p^n_K - m(p^n))/\veps^2$, where $m(p^n) = \abs{\Omega}^{-1}\sum_{K\in\T}\absk p^n_K$ is the average value of pressure. Performing a Hilbert expansion of the pressure term followed by a formal asymptotic analysis at the continuous level yields a decomposition $p = p_{(0)} + \veps^2\pi$ in the low Mach regime. Here , $p_{(0)}$ is the thermodynamic pressure. The second order pressure, also known as the hydrodynamic pressure, $\pi$ survives in the limit $\veps\to 0$; cf.\ \eqref{eqn:div-free-incomp}-\eqref{eqn:mom-bal-incomp}.  This needs to be respected also at the discrete level and hence, a control on the discrete second-order pressure term is needed in order to pass to the limit. We begin by recalling a discrete variant of the Sobolev-Poincar\'{e} inequality from \cite{FLM+21a}.

\begin{lemma}[Discrete Sobolev-Poincar\'{e} inequality]
\label{lem:sob-poinc-ineq}
    For any $q\in\Lt(\Omega)$, the following inequality holds:
    \begin{equation}
    \label{eqn:sob-poinc-ineq}
        \norm{q - m(q)}_{L^2(\Omega)}\lesssim\norm{\gradd q}_{L^2(\Omega)}.
    \end{equation}
    where $m(q) = \abs{\Omega}^{-1}\sum_{K\in\T}\absk q_K.$
\end{lemma}

Next, we prove a `negative' estimate on the discrete dual gradient \eqref{eqn:disc-grad2}.

\begin{lemma}
\label{lem:neg-grad-est}
    Let $q\in\Lt(\Omega)$ and $\mu\in(0,1)$ be a fixed constant. Then, there exists a constant $C_\T>0$ which depends on the inverse of the mesh size $h_\T$, the mesh regularity constants $\overline{\vartheta},\alpha$, the constant $\mu$ and the dimension $d$ such that
    \begin{equation}
    \label{eqn:neg-grad-est}
        \norm{\gradd q}_{L^p(\Omega)}\leq C_\T \norm{q}_{L^p(\Omega)},
    \end{equation}
    for any $1\leq p<\infty$.
\end{lemma}
\begin{proof}
    Let $\mu\in(0,1)$. Let $q_\mu = \sum_{\sigma\in\E} \hat{q}_\sigma\X_{\dsig}$ denote a reconstruction of $q$ on the dual mesh where $\hat{q}_{\sigma} = \mu q_L + (1-\mu)q_K$ if $\sigma = K\vert L\in\E_{int}$, else $\hat{q}_\sigma = q_K$ if $\sigma\in\E_{ext}\cap\E(K)$. 
    
    As a consequence of the mesh regularity assumption \eqref{eqn:non-flat-cond}, we have that for any $\sigma\in\E$, $\frac{\abssig}{\abs{\dsig}}\leq\alpha^{-1}h_\T^{-1}$. Utilizing the inequality $\abs{b-a}^p \leq \frac{2^{p-1}}{\mu^p}\bigl(\abs{\mu b + (1-\mu)a}^p + \abs{a}^p\bigr)$ along with the previous estimate yields
    \begin{equation}
    \label{eqn:neg-est-1}
        \begin{split}
            \norm{\gradd q}_{L^p(\Omega)}^p &= \sum_{\substack{\sigma\in\E_{int} \\ \sigma = K\vert L }}\abs{\dsig}\biggl(\frac{\abssig}{\abs{\dsig}}\abs{q_L - q_K}\biggr)^p \\
            &\leq \frac{2^{p-1}}{\mu^p\alpha^p h_\T^p}\biggl(\sum_{\sigma\in\E_{int}}\abs{\dsig}\abs{\mu q_L + (1-\mu)q_K}^p + \sum_{\sigma\in\E_{int}}\abs{\dsig}\abs{q_K}^p\biggr).
        \end{split}
    \end{equation}
    Now, note that 
    \begin{align}
        &\sum_{\sigma\in\E_{int}}\abs{\dsig}\abs{\mu q_L + (1-\mu)q_K}^p\leq \norm{q_\mu}^p_{L^p(\Omega)}, \label{eqn:neg-est-2}\\
        &\sum_{\sigma\in\E_{int}}\abs{\dsig}\abs{q_K}^p \leq\sum_{K\in\T}\biggl(\sum_{\sigma\in\E(K)}\abs{\dsig}\biggr)\abs{q_K}^p \leq 2d\overline{\vartheta}\norm{q}^p_{L^p(\Omega)},\label{eqn:neg-est-3}
    \end{align}
    wherein we have used the mesh regularity assumption \eqref{eqn:mesh-regul} and the fact that the number of edges for a control volume $K\in\T$ is $2d$ since we consider a structured mesh. Utilizing \eqref{eqn:neg-est-2}-\eqref{eqn:neg-est-3} in \eqref{eqn:neg-est-1} yields
    \[
    \norm{\gradd q}^p_{L^p(\Omega)}\leq \frac{2^{p-1}}{\mu^p\alpha^p h_\T^p}(\norm{q_\mu}^p_{L^p(\Omega)} + 2d\overline{\vartheta}\norm{q}^p_{L^p(\Omega)}).
    \]
    Finally, Lemma \ref{lem:stab-recon-op} tells us that $\norm{q_\mu}_{L^p(\Omega)}\leq c\norm{q}_{L^p(\Omega)}$ for some constant $c>0$. Combining this with the above estimate yields the existence of a constant $C_\T>0$ which is dependent on the inverse of the mesh size $h_\T$, the mesh regularity constants $\overline{\vartheta},\alpha$, the constant $\mu$ and the dimension $d$ such that 
    \[
    \norm{\gradd q}_{L^p(\Omega)}\leq C_\T\norm{q}_{L^p(\Omega)},
    \]
    which completes the proof.
\end{proof}

The following lemma gives us the desired control on the norm of the second order pressure.

\begin{lemma}[Control of the Second Order Pressure]
\label{lem:sec-ord-press-bnd}
    Let $\veps>0$ and suppose that the initial data ($\vrho_\epso,\bu_\epso$) are well-prepared in the sense of Definition \ref{def:well-prep}. Let $\pi^n = \sum_{K\in\T}\pi^n_K\X_K$, where $\pi^n_K = (p^n_K-m(p^n))/\veps^2$. Then, for $1\leq n\leq N$, one has: 
    \[
    \norm{\pi^n}\leq C_\T,
    \]
    where $C_\T$ is a constant independent of $\veps$, but not the mesh, and $\norm{\cdot}$ is any discrete norm on $\Lt(\Omega)$.
\end{lemma}
\begin{proof}
    For each $1\leq n\leq N$, the global entropy estimate \eqref{eqn:disc-glob-ent-est} yields
    \begin{equation}
    \label{eqn:pres-ctrl-1}
        \frac{1}{\veps^2}\sum_{K\in\T}\absk\psig(\vrho^n_K)\leq C
    \end{equation}
    where $C >0$ is a constant independent of $\veps$ and $h_\T$. The above equation \eqref{eqn:pres-ctrl-1} in turn implies that $\norm{p^n}_{L^1(\Omega)}\leq C\veps^2$, where $p^n = \sum_{K\in\T}p^n_K\X_K$, since $\psig(z) = \frac{1}{\gamma - 1}p(z)$, cf.\ \eqref{eqn:pres-pot}. Applying Lemma \eqref{lem:neg-grad-est} with $q = p^n$, $\mu = \half$ gives us 
    \begin{equation}
    \label{eqn:pres-ctrl-2}
        \norm{\gradd p^n}_{L^1(\Omega)}\leq C_\T \veps^2.    
    \end{equation}
    Now, since the mesh is fixed, a finite dimensional norm equivalence argument allows us to conclude that the estimate \eqref{eqn:pres-ctrl-2} holds for the $L^2$-norm as well. This, combined with the Sobolev-Poincar\'{e} inequality \eqref{eqn:sob-poinc-ineq} gives us
    \[
        \norm{p^n - m(p^n)}_{L^2(\Omega)}\leq C_\T\veps^2.
    \]
    Dividing both sides by $\veps^2$ gives us an $L^2$-control over the norm of the second order pressure, and hence, a control in any discrete norm by a finite dimensional norm equivalence argument, which completes the proof.
\end{proof}

The next ingredient that is needed is the convergence of $\vrho_{\T,\delt,\veps}\to 1$ as $\veps\to 0$. This, however, is a straightforward consequence of the initial data being well-prepared, the entropy estimate \eqref{eqn:disc-glob-ent-est} and the bounds on $\pig$ \cite[Lemma 2.3]{HLS21}. We recall the result here for the sake of convenience and refer to \cite{AGK23, HLS21} for its proof.

\begin{lemma}
\label{lem:den-conv}
    Let $(\T,\delt)$ be a fixed space time discretization of $\dom$. Suppose that the initial data $\lbrace(\vrho_\epso,\bu_\epso)\rbrace_{\veps>0}$ are well-prepared (in the sense of Definition \ref{def:well-prep}) and let $(\vrho_{\T,\delt,\veps},\bu_{\T,\delt,\veps})$ be the numerical solution generated by the scheme $\pepsh$ for each $\veps>0$. Then 
    \begin{enumerate}
        \item For $\gamma\geq 2$ and $\veps$ sufficiently small, we have
        \[
        \frac{1}{\veps}\norm{\vrho_{\T,\delt,\veps} - 1}_{L^\infty(0,T;L^2(\Omega))}\leq C_\gamma,
        \]
        where $C_\gamma>0$ is a constant independent of $\veps$.
        \item For $1<\gamma<2$, given $\veps$ is sufficiently small and for any $R>2$, we have
        \[
        \frac{1}{\veps}\bigl\lVert(\vrho_{\T,\delt,\veps}-1)\X_{\lbrace\vrho_{\T,\delt,\veps}<R\rbrace}\bigr\rVert_{L^\infty(0,T;L^2(\Omega))} + \veps^{-\frac{2}{\gamma}}\bigl\lVert(\vrho_{\T,\delt,\veps}-1)\X_{\lbrace\vrho_{\T,\delt,\veps}>R\rbrace}\bigr\rVert_{L^\infty(0,T;L^\gamma(\Omega))} \leq C_{\gamma,R},
        \]
        where $C_{\gamma,R}>0$ is a constant independent of $\veps$.

        \item $\vrho_{\T,\delt,\veps}\to 1$ as $\veps\to 0$ in $L^\infty(0,T;L^r(\Omega))$ for any $r\in \lbrack 1,\min\lbrace 2,\gamma\rbrace\rbrack$. Furthermore, for any $\gamma>1$, we have $\vrho_{\T,\delt,\veps}\to 1$ as $\veps\to 0$ in $L^\infty(0,T;L^\gamma(\Omega))$.

        \item For $\veps$ sufficiently small, 
        \[
        \norm{\bu_{\T,\delt,\veps}}_{L^\infty(0,T;L^2(\Omega;\R^d))}\leq C,
        \]
        where $C>0$ is a constant independent of $\veps$.
    \end{enumerate}
\end{lemma}

We now have all the necessary ingredients required to pass to the limit $\veps\to 0$ in the scheme \eqref{eqn:disc-mss-bal}-\eqref{eqn:disc-mom-bal}. We define 
\begin{equation}
\label{eqn:pi-t-delt}
\pi_{\T,\delt,\veps}(t,x) = \sum_{n=0}^{N-1}\sum_{K\in\T}\pi^n_K\X_K(x)\X_{\lbrack t_n,t_{n+1})}(t), 
\end{equation}
with $\pi^n_K$ as defined in Lemma \ref{lem:sec-ord-press-bnd}.

\begin{theorem}[Asymptotic Limit of the Scheme]
\label{thm:asymp-lim-scheme}
    Let $(\T,\delt)$ be a fixed space-time discretization of $\dom$ with mesh size $h_\T$ and let $(\veps_k)_{k\in\N}$ be a sequence of positive numbers such that $\veps_k\to 0$ as $k\to\infty$. Suppose that the initial data $\lbrace(\vrho_{0,\veps_k},\bu_{0,\veps_k})\rbrace_{k\in\N}$ are well-prepared in the sense of Definition \ref{def:well-prep} and let $\lbrace(\vrho^{(k)}_{\T,\delt},\bu^{(k)}_{\T,\delt})\rbrace_{k\in\N}$, $\vrho^{(k)}_{\T,\delt} = \vrho_{\T,\delt,\veps_k}\,,\,\bu^{(k)}_{\T,\delt} = \bu_{\T,\delt,\veps_k}$, be the family of numerical solutions generated by the scheme $\lbrack\Pro^{h_\T}_{\veps_k}\rbrack$. Also, let $\pi^{(k)}_{\T,\delt} = \pi_{\T,\delt,\veps_k}$, cf.\ \eqref{eqn:pi-t-delt}. Then, $\vrho^{(k)}_{\T,\delt}\to 1$  in $L^\infty(0,T;L^\gamma(\Omega))$ as $k\to\infty$ and $(\bu^{(k)}_{\T,\delt}, \pi^{(k)}_{\T,\delt})$ converges to $(\bv_{\T,\delt}, \pi_{\T,\delt})\in L^\infty(0,T;\Lt(\Omega;\R^d)\times\Lt(\Omega))$
    in any discrete norm as $k\to\infty$, where the sequence $\lbrace(\bv^n,\pi^n)\rbrace_{0\leq n\leq N}$ is defined as follows. 

    Given $(\bv^n,\pi^n)\in\Lt(\Omega;\R^d)\times\Lt(\Omega)$, at time $t^n$, $(\bv^{n+1},\pi^{n+1})\in\Lt(\Omega;\R^d)\times\Lt(\Omega)$ is obtained as the solution of the following semi-implicit scheme:
    \begin{subequations}
        \begin{align}
            &(\divt(\bv^n-\delta\bv^{n+1}))_K = 0, \label{eqn:disc-incomp-div-free} \\
            &\frac{1}{\delt}(\bv^{n+1}_K - \bv^n_K) + (\divup(\bv^{n}, \bv^{n} - \delta\bv^{n+1}))_K + (\gradt\pi^{n+1})_K = 0, \label{eqn:disc-incomp-mom}
        \end{align}
    \end{subequations}
    with the correction term $\delta\bv^{n+1}_K$ being defined as 
    \begin{equation}
    \label{eqn:correc-term-incomp}
        \delta\bv^{n+1}_K = \eta\delt(\gradt\pi^{n+1})_K \text{ for each }K\in\T.
    \end{equation}
\end{theorem}

\begin{proof}
    In Lemma \ref{lem:den-conv}, the convergence of $\vrho^{(k)}_{\T,\delt}\to 1$ in $L^\infty(0,T;L^\gamma(\Omega))$ for $k\to\infty$ ($\veps_k\to 0)$ was presented. Further, given that both $\lbrace\bu^{(k)}_{\T,\delt}\rbrace_{k\in\N}$ and $\lbrace\pi^{(k)}_{\T,\delt}\rbrace_{k\in\N}$ are uniformly bounded sequences respectively by Lemmas \ref{lem:den-conv} and \ref{lem:sec-ord-press-bnd}, we can find subsequences that converge in any discrete norm to some $(\bv_{\T,\delt}, \pi_{\T,\delt})\in\Lt(\Omega;\R^d)\times\Lt(\Omega)$. Also, we have $\frac{1}{\veps_k^2}\gradt (p^{(k)})^{n+1} = \gradt (\pi^{(k)})^{n+1}$ for each $n = 0,\dots, N-1$ and further noting that as the mesh is fixed, \eqref{eqn:disc-glob-ent-est} gives us an $L^\infty$-estimate on $\gradt(\pi^{(k)})^{n+1}$ using a finite dimensional norm equivalence argument. Since $\pi^{(k)}_{\T,\delt}$ has zero mean, we can conclude that there exists a subsequence such that $\gradt(\pi^{(k)})^{n+1}\to\gradt\pi^{n+1}$ in any discrete norm, cf.\ \cite{Cia13}, and as a consequence, $(\delta u^{(k)})^{n+1}_K = \frac{\eta\delt}{\veps_k^2}(\gradt (p^{(k)})^{n+1})_K \to \eta\delt(\gradt\pi^{n+1})_K$. Passing to the limit cell-by-cell in the scheme \eqref{eqn:disc-mss-bal}-\eqref{eqn:disc-mom-bal}, we obtain that $(\bv_{\T,\delt},\pi_{\T,\delt})$ is a solution of \eqref{eqn:disc-incomp-div-free}-\eqref{eqn:disc-incomp-mom}, with the initial data now being given by $\bv^0_K = (\Pi_\T\bv_0)_K$ as we have $\bu_{0,\veps_k}\to\bv_0$ due to the assumption on initial data.
\end{proof}

We label by $\pepsho$, the scheme \eqref{eqn:disc-incomp-div-free}-\eqref{eqn:disc-incomp-mom}. We can infer from Theorem \ref{thm:asymp-lim-scheme} that $\lim_{\veps\to 0}\pepsh = \pepsho$. The scheme $\pepsho$ can be viewed as a finite volume approximation of the velocity stabilized incompressible Euler equations:
\begin{subequations}
    \begin{align}
        &\Div(\bv-\delta\bv) = 0, \label{eqn:vel-stab-div-free} \\
        &\Dt\bv + \Div(\bv\otimes(\bv-\delta\bv)) + \grad\pi = 0, \label{eqn:vel-stab-incomp-mom}
    \end{align}
\end{subequations}
where $\delta\bv = \eta\grad\pi$, $\eta>0$ a constant.

\begin{remark}
    Note that as $\eta>\frac{1}{\rk^{n+1}}$, Hypothesis \ref{hyp:den-bound} tells us that $\eta$ is bounded below by $\frac{1}{\overline{\vrho}}$. Furthermore, we also note that $\frac{1}{\eta} - 1 <\rk^{n+1} - 1$ and the right-hand side goes to $0$ as $\veps\to 0$, cf.\ Lemma \ref{lem:den-conv}. As a consequence, one obtains $\eta > 1$ for the limit scheme \eqref{eqn:disc-incomp-div-free}-\eqref{eqn:disc-incomp-mom}, which is also reflected in the stability conditions imposed to obtain the energy stability of the limit scheme in Theorem \ref{thm:disc-loc-ent-incomp}.
\end{remark}

\begin{remark}
The existence of a discrete solution to the scheme \eqref{eqn:disc-incomp-div-free}-\eqref{eqn:disc-incomp-mom} follows from the fact that the scheme is obtained after passing to the limit cell-by-cell in the scheme \eqref{eqn:disc-mss-bal}-\eqref{eqn:disc-mom-bal}, for which the existence of a solution is given by Theorem \ref{thm:exist-num-sol}.
\end{remark}

We now want to rigorously show that the scheme $\pepsho$ is indeed energy stable and consistent with the weak formulation of $\pepso$. This will enable us to pass to the limit $h_\T\to 0$, as we proceed to show that the numerical solutions generated by $\pepsho$ will converge to a DMV solution of the incompressible Euler system $\pepso$ as $h_\T\to 0$. To this end, we begin by proving a local in time energy balance. As before, we stick to the notation $\bw^n = \bv^n - \delta\bv^{n+1}$ and $w^n_{\sk} = v^n_{\sk}-\delta v^{n+1}_{\sk}$ with the positive and negative halves of $w^n_{\sk}$ being defined analogously as in \eqref{eqn:pos-neg-stab-velo}.

\begin{theorem}[Local In-Time Energy Balance]
\label{thm:disc-loc-ent-incomp}
    Suppose that the following conditions hold:
    \begin{enumerate}
        \item $\eta>1$;
        \item $\delt\leq\dfrac{\absk}{2\displaystyle\sum_{\substack{\sigma\in\E(K) \\ \sigma = K\vert L}}\abssig(-(w^n_{\sk})^{-}))}$.
    \end{enumerate}
    Then, any numerical solution to $\pepsho$ satisfies the following local in-time energy inequality
    \begin{equation}
    \label{eqn:disc-loc-ent-incomp}
        \half\sum_{K\in\T}\absk\lvert\bv^{n+1}_K\rvert^2\leq \half\sum_{K\in\T}\absk\lvert\bv^n_K\rvert^2
    \end{equation}
    for each $0\leq n\leq N-1$.
\end{theorem}

\begin{proof}
    We take the dot product of the momentum balance \eqref{eqn:disc-incomp-mom} with $\absk\bv^n_K$ and sum over all $K\in\T$ to get $\sum_{K\in\T}\absk(A^n_K+B^n_K+C^n_K) = 0$, where 
    \begin{align*}
        &A^n_K = \frac{1}{\delt}(\bv^{n+1}_K - \bv^n_K)\cdot\bv^n_K, \\
        &B^n_K = (\divup(\bv^n,\bw^n))_K\cdot\bv^n_K, \\
        &C^n_K = (\gradt\pi^{n+1})_K\cdot\bv^n_K.
    \end{align*}
    Recalling the identity $(a-b)\cdot b = (\abs{a}^2 - \abs{b}^2-\abs{a-b}^2)/2$, we can express $A^n_K$ as
    \begin{equation}
    \label{eqn:A}
        A^n_K = \frac{1}{\delt}\biggl(\half\lvert\bv^{n+1}_K\rvert^2 - \half\lvert\bv^n_K\rvert^2\biggr) - \frac{\lvert\bv^{n+1}_K-\bv^n_K\rvert^2}{2\delt}.
    \end{equation}
    Also, as a consequence of \eqref{eqn:disc-incomp-div-free}, one has 
    \[
    (\divup(\bv^n,\bw^n))_K = \frac{1}{\absk}\sum_{\substack{\sink \\ \sigma = K\vert L}}\abssig(\bv^n_L - \bv^n_K)(w^n_{\sk})^{-}.
    \]
    In the light of the above, one immediately obtains
    \begin{equation}
    \label{eqn:B}
        B^n_K = (\divup(\lvert \bv^n\rvert^2/2,\bw^n))_K - \frac{1}{\absk}\sum_{\substack{\sink \\ \sigma = K\vert L}}\abssig\frac{\lvert \bv^n_L - \bv^n_K\rvert^2}{2}(w^n_{\sk})^{-}.
    \end{equation}
    Now, writing $\bv^n_K = \bw^n_K + \delta\bv^{n+1}_K$ and using the grad-div duality \eqref{eqn:disc-grad-div-primal} along with \eqref{eqn:disc-incomp-div-free} and \eqref{eqn:correc-term-incomp}, we get 
    \begin{equation}
    \label{eqn:C}
        \sum_{K\in\T}\absk C^n_K = \sum_{K\in\T} \absk(\gradt\pi^{n+1})_K\cdot\delta\bv^{n+1}_K = \eta\delt\sum_{K\in\T}\absk \abs{(\gradt \pi^{n+1})_K}^2.
    \end{equation}
    Using the conservativity of the mass flux, we note that $\sum_{K\in\T}\absk(\divup(\lvert \bv^n\rvert^2/2,\bw^n))_K = 0$. Thus, using \eqref{eqn:A}-\eqref{eqn:C}, we get
    \begin{equation}
    \label{eqn:1-loc-ent}
        \sum_{K\in\T}\frac{\absk}{\delt}\biggl(\half\lvert\bv^{n+1}_K\rvert^2 - \half\lvert\bv^n_K\rvert^2\biggr)+ \eta\delt\sum_{K\in\T}\absk \abs{(\gradt \pi^{n+1})_K}^2 = \sum_{K\in\T}\absk R^n_{K,\delt},
    \end{equation}
    where 
    \[
    R^n_{K,\delt} = \frac{\lvert\bv^{n+1}_K-\bv^n_K\rvert^2}{2\delt} +  \frac{1}{\absk}\sum_{\substack{\sink \\ \sigma = K\vert L}}\abssig\frac{\lvert \bv^n_L - \bv^n_K\rvert^2}{2}(w^n_{\sk})^{-}.
    \]
    In order to estimate the above remainder term, we use the momentum balance \eqref{eqn:disc-incomp-mom} and the identity $\abs{a+b}^2\leq 2(\abs{a}^2+\abs{b}^2)$ to get,  
    \begin{equation}
    \label{eqn:2-loc-ent}
        \begin{split}
            \frac{1}{2}\abs{\frac{\bv^{n+1}_K - \bv^n_K}{\delt}}^2 &\leq \biggl(\frac{1}{\absk}\sum_{\substack{\sink \\ \sigma = K\vert L}}\abssig(\bv^n_L - \bv^n_K)(w^n_{\sk})^{-}\biggr)^2 + \abs{(\gradt\pi^{n+1})_K}^2 \\
            &\leq \biggl(\sum_{\substack{\sink \\ \sigma = K\vert L}}\frac{\abssig}{\absk}\abs{\bv^n_L - \bv^n_K}^2(w^n_{\sk})^{-}\biggr)\biggl(\sum_{\substack{\sink \\ \sigma = K\vert L}}\frac{\abssig}{\absk}(w^n_{\sk})^{-}\biggr) + \abs{(\gradt\pi^{n+1})_K}^2,
        \end{split}
    \end{equation}
    where we have used the Cauchy-Schwarz inequality after writing $\abssig(w^n_{\sk})^{-}$ as
    \[
    \abssig(w^n_{\sk})^{-} = \sqrt{-\abssig(w^n_{\sk})^{-}}\,\sqrt{-\abssig(w^n_{\sk})^{-}}.
    \]
    Finally, we can estimate $R^n_{K,\delt}$ 
    \begin{equation}
    \label{eqn:3-loc-ent}
        R^n_{K,\delt}\leq\biggl(\sum_{\substack{\sink \\ \sigma = K\vert L}}\frac{\abssig}{\absk}\abs{\bv^n_L - \bv^n_K}^2(w^n_{\sk})^{-}\biggr)\biggr(\half + \delt\sum_{\substack{\sink \\ \sigma = K\vert L}}\frac{\abssig}{\absk}(w^n_{\sk})^{-}\biggl) + \delt\abs{(\gradt\pi^{n+1})_K}^2.
    \end{equation}
    Combining \eqref{eqn:1-loc-ent} and \eqref{eqn:3-loc-ent} yields 
    \begin{equation}
    \label{eqn:4-loc-ent}
        \begin{split}
            \sum_{K\in\T}\frac{\absk}{\delt}\biggl(\half\lvert\bv^{n+1}_K\rvert^2 - &\half\lvert\bv^n_K\rvert^2\biggr) + (\eta - 1)\delt\sum_{K\in\T}\absk\abs{(\gradt\pi^{n+1})_K}^2 \\
            &\leq\sum_{K\in\T}\absk\biggl(\sum_{\substack{\sink \\ \sigma = K\vert L}}\frac{\abssig}{\absk}\abs{\bv^n_L - \bv^n_K}^2(w^n_{\sk})^{-}\biggr)\biggl(\half + \delt\sum_{\substack{\sink \\ \sigma = K\vert L}}\frac{\abssig}{\absk}(w^n_{\sk})^{-}\biggr).
        \end{split}
    \end{equation}
    Under the assumed conditions (1) and (2), we observe that the second term on the left-hand side remains non-negative while the right-hand side is non-positive. This gives us the desired inequality \eqref{eqn:disc-loc-ent-incomp}.
\end{proof}

The global energy inequality now follows due to \eqref{eqn:4-loc-ent}. 

\begin{theorem}[Global Energy Estimate]
\label{thm:disc-glob-ent-incomp}
    Assume that the hypotheses of Theorem \ref{thm:disc-loc-ent-incomp} are satisfied. Then, there exists $C>0$ such that for any $1\leq n\leq N-1$, 
    \begin{equation}
    \label{eqn:disc-glob-ent-incomp}
        \sum_{K\in\T}\absk\half\lvert\bv^n_K\rvert^2+(\eta-1)\sum_{r=0}^{n-1}\delt^2\sum_{K\in\T}\absk\abs{(\gradt\pi^{r+1})_K}^2\leq C.
    \end{equation}
\end{theorem}
\begin{proof}
    \eqref{eqn:4-loc-ent} guarantees that 
    \[
    \sum_{K\in\T}\frac{\absk}{\delt}\biggl(\half\lvert\bv^{r+1}_K\rvert^2 - \half\lvert\bv^r_K\rvert^2\biggr) + (\eta - 1)\delt\sum_{K\in\T}\absk\abs{(\gradt\pi^{r+1})_K}^2\leq 0.
    \]
    Multiplying by $\delt$ and summing over $r$ from $0$ to $n-1$, and noting that as $\bv_0\in H^k(\Omega;\R^d)$, $\norm{\Pi_\T\bv_0}_{L^2(\Omega;\R^d)}$ is bounded independent of the mesh parameters as a consequence of \eqref{eqn:stab-proj-op}, one immediately gets \eqref{eqn:disc-glob-ent-incomp}.
\end{proof}

\begin{remark}
\label{rem:incomp-vel-est}
    Analogous to the compressible case, as a consequence of \eqref{eqn:disc-glob-ent-incomp} we obtain $\bv_{\T,\delt}\in L^\infty(0,T;L^2(\Omega;\R^d)$ and $\sqrt{\delt}\gradt\pi_{\T,\delt}\in L^2(\odom;\R^d)$ where 
    \begin{align}
        &\bv_{\T,\delt}(t,x) = \sum_{n=0}^{N-1}\sum_{K\in\T}\bv^n_K\X_K(x)\X_{[t_n,t_{n+1})}(t) \label{eqn:v-t-delt}\\
        &\pi_{\T,\delt}(t,x) = \sum_{n=0}^{N-1}\sum_{K\in\T}\pi^n_K\X_K(x)\X_{[t_n,t_{n+1})}(t) \label{eqn:pi-t-delt-incomp}
    \end{align}
\end{remark}

\begin{remark}
    One can observe that the stability conditions imposed by Theorem \ref{thm:disc-loc-ent-ineq} and Theorem \ref{thm:disc-loc-ent-incomp} also respect the zero Mach number limit $\veps\to 0$. Indeed, if one simply sets the terms involving density to be equal to 1 in the conditions imposed by Theorem \ref{thm:disc-loc-ent-ineq}, one obtains the stability conditions that we impose in Theorem \ref{thm:disc-loc-ent-incomp}.
\end{remark}

Next, we  turn our attention towards proving the consistency of the scheme $\pepsho$.

\begin{theorem}[Consistency of the Incompressible Limit Scheme]
\label{thm:cons-incomp-scheme}
    Let $(\T,\delt)$ be a given space-time discretization of $\dom$. Let $\bv_{\T,\delt}\in\Lt(\Omega;\R^d)$ defined as in \eqref{eqn:v-t-delt}
    be the numerical solution generated by the scheme $\pepsho$ with initial data $\bv^0_K = (\Pi_\T\bv_0)_K$. Suppose that the hypotheses of Theorem \ref{thm:disc-loc-ent-incomp} hold. Further, suppose that there exists a constant $\overline{\pi}>0$ such that
    \[
    \abs{\pi_{\T,\delt}}\leq\overline{\pi}.
    \]
    Then, the scheme $\pepsho$ is consistent with the weak formulation of $\pepso$, i.e. 
    \begin{equation}
    \label{eqn:div-free-cons}
        \int_{\Omega}\bv_{\T,\delt}(t,\cdot)\cdot\grad\varphi\,\dx + \mathcal{Q}^1_{\T,\delt} = 0,
    \end{equation}
    for any $\varphi\in\Cinf(\Omega)$ and a.e. $t\in (0,T)$;
    \begin{equation}
    \label{eqn:incomp-mom-cons}
        -\int_\Omega\bv_{\T,\delt}(0,\cdot)\cdot\uphi(0,\cdot)\,\dx = \int_0^T\int_\Omega\lbrack\bv_{\T,\delt}\cdot\Dt\uphi+(\bv_{\D,\delt}\otimes\bv_{\D,\delt})\colon\grad\uphi\rbrack\,\dx\,\dt + \mathcal{Q}^2_{\T,\delt},
    \end{equation}
    for any $\uphi\in\Cinf(\lbrack 0,T)\times\Omega;\R^d)$ with $\Div\uphi = 0$.

    Here $\bv_{\D,\delt}$ denotes a reconstruction of $\bv_{\T,\delt}$ on the dual mesh and the consistency errors $\mathcal{Q}^1_{\T,\delt}$ and $\mathcal{Q}^2_{\T,\delt}$ are such that 
    \begin{align*}
        &\abs{\mathcal{Q}^1_{\T,\delt}} = o(h_\T,\sqrt{\delt})\to 0\text{ as }h_\T\to 0, \\
        &\abs{\mathcal{Q}^2_{\T,\delt}} = o(h_\T,\sqrt{\delt})\to 0\text{ as }h_\T\to 0.
    \end{align*}
\end{theorem}

\begin{proof}
    The proof of the consistency of the momentum balance \eqref{eqn:disc-incomp-mom} is exactly the same as the one for its compressible counterpart given in Theorem \ref{thm:cons-scheme}. The only difference here is that gradient term simply vanishes as we have a divergence free test function $\uphi\in\Cinf(\lbrack 0,T)\times\Omega;\R^d)$ and we assume that $\pi_{\T,\delt}$ remains bounded.
    
    Thus, we only prove the consistency of the discrete divergence free condition \eqref{eqn:disc-incomp-div-free}. Let $\varphi\in\Cinf(\Omega)$. Multiplying \eqref{eqn:disc-incomp-div-free} by $\absk\varphi_K$, $\varphi_K = (\Pi_\T\varphi)_K$, and summing over all $K\in\T$, we get
    \[
    \sum_{K\in\T}\absk\varphi_K(\divt\bv^n)_K - \sum_{K\in\T}\absk\varphi_K(\divt\delta\bv^{n+1})_K = 0.
    \]
    We can use the discrete grad-div duality \eqref{eqn:disc-grad-div-primal} to rewrite the first term as 
    \begin{equation}
    \label{eqn:cons-1}
        \sum_{K\in\T}\absk\varphi_K(\divt\bv^n)_K = -\sum_{K\in\T}\absk(\gradt\Pi_\T\varphi)_K\cdot\bv^n_K.
    \end{equation}
    Similarly, we can simplify the second term using the grad-div duality along with \eqref{eqn:correc-term-incomp} to obtain
    \begin{equation}
    \label{eqn:cons-2}
        \sum_{K\in\T}\absk\varphi_K(\divt\delta\bv^{n+1})_K = -\eta\delt\sum_{K\in\T}\absk(\gradt\Pi_\T\varphi)_K\cdot(\gradt\pi^{n+1})_K.
    \end{equation}
    Equations \eqref{eqn:cons-1}-\eqref{eqn:cons-2} together yield 
    \[
    \int_\Omega\bv_{\T,\delt}(t,\cdot)\cdot\grad\varphi\,\dx + \mathcal{Q}^1_{\T,\delt} = 0,
    \]
    where 
    \begin{equation}
    \label{eqn:q1}
        \mathcal{Q}^1_{\T,\delt} = \int_\Omega\bv_{\T,\delt}(t,\cdot)\cdot(\gradt\Pi_\T\varphi - \grad\varphi)\,\dx - \eta\delt\int_\Omega\gradt\Pi_\T\varphi\cdot\gradt\pi_{\T,\delt}(t+\delt,\cdot)\,\dx = Q_1 + Q_2
    \end{equation}
    $Q_1$ and $Q_2$ can be approximated using the estimates provided by \eqref{eqn:op-est} and Remark \ref{rem:incomp-vel-est} as follows:
    \begin{align*}
        \abs{Q_1}&\leq \norm{\bv_{\T,\delt}(t,\cdot)}_{L^2}\norm{\gradt\Pi_\T\varphi - \grad\varphi}_{L^2}\lesssim h_\T, \\
        \abs{Q_2}&\leq \eta\sqrt{\delt}\norm{\gradt\Pi_\T\varphi - \grad\varphi}_{L^2}\norm{\sqrt{\delt}\gradt\pi_{\T,\delt}(t+\delt,\cdot)}_{L^2} + \eta\sqrt{\delt}\norm{\grad\varphi}_{L^2}\norm{\sqrt{\delt}\gradt\pi_{\T,\delt}(t+\delt,\cdot)}_{L^2} \\
        &\lesssim h_\T\sqrt{\delt}+\sqrt{\delt} = o(h_\T,\sqrt{\delt}).
    \end{align*}
    The above yields $\mathcal{Q}^1_{\T,\delt} = o(h_\T,\sqrt{\delt})$ and this completes the proof.
\end{proof}

We can now prove that the sequence of numerical solutions generated by $\pepsho$ indeed converge to a DMV solution of $\pepso$. 

\begin{theorem}[Weak Convergence (Incompressible Case)]
\label{thm:scheme-wk-con-incomp}
    Let $\lbrace(\T^\m,\delt^\m)\rbrace_{m\in\N}$ be a sequence of space time discretizations such that $h^\m = h_{\T^\m}\to 0$ as $m\to\infty$. Let $\bv^\m = \bv_{\T^\m,\delt^\m}$, be the numerical solution generated by the scheme $\lbrack\Pro^{h^\m}_0\rbrack$ for each $m\in\N$. Assume that the hypotheses of Theorem \ref{thm:disc-loc-ent-incomp} hold. Then, 
    \[
    \bv^\m\weakstar\langle\U;\tbv\rangle\text{ in }L^\infty(0,T;L^2(\Omega;\R^d))\text{ as }m\to\infty,
    \]
    where $\mcu = \lbrace\U\rbrace_{(t,x)\in\odom}$ is a DMV solution of $\pepso$ emanating from $\bv_0$ with concentration defects 
    \begin{align*}
        &\mathfrak{D}_{cd}(t) = \overline{\half\abs{\bv}^2}(t) - \half\langle\U,\abs{\tbv}^2\rangle, \\
        &\mathfrak{M}_{cd}(t) = \overline{\bv\otimes\bv}(t) - \langle\U, \tbv\otimes\tbv\rangle,
    \end{align*}
    where 
    \begin{align*}
        &\half\lvert\bv^\m\rvert^2\weakstar \overline{\half\abs{\bv}^2}\text{ in }L^\infty(0,T;\M^{+}(\overline{\Omega})), \\
        &\bv^\m\otimes\bv^\m\weakstar \overline{\bv\otimes\bv}\text{ in }L^\infty\bigl(0,T;\M\bigl(\overline{\Omega};\R^{d\times d}\bigr)\bigr)
    \end{align*}
    as $m\to\infty$.
\end{theorem}

\begin{proof}
    Because of the energy estimate \eqref{eqn:disc-glob-ent-incomp}, we get uniform bounds for the following quantities in the corresponding spaces.
    \begin{equation}
    \label{eqn:fam-est}
        \begin{split}
            &(\bv^\m)_{m\in\N}\subset L^\infty(0,T;L^2(\Omega;\R^d)), \biggl(\half\lvert\bv^\m\rvert^2\biggr)_{m\in\N}\subset L^\infty(0,T;L^1(\Omega)), \\
            &(\bv^\m\otimes\bv^\m)_{m\in\N}\subset L^\infty(0,T;L^1(\Omega;\R^{d\times d})).
        \end{split}
    \end{equation}
    The Fundamental Theorem of Young measures, cf.\ \cite{JBal89, Ped97}, asserts the existence of a parameterized family of probability measures $\mcu = \lbrace\U\rbrace_{(t,x)\in\odom}\in L^\infty(\odom;\Pro(\F_{incomp}))$ such that 
    \[
    \bv^\m \weakstar \langle\U; \tbv\rangle\text{ in }L^\infty(0,T;L^2(\Omega;\R^d))\text{ as }m\to\infty.
    \]
    As a consequence of the uniform in time bounds given in \eqref{eqn:fam-est}, we infer the existence of $\overline{\half\abs{\bv}^2}\in L^\infty(0,T;\M^{+}(\overline{\Omega}))$ and $\overline{\bv\otimes\bv}\in L^\infty\bigl(0,T;\M\bigl(\overline{\Omega};\R^{d\times d}\bigr)\bigr)$ such that 
    \begin{align*}
        &\half\lvert\bv^\m\rvert^2\weakstar \overline{\half\abs{\bv}^2}\text{ in }L^\infty(0,T;\M^{+}(\overline{\Omega})), \\
        &\bv^\m\otimes\bv^\m\weakstar \overline{\bv\otimes\bv}\text{ in }L^\infty\bigl(0,T;\M\bigl(\overline{\Omega};\R^{d\times d}\bigr)\bigr).
    \end{align*}
    Consequently, the concentration defects take the desired form. We can then pass to the limit $h^\m\to 0$ in the consistency formulation \eqref{eqn:div-free-cons}-\eqref{eqn:incomp-mom-cons}, after noting that the reconstructions will also converge to the same limit (cf.\ Lemma \ref{lem:wk-cong-mesh-recon}), to get
    \begin{equation}
        \label{eqn:div-free-dmv-2}
            \int_\Omega\langle\U;\tbv\rangle\cdot\grad\varphi\,\dx = 0
        \end{equation}
        holds for a.e.\ $t\in (0,T)$ and $\varphi\in\Cinf(\Omega)$;
        \begin{equation}
        \label{eqn:incomp-mom-dmv-2}
            \int_0^T\int_\Omega\langle\lbrack\U; \tbv\rangle\cdot\Dt\uphi + \langle\U; \tbv\otimes\tbv\rangle\colon\grad\uphi\rbrack\,\dx\,\dt + \int_0^T\int_\Omega\grad\uphi\colon\mathrm{d}\mathfrak{M}_{cd}(t)\,\dt + \int_\Omega\bv_0\cdot\uphi(0,\cdot)\,\dx = 0
        \end{equation}
        holds for all $\uphi\in\Cinf(\lbrack 0,T)\times\Omega;\R^d)$ with $\Div\uphi = 0$.

        Finally, the local in-time energy estimate \eqref{eqn:disc-loc-ent-incomp} tells us that for a.e. $t\in (0,T)$,
        \[
        \int_\Omega\half\lvert\bv^\m(t,\cdot)\rvert^2\,\dx\leq \int_\Omega\half\lvert\bv^\m(0,\cdot)\rvert^2\,\dx.
        \]
        Passing to the limit $m\to\infty$, we obtain 
        \begin{equation}
        \label{eqn:enrg-ineq-incomp-2}
            \half\int_{\Omega}\langle\U; \abs{\tbv}^2\rangle\,\dx + \int_{\Omega}\mathrm{d}\mathfrak{D}_{cd}(t) \leq \half\int_{\Omega}\abs{\bv_0}^2\,\dx. 
        \end{equation}
        
        The defect compatibility condition is a direct consequence of \cite[Lemma 2.3]{BF18a}, \cite[Lemma 2.1]{FGS+16}. This proves that $\mcu$ is indeed a DMV solution of the incompressible Euler system $\pepso$ which completes the proof.        
\end{proof}

Thus, the final result obtained is $\lim_{h_\T\to 0}\biggl(\lim_{\veps\to 0}\pepsh\biggr) =$ DMV solution of the incompressible Euler system.

\section{AP property of the scheme}
\label{sec:AP-prop-scheme}

We can now finally prove the main result of this paper, that is the AP property of the present scheme \eqref{eqn:disc-mss-bal}-\eqref{eqn:disc-mom-bal}. We deduce this from the weak-strong uniqueness principle for the incompressible Euler system, cf.\ Theorem \ref{thm:wk-str-incomp}.

\begin{theorem}[AP property of the scheme]
\label{thm:comm-lim}
    Consider the finite volume scheme $\pepsh$, approximating the parameterized barotropic Euler equations $\peps$. Assume that the initial data $\lbrace(\vrho_\epso, \bu_\epso)\rbrace_{\veps>0}$ are well-prepared in the sense of Definition \ref{def:well-prep}, where $\bv_0\in H^k(\Omega;\R^d)$, $k>d/2+1$ with $\Div\bv_0 = 0$ and $T<T_{max}$. In addition, suppose that the hypotheses of Theorem \ref{thm:disc-loc-ent-ineq} and Theorem \ref{thm:disc-loc-ent-incomp} hold. Then, the limits $h_\T\to 0$ and $\veps\to 0$ commute, in the sense that the numerical solutions generated by the scheme $\pepsh$ will converge to a classical solution $\bv$ of the incompressible Euler system $\pepso$, irrespective of the order in which the limits are performed.
\end{theorem}

\begin{proof}
    The proof follows from the analysis presented in Section \ref{sec:lim-analys-scheme}. Performing the limit $h_\T\to 0$ first results in a family of DMV solutions $\lbrace\mcv^\veps\rbrace_{\veps>0}$ to the barotropic Euler system $\peps$ in accordance with Theorem \ref{thm:scheme-wk-con}. Taking the zero Mach limit $\veps\to 0$ subsequently results in the convergence of the DMV solutions $\lbrace\mcv^\veps\rbrace_{\veps>0}$ to a classical solution $\bv$ of the incompressible Euler system $\pepso$ with initial data $\bv_0$ as a consequence of Theorem \ref{thm:asymp-lim-dmv}, cf.\ Theorem \ref{thm:eps-h-lim}.

    On the other hand, first performing the zero Mach limit $\veps\to 0$ yields a velocity stabilized scheme $\pepsho$, approximating the incompressible Euler equations $\pepso$, as given in Theorem \ref{thm:asymp-lim-scheme}. Letting $h_\T\to 0$ thereafter generates a DMV solution $\mcu$ of the incompressible Euler equations, emanating from the initial data $\bv_0$ as given in Theorem \ref{thm:scheme-wk-con-incomp}. 

    In summary, we have a DMV solution and a classical solution of the incompressible Euler system $\pepso$ emanating from the same initial data $\bv_0$. The weak-strong uniqueness principle, cf.\ Theorem \ref{thm:wk-str-incomp}, then asserts that the two solutions necessarily coincide. Therefore, the order in which the limits are performed is irrelevant and hence, the limits commute.
\end{proof}

\section{$\mathcal{K}$-convergence}
\label{sec:K-con}
The convergence results of the numerical solutions obtained in Theorem \ref{thm:scheme-wk-con} and Theorem \ref{thm:scheme-wk-con-incomp} are only in the weak sense. In order to visualize this limit, we apply the concept of $\K$-convergence that turns weakly convergent sequences to strongly convergenct sequences of their C\'{e}saro averages; see \cite{FLM+21a, FLM+21b, FLM20a, Luk20}. Using these techniques, we arrive at the following results in an analogous manner as given in \cite{FLM+21a}. 

\begin{theorem}[$\K$-convergence (compressible case)]
\label{thm:kcon-comp}
    Let the assumptions of Theorem \ref{thm:scheme-wk-con} hold. Then, we obtain the following convergences (passing to another subsequence if needed). 
     \begin{enumerate}[(i)]

        \item Strong convergence of the Ces\`{a}ro averages
        \[
        \begin{split}
        &\frac{1}{N}\sum_{m=1}^N\vrho^\m\to\langle\V;\tvrho\rangle\text{ in }L^q(\odom), \\
        &\frac{1}{N}\sum_{m=1}^N\bm^\m\to\langle\V;\tbm\rangle\text{ in }L^s(\odom;\R^d)
        \end{split}
        \]
        as $N\to\infty$, for any $1\leq q<\infty$ and $1\leq s\leq 2$;
        
        \item $L^q$-convergence of the s-Wasserstein distance  
        \[
        \norm{W_s\Biggl\lbrack\frac{1}{N}\sum_{m=1}^N\delta_{(\vrho^\m(t,x),\textbf{m}^\m(t,x))};\V\Biggr\rbrack}_{L^q(\odom)}\longrightarrow 0
        \]
        as $N\to\infty$, for any $1\leq q\leq s<2$;

        \item $L^1$-convergence of the deviations or the first variance
        \[
        \begin{split}
             &\frac{1}{N}\sum_{m=1}^N\abs{\vrho^\m-\frac{1}{N}\sum_{m=1}^N\vrho^\m}\longrightarrow \langle\V; \abs{\tvrho - \langle\V;\tvrho\rangle}\rangle \text{ in }L^1(\odom), \\
             &\frac{1}{N}\sum_{m=1}^N\abs{\bm^\m-\frac{1}{N}\sum_{m=1}^N\bm^\m}\longrightarrow \langle\V; \abs{\tbm - \langle\V;\tbm\rangle}\rangle\text{ in }L^1(\odom;\R^d)
        \end{split}
        \]
        as $N\to\infty$.
    \end{enumerate}
\end{theorem}

\begin{theorem}[$\K$-convergence (incompressible case)]
\label{thm:kcon-incomp}
    Let the assumptions of Theorem \ref{thm:scheme-wk-con-incomp} hold. Then, we obtain the following convergences (passing to another subsequence if needed).
     \begin{enumerate}[(i)]
        \item Strong convergence of the Ces\`{a}ro average
        \[
        \frac{1}{N}\sum_{m=1}^N\bv^\m\to\langle\U;\tbv\rangle\text{ in }L^s(\odom;\R^d)
        \]
        as $N\to\infty$, for any $1\leq q<\infty$ and $1\leq s\leq 2$;
        
        \item $L^q$-convergence of the s-Wasserstein distance  
        \[
        \norm{W_s\Biggl\lbrack\frac{1}{N}\sum_{m=1}^N\delta_{\textbf{v}^\m(t,x)};\U\Biggr\rbrack}_{L^q(\odom)}\longrightarrow 0
        \]
        as $N\to\infty$, for any $1\leq q\leq s<2$;

        \item $L^1$-convergence of the deviations or the first variance
        \[
        \frac{1}{N}\sum_{m=1}^N\abs{\bv^\m-\frac{1}{N}\sum_{m=1}^N\bv^\m}\longrightarrow \langle\U; \abs{\tbv - \langle\U;\tbv\rangle}\rangle\text{ in }L^1(\odom;\R^d)
        \]
        as $N\to\infty$.
    \end{enumerate}
\end{theorem}

\section{Numerical Results}
\label{sec:num-exp}

In this section, we report the results of the numerical case study conducted. For the scheme $\pepsh$, in accordance with Theorem \ref{eqn:disc-loc-ent-ineq}, note that the timestep $\delt$ needs to satisfy 
\begin{equation}
\label{eqn:time-step-comp}
\frac{\delt}{\absk}\sum_{\substack{\sink \\ \sigma =  K\vert L}}\abssig\frac{\vrho^{n+1}_L(-(\wsk^n)^{-})}{\vrho^{n+1}_K}\leq\half
\end{equation}
and $\eta>(\rk^{n+1})^{-1}$. Both conditions are implicit and the former involves the flux as well. Consequently, they are difficult to implement in practice. However, we note that if the following implicit restriction
\[
\frac{\delt}{\absk}\sum_{\substack{\sink \\ \sigma =  K\vert L}}\abssig\frac{\abs{\flx(\vrho^{n+1}, \bw^n)}}{\vrho^{n+1}_K}\leq\half,
\]
holds, then
\[
\frac{3}{2}\rk^{n+1}-\rk^n \geq \rk^{n+1}-\rk^n+\frac{\delt}{\absk}\sum_{\substack{\sink \\ \sigma = K\vert L}}\abssig\abs{\flx(\vrho^{n+1}, \bw^n)}\geq 0.
\]
Therefore, we get that 
\[
\frac{3}{2\rk^n} \geq \frac{1}{\rk^{n+1}}
\]
and as a consequence, choosing $\eta>\dfrac{3}{2\rk^n}$ satisfies the necessary condition.

We now derive a sufficient timestep restriction which is easy to implement in practice for the scheme $\pepsh$. 

\begin{proposition}
\label{prop:suff-time-step-comp}
    For each $\sigma = K\vert L\in\E(K)$, suppose the following holds:
    \begin{equation}
    \label{eqn:suff-time-step-comp}
        \delt\max\biggl\lbrace\frac{\abs{\partial K}}{\abs{K}},\frac{\abs{\partial L}}{\abs{L}}\biggr\rbrace\biggl\lbrace\abs{\ldblbrace\bu^n\rdblbrace_\sigma} +\sqrt{\frac{\eta}{\veps^2}\abs{\ldblbrace\gradt p^{n+1}\rdblbrace_{\sigma}}}\biggr\rbrace\leq\min\biggl\lbrace 1, \frac{1}{3}\frac{\min\lbrace\rk^n, \vrho^n_L\rbrace}{\max\lbrace\vrho^{n+1}_K,\vrho^{n+1}_L\rbrace} \biggr\rbrace,
    \end{equation}
    where $\abs{\partial K} = \sum_{\sink}\abssig$. Then, $\delt$ satisfies \eqref{eqn:time-step-comp}.
\end{proposition}

\begin{proof}
    The proof follows exactly as given in \cite[Proposition 3.2]{CDV17}.
\end{proof}

For the limit scheme $\pepsho$, the timestep has to be chosen according to 
\begin{equation}
\label{eqn:time-step-incomp}
    \frac{\delt}{\absk}\sum_{\substack{\sink \\ \sigma = K\vert L}}\abssig(-(\wsk^n)^{-}) \leq \half,
\end{equation}
as required in Theorem \ref{thm:disc-loc-ent-incomp}. Analogously, we can derive a sufficient condition as follows.

\begin{proposition}
\label{prop:suff-time-step-incomp}
    For each $\sigma = K\vert L\in\E(K)$, suppose the following holds:
    \begin{equation}
    \label{eqn:suff-time-step-incomp}
        \delt\max\biggl\lbrace\frac{\abs{\partial K}}{\abs{K}},\frac{\abs{\partial L}}{\abs{L}}\biggr\rbrace\biggl\lbrace\abs{\ldblbrace\bv^n\rdblbrace_\sigma} +\sqrt{\eta\abs{\ldblbrace\gradt\pi^{n+1}\rdblbrace_{\sigma}}}\biggr\rbrace\leq \beta_d,
    \end{equation}
    where $\beta_d = 1/8$ if $d=2$ and $\beta_d = 1/12$ if $d=3$. Then, $\delt$ satisfies $\eqref{eqn:time-step-incomp}$.
\end{proposition}
\begin{proof}
    Note that 
    \[
    \frac{\delt}{\absk}\sum_{\substack{\sink \\ \sigma = K\vert L}}\abssig(-(\wsk^n)^{-})  \leq \frac{\delt}{\absk}\sum_{\substack{\sink \\ \sigma = K\vert L}}\abssig \bigl(\abs{\ldblbrace\bv^n\rdblbrace_\sigma} + \eta\delt\abs{\ldblbrace\gradt\pi^{n+1}\rdblbrace_{\sigma}}\bigr)
    \]
    Therefore, the required condition will hold true if,
    \[
    \frac{\delt}{\absk}\sum_{\substack{\sink \\ \sigma = K\vert L}}\abssig \bigl(\abs{\ldblbrace\bv^n\rdblbrace_\sigma} + \eta\delt\abs{\ldblbrace\gradt\pi^{n+1}\rdblbrace_{\sigma}}\bigr) \leq \half
    \]
    Depending on whether $d=2$ or $3$, the number of edges $\sigma$ for a control volume $K$ will be either $4$ or $6$. Therefore, the above condition will hold true when for each $\sigma = K\vert L\in\E(K)$,
    \[
    \frac{\abs{\partial K}}{\absk}\biggl\lbrace\delt\abs{\ldblbrace\bv^n\rdblbrace_\sigma} + \eta\delt^2\abs{\ldblbrace\gradt\pi^{n+1}\rdblbrace_{\sigma}}\biggr\rbrace\leq\beta_d,
    \]
    where $\beta_d = 1/8$ if $d=2$ and $\beta_d = 1/12$ if $d=3$.

    Now, if $0\leq ax^2+bx\leq 1$ for $a,b,x\geq 0$, then $ax^2+bx\leq x(\sqrt{a}+b)$. Given that $\beta_d < 1$ when $d = 2$ or 3, we can deduce that for each $\sigma = K\vert L$, assuming a condition of the form
    \[ 
    \delt\max\biggl\lbrace\frac{\abs{\partial K}}{\abs{K}},\frac{\abs{\partial L}}{\abs{L}}\biggr\rbrace\biggl\lbrace\abs{\ldblbrace\bv^n\rdblbrace_\sigma} +\sqrt{\eta\abs{\ldblbrace\gradt\pi^{n+1}\rdblbrace_{\sigma}}}\biggr\rbrace\leq \beta_d,
    \]
    guarantees $\delt$ will satisfy \eqref{eqn:time-step-incomp}.
\end{proof}

\begin{remark}
    The sufficient conditions \eqref{eqn:suff-time-step-comp} and \eqref{eqn:suff-time-step-incomp} are still implicit in nature. Thus, in our computations, we have implemented them in an explicit manner b.y imposing them at time $t^n$.
\end{remark}

\subsection*{Convergence Test}
\label{subsec:exp}
We consider the domain $\Omega = \lbrack 0, 1\rbrack\times\lbrack 0,1\rbrack$ with periodic boundary conditions, which is divided into a uniform grid of $k\times k$ cells with mesh size $h_k = 1/k$. In what follows, we consider $k = 32\cdot 2^j$ with $j = 0,\dots, 5$. 

Let $U_k = U_{h_k}$ denote the numerical solution computed on the mesh with size $h_k$. Hence, $U_k = \lbrack\vrho_k,m_{1,k},m_{2,k}\rbrack$ in the compressible case or $U_k = \lbrack v_{1,k}, v_{2,k}\rbrack$ in the incompressible case. Let $\overline{U}_k$ and $\Tilde{U}_k$ respectively denote the Ces\`{a}ro average of the numerical solutions and their first variance, i.e. 
\[
\overline{U}_k = \frac{1}{k}\sum_{j=1}^k U_j,\,\,\,\,\,\Tilde{U}_k = \frac{1}{k}\sum_{j=1}^k\abs{U_j-\overline{U}_k}.
\]

Let $U_{ref}$ denote the reference solution computed on the finest possible grid. In our case, $U_{ref} = U_{1024}$ which is computed on a grid with $1024\times 1024$ cells. Analogous to \cite{FLM+21a, FLM+21b}, we compute the four errors 
\begin{equation}
\label{eqn:dmv-errors}
E_1 = \|U_k-U_{ref}\|,\, E_2 = \|\overline{U}_k-\overline{U}_{ref}\|,\, E_3 = \|\Tilde{U}_k-\Tilde{U}_{ref}\|,\, E_4 = \|W_1(\overline{\mcv}_{t,x}^k,\overline{\mcv}_{t,x}^{ref})\|.
\end{equation}
Here, $\norm{\cdot}$ denotes the $L^1$-norm and $\overline{\mcv}^k_{t,x} = \frac{1}{k}\sum_{j=1}^k\delta_{U_j(t,x)}$,  where $\delta_{U_j(t,x)}$ denotes the Dirac measure centered at $U_j(t,x)$ for $(t,x)\in\odom$. $W_1$ is the 1-Wasserstein distance and we compute it using the {\texttt{wasserstein\textunderscore distance}} function available in the {\texttt{scipy.stats}} package of the open source library SciPy, cf.\ \cite{JOP01}.

We consider the following initial data from \cite{DT11}.
\begin{align}
    \vrho_\veps(0,x_1,x_2) &= 1 + \veps^2\sin^2(2\pi(x_1+x_2)), \label{eqn:deg-tg-den}\\
    m_{1,\veps}(0,x_1,x_2) &= \sin(2\pi(x_1-x_2)) + \veps^2\sin(2\pi(x_1+x_2)), \label{eqn:deg-tg-momx}\\
    m_{2,\veps}(0,x_1,x_2) &= \sin(2\pi(x_1-x_2)) + \veps^2\cos(2\pi(x_1+x_2)),\label{eqn:deg-tg-momy}
\end{align}
where $\bm_\veps = (m_{1,\veps}, m_{2,\veps})$.

The final time is set as $T = 0.02$ with pressure law $p(\vrho) = \vrho^2$ and we consider $\veps = 10^{-i}$, $i=0,\cdots,4$.
\begin{figure}[htpb]
    \centering
    \includegraphics[height = 0.18\textheight]{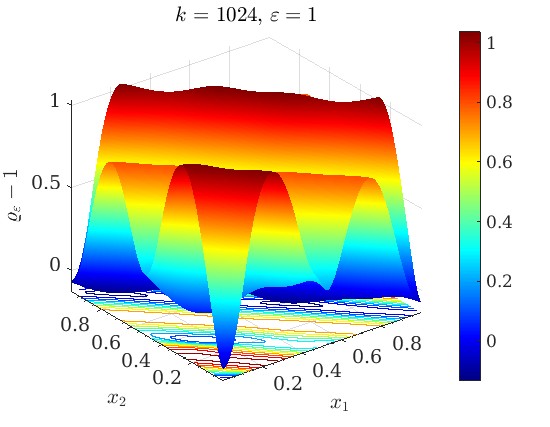}
    \includegraphics[height = 0.18\textheight]{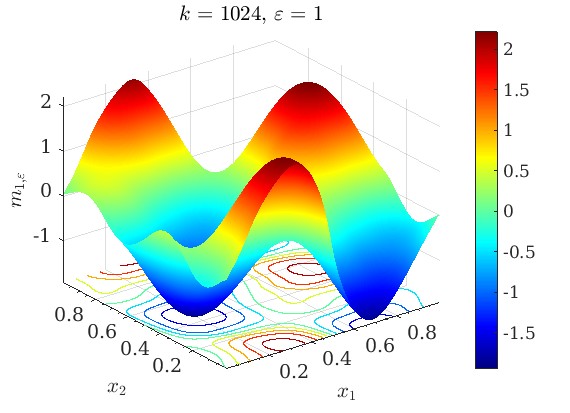}
    \includegraphics[height = 0.18\textheight]{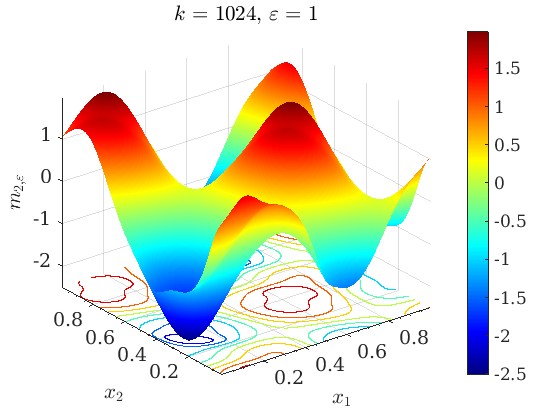}
    \includegraphics[height = 0.18\textheight]{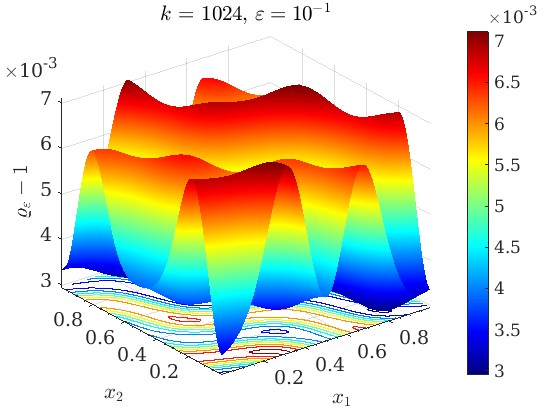}
    \includegraphics[height = 0.18\textheight]{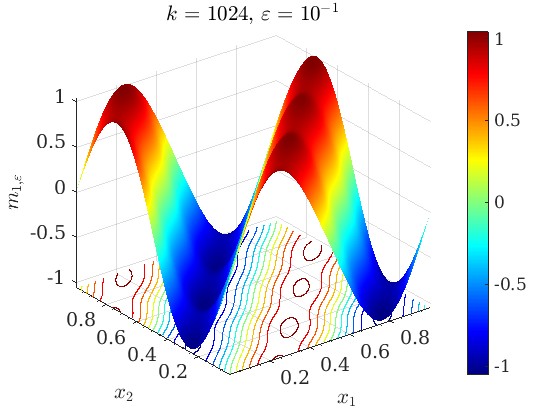}
    \includegraphics[height = 0.18\textheight]{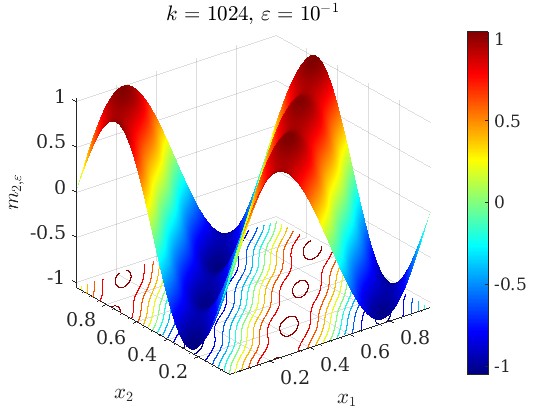}
    \includegraphics[height = 0.18\textheight]{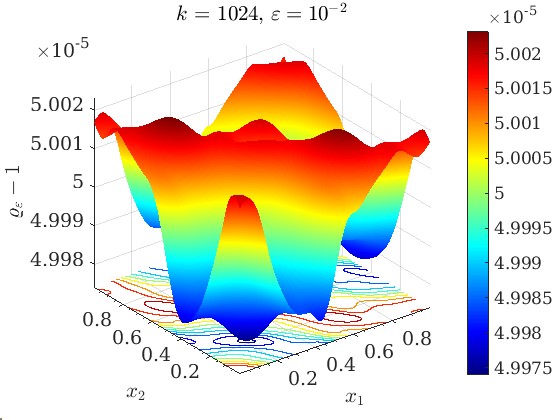}
    \includegraphics[height = 0.18\textheight]{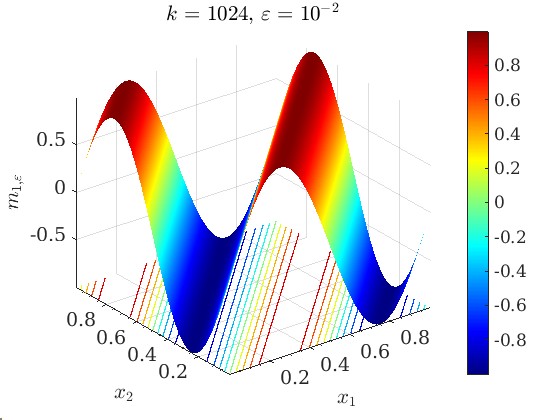}
    \includegraphics[height = 0.18\textheight]{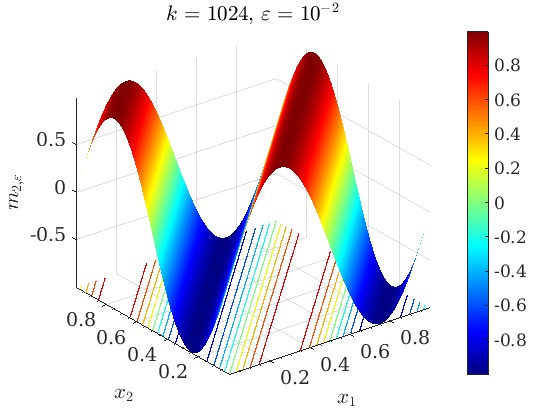}
    \includegraphics[height = 0.18\textheight]{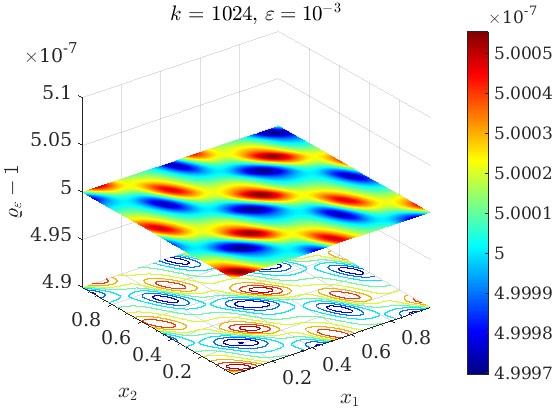}
    \includegraphics[height = 0.18\textheight]{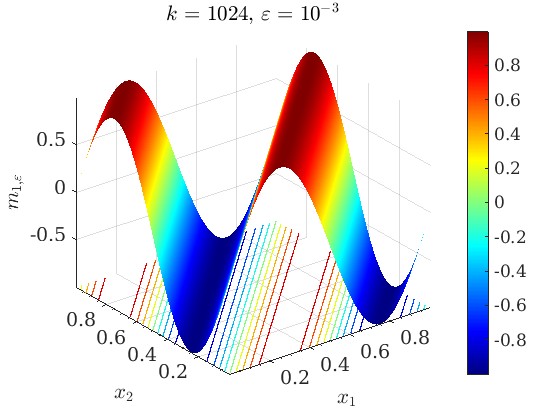}
    \includegraphics[height = 0.18\textheight]{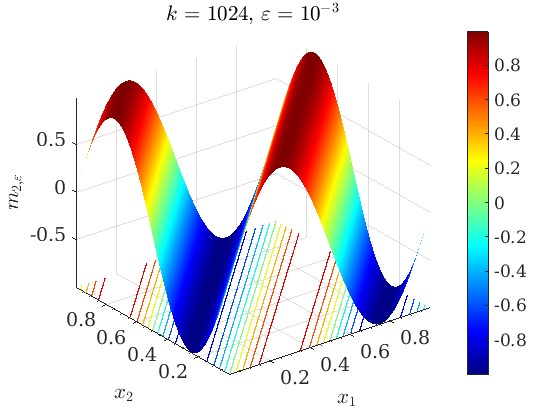}
    \includegraphics[height = 0.18\textheight]{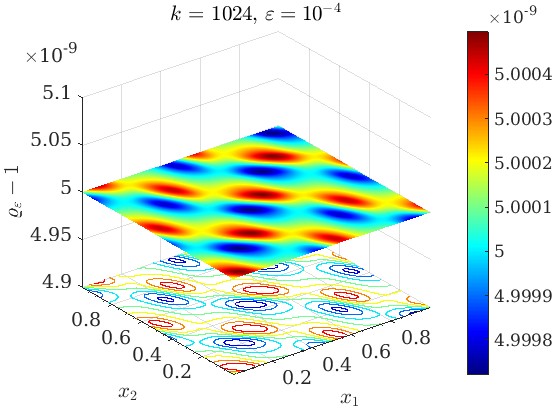}
    \includegraphics[height = 0.18\textheight]{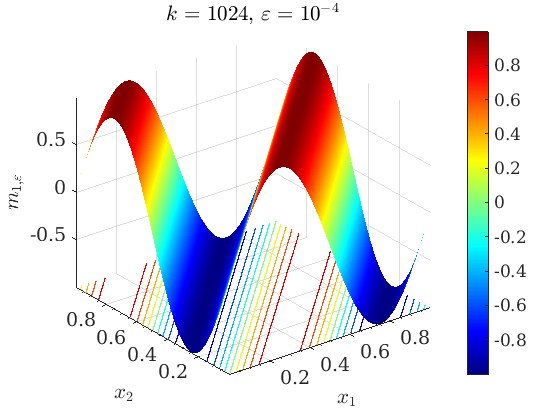}
    \includegraphics[height = 0.18\textheight]{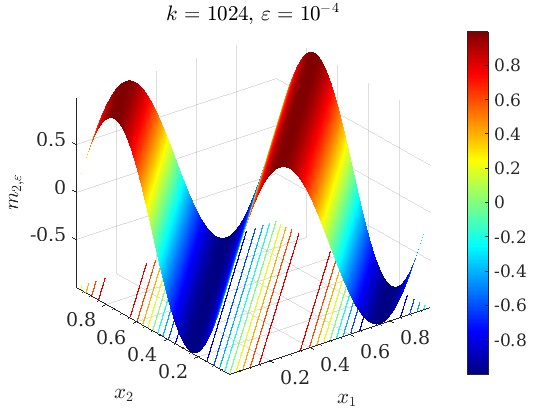}
    \caption{Deviations of density from 1 and momentum profiles when $k=1024$ for $\veps = 1, 10^{-1},\dots, 10^{-4}$.}
    \label{fig:profiles1}
\end{figure}

In Figure \ref{fig:profiles1}, the profiles of the deviation of the density $\vrho_\veps$ from the constant value $1$ and the momentum $\bm_\veps = (m_{1,\veps}, m_{2,\veps})$ at the final time $T=0.02$, computed using the scheme \eqref{eqn:disc-mss-bal}-\eqref{eqn:disc-mom-bal}, are presented for $\veps = 1,10^{-1},\dots,10^{-4}$. One can observe that the magnitude of the deviation is proportional to $\veps^2$ and in the cases of $\veps = 10^{-3}$ and $\veps = 10^{-4}$, the density is near constant. Also, one can clearly observe the symmetric nature of the momentum components, which is along the same lines as reported in \cite{DT11}. In Figure \ref{fig:div-prof}, we present the discrete divergence of the velocity field $\bu_\veps$ at final time $T=0.02$ when $k = 1024$ and $\veps = 10^{-4}$, and we note that the divergence is of the order $10^{-5}$. In what follows, we divide the case study into two parts, Part A and Part B. In Part A, we present the various results obtained using the solutions computed by the velocity stabilized semi-implicit scheme \eqref{eqn:disc-mss-bal}-\eqref{eqn:disc-mom-bal}. On the other hand, we present the results obtained using the incompressible limit s\label{eqn:deg-tg-ini-den}cheme \eqref{eqn:disc-incomp-div-free}-\eqref{eqn:disc-incomp-mom} in Part B.

\begin{figure}[htpb]
    \centering
    \includegraphics[height = 0.25\textheight]{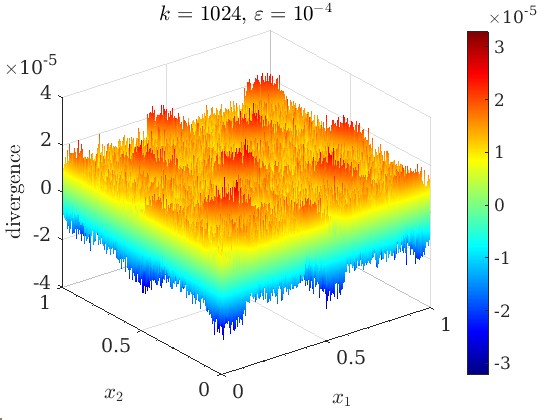}
    \caption{Divergence of the velocity field when $k = 1024$ and $\veps = 10^{-4}$ at final time $T = 0.02$.}
    \label{fig:div-prof}
\end{figure}

\textbf{Part A :\ }
First, we present the asymptotic convergence results pertaining to the limit $\veps\to 0$ on a fixed grid of size 128 $\times 128$. As observed in Figure \ref{fig:profiles1}, $\vrho_\veps$ is nearly 1 for small values of $\veps$. We now wish to determine the rate of convergence of $\vrho_\veps$ to 1 as $\veps\to 0$. We recall from Lemma \ref{lem:den-conv} that $\norm{\vrho_\veps - 1}_{L^\infty(0,T; L^\gamma(\Omega))} = O(\veps)$, as $\gamma = 2$ in this case. In order to verify this claim, we first present the plots of $\norm{\vrho_\veps(t,\cdot) - 1}_{L^\gamma(\Omega)}$ over time for different values of $\veps$ and then, we present the plot of $\norm{\vrho_\veps - 1}_{L^\infty(0,T;L^\gamma(\Omega))}$ versus $\veps$ in Figure \ref{fig:lgamma-norm-den}. 
\begin{figure}[htpb]
    \centering
    \includegraphics[height = 0.25\textheight]{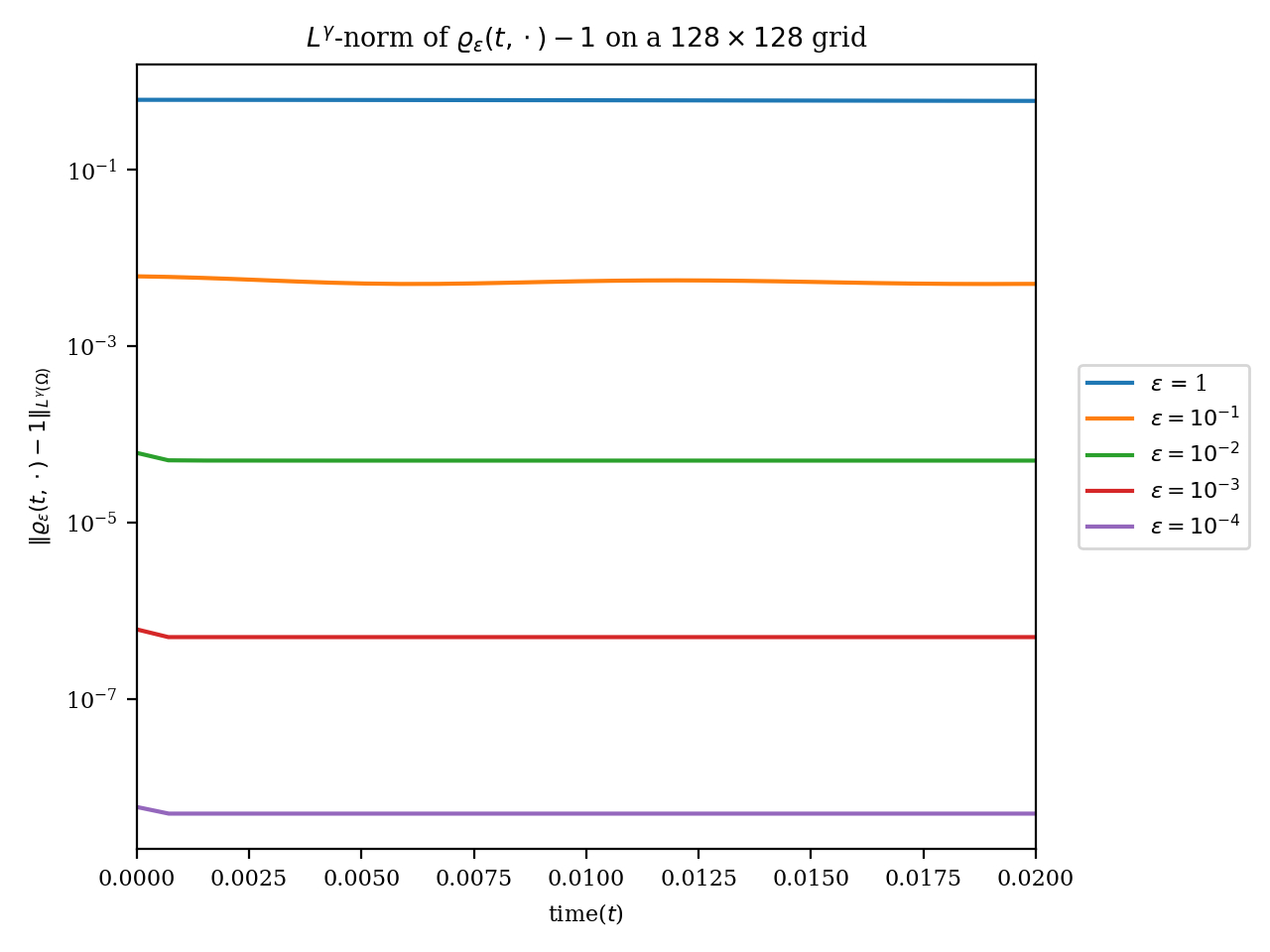}
    \includegraphics[height = 0.25\textheight]{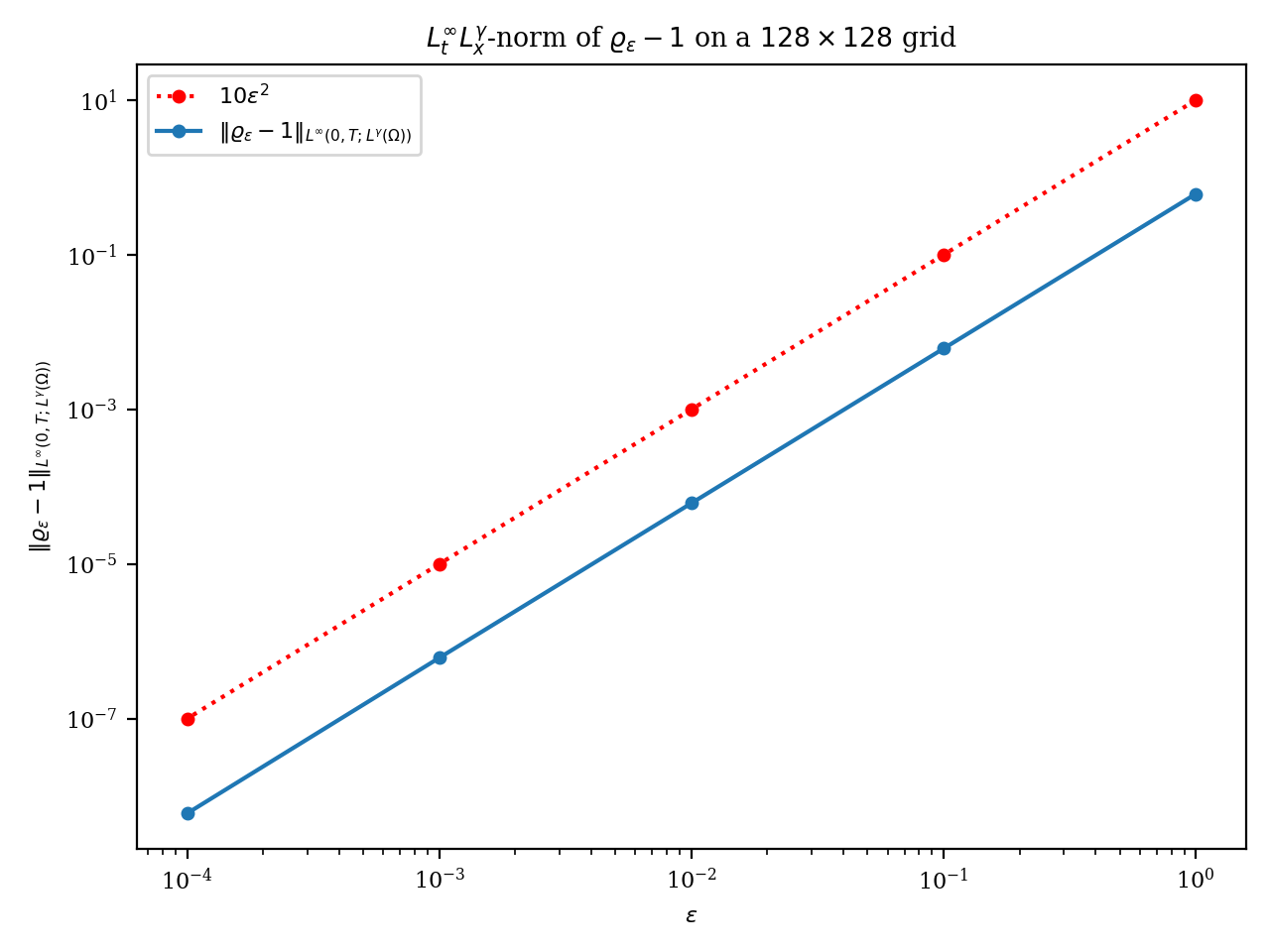}
    \caption{$\norm{\vrho_\veps(t,\cdot) - 1}_{L^\gamma(\Omega)}$ over time for different values of $\veps$ (left). \\
    $\norm{\vrho_\veps - 1}_{L^\infty(0,T;L^\gamma(\Omega))}$ versus $\veps$ (right). Results are computed on a $128\times 128$ grid.}
    \label{fig:lgamma-norm-den}
\end{figure}

For any $t\in (0,T)$, one can clearly observe from the figure on the left in Figure \ref{fig:lgamma-norm-den} that $\norm{\vrho_\veps(t,\cdot) - 1}_{L^\gamma(\Omega)}$ is decreasing as $\veps\to 0$. Further, one can also observe that the rate of decrease is proportional to $\veps^2$. In order to further support this claim, we also plot the reference quantity $10\veps^2$ in the plot of $\norm{\vrho_\veps - 1}_{L^\infty(0,T;L^\gamma(\Omega))}$. This leads us to conclude a better convergence rate than the one claimed in Lemma \ref{lem:den-conv}, namely,  $\norm{\vrho_\veps - 1}_{L^\infty(0,T;L^\gamma(\Omega))} = O(\veps^2)$.

Next, according to Theorem \ref{thm:asymp-lim-scheme}, $\bu_{\T,\delt,\veps}\to\bv_{\T,\delt}$ as $\veps\to 0$ in any discrete norm on a fixed grid, where $\bv_{\T,\delt}$ is computed using the limit scheme \eqref{eqn:disc-incomp-div-free}-\eqref{eqn:disc-incomp-mom}. In order to substantiate this claim, we compute the $L^1$-norm of $\bu_{\T,\delt,\veps}(T,\cdot) - \bv_{\T,\delt}(T,\cdot)$ at the final time $T = 0.02$ and present the values in Table \ref{tab:vel-err-norm}, where we can clearly see the decrease in the norm as $\veps\to 0$.
\begin{center}
\begin{table}[htpb]
\begin{tabular}{|c|c|c|c|c|c|}
    \hline
    $\veps$ & 1 & $10^{-1}$ & $10^{-2}$ & $10^{-3}$ & $10^{-4}$ \\
    \hline
    $\norm{\bu_{\T,\delt,\veps}(T,\cdot) - \bv_{\T,\delt}(T,\cdot)}_{L^1(\Omega)}$ & 9.49e-1 & 3.15e-2 & 8.72e-4 & 8.13e-4 & 7.96e-4 \\
    \hline
\end{tabular}
\caption{$L^1$-norm of $\bu_{\T,\delt,\veps}(T,\cdot) - \bv_{\T,\delt}(T,\cdot)$ on a 128 $\times$ 128 grid.}
\label{tab:vel-err-norm}
\end{table}
\end{center}
Now, we present the mesh convergence results pertaining to the limit $h_\T\to 0$ while $\veps$ is fixed. In Figure \ref{fig:dmv_err_comp}, we present the profiles of the errors $E_1,\dots,E_4$, that are defined in \eqref{eqn:dmv-errors}, for $\veps = 10^{-i},\,i=0,\dots,4$. As expected, in each case, the errors decrease upon mesh refinement which leads us to conclude that we indeed have convergence towards a DMV solution of the compressible Euler equations for each fixed $\veps$. Further, as the error $E_1$ is also decreasing for each $\veps$, we can conclude that we obtain convergence to a weak solution in every case, meaning the DMV solution $\mcv^\veps$ is simply the Dirac mass centered at $\lbrack\vrho_\veps, \bm_\veps\rbrack$, which is the weak solution to the compressible Euler system that the numerical solutions converge to.

\begin{figure}[htpb]
    \centering
    \begin{subfigure}{0.35\textheight}
        \includegraphics[width=\textwidth ]{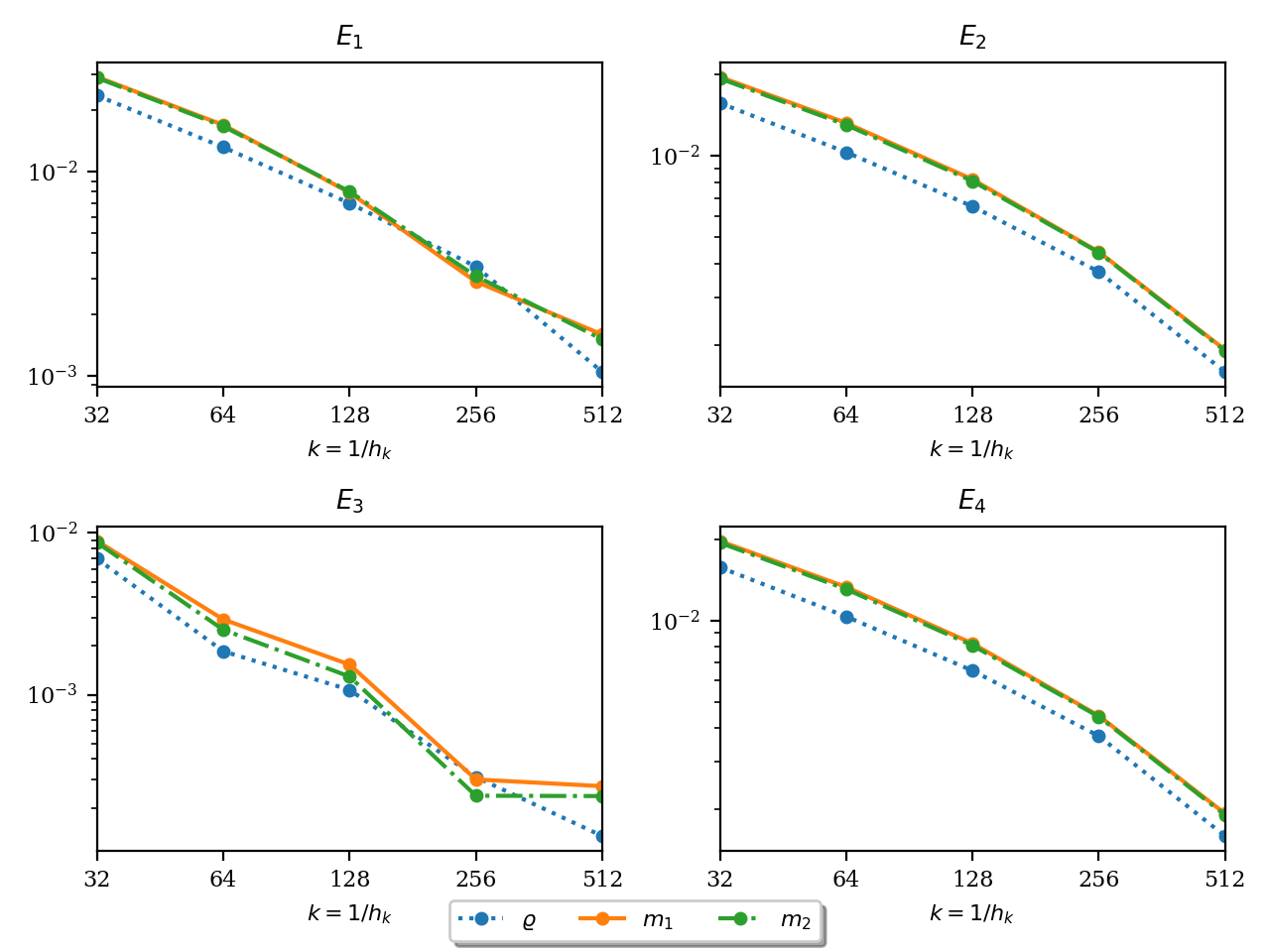}
        \caption{Error Profiles when $\veps = 1$.}
        \label{fig:dmv_err_eps_1}
    \end{subfigure}
    \hfill
    \begin{subfigure}{0.35\textheight}
        \includegraphics[width=\textwidth]{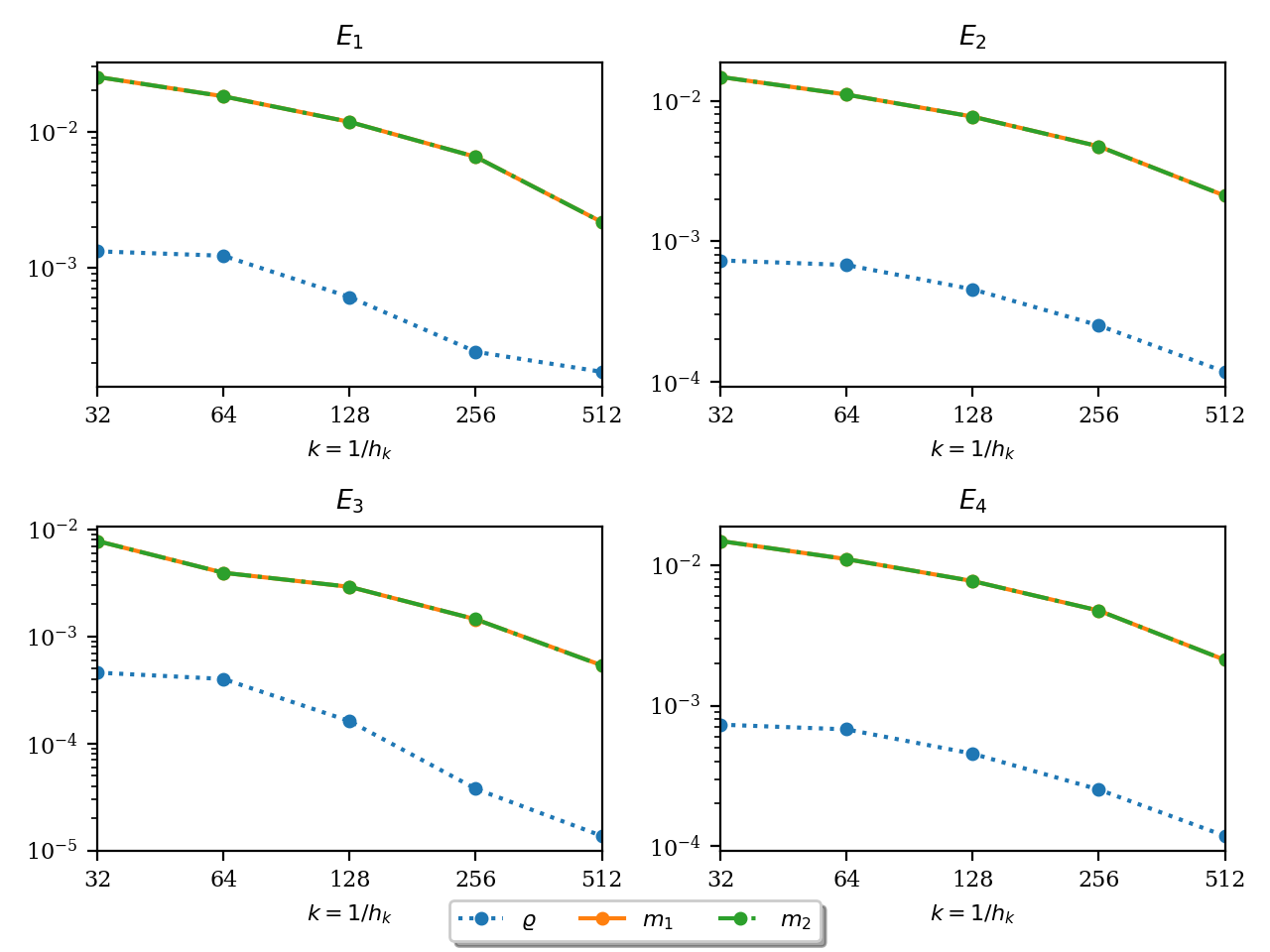}
        \caption{Error Profiles when $\veps = 10^{-1}$.}
    \label{fig:dmv_err_eps_10e-1}
    \end{subfigure}
    \hfill
    \begin{subfigure}{0.35\textheight}
        \includegraphics[width=\textwidth]{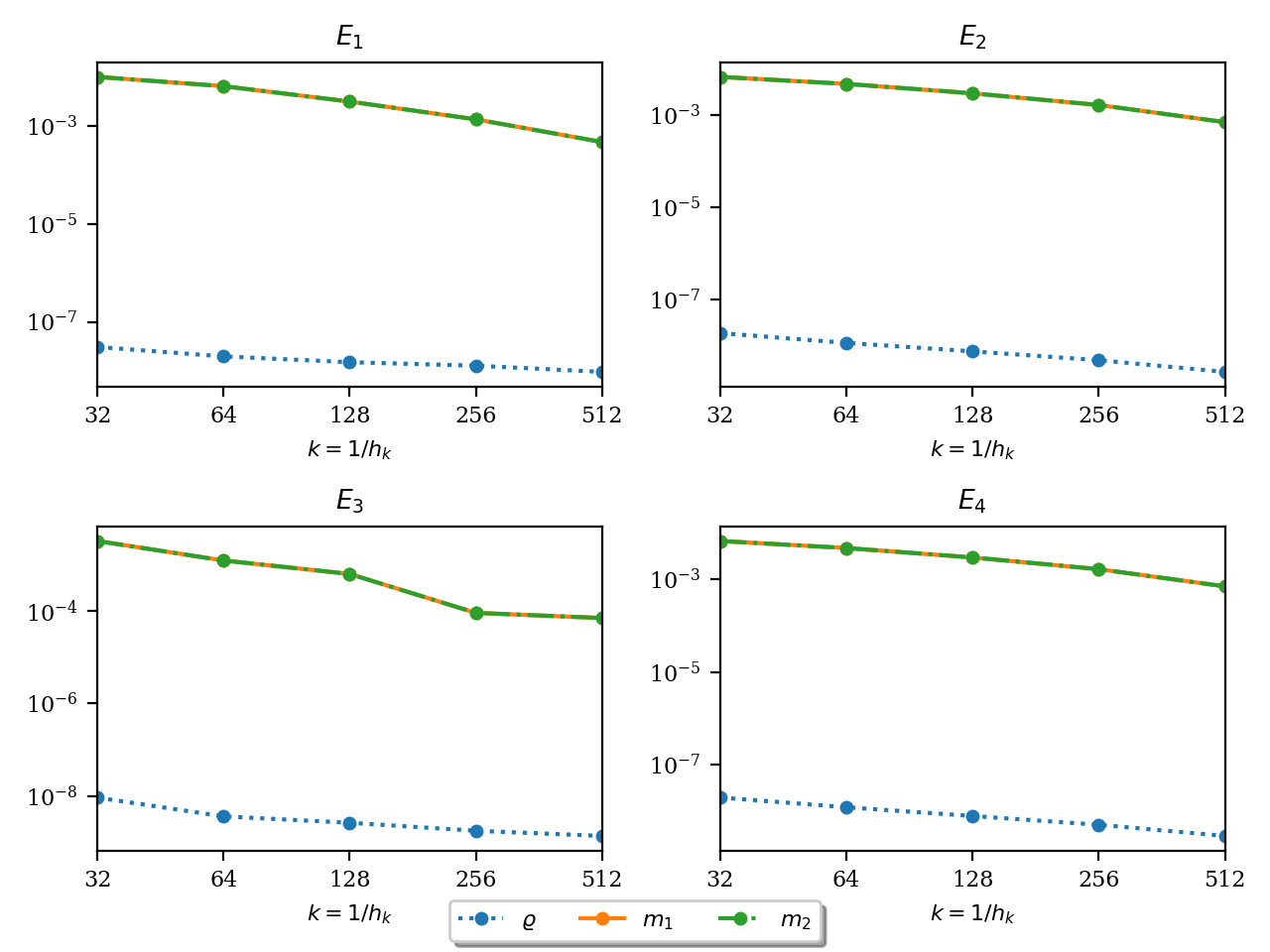}
        \caption{Error Profiles when $\veps = 10^{-2}$.}
    \label{fig:dmv_err_eps_10e-2}
    \end{subfigure}
    \hfill
    \begin{subfigure}{0.35\textheight}
        \includegraphics[width=\textwidth]{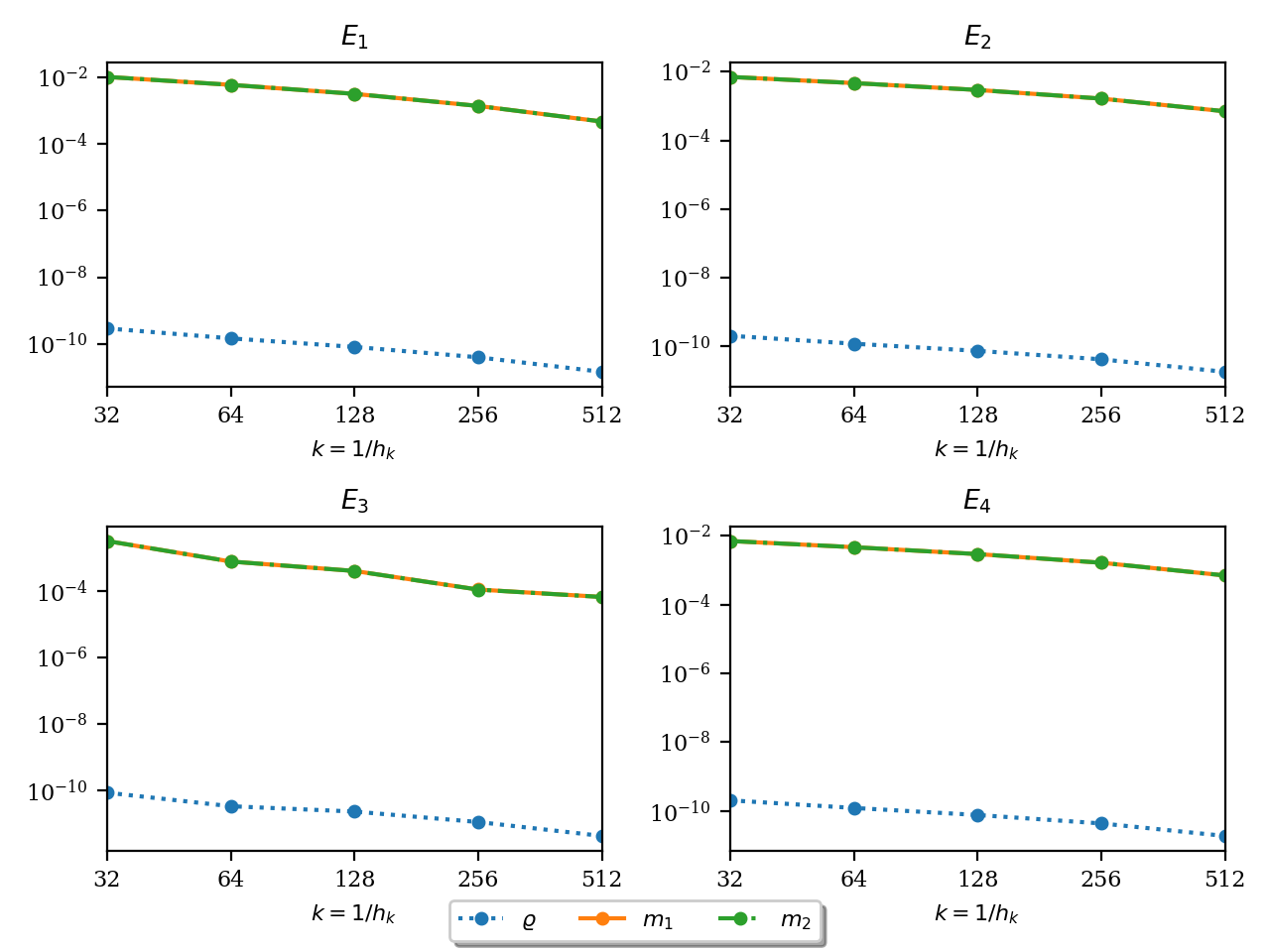}
        \caption{Error Profiles when $\veps = 10^{-3}$.}
    \label{fig:dmv_err_eps_10e-3}
    \end{subfigure}
    \hfill
    \begin{subfigure}{0.35\textheight}
        \includegraphics[width=\textwidth]{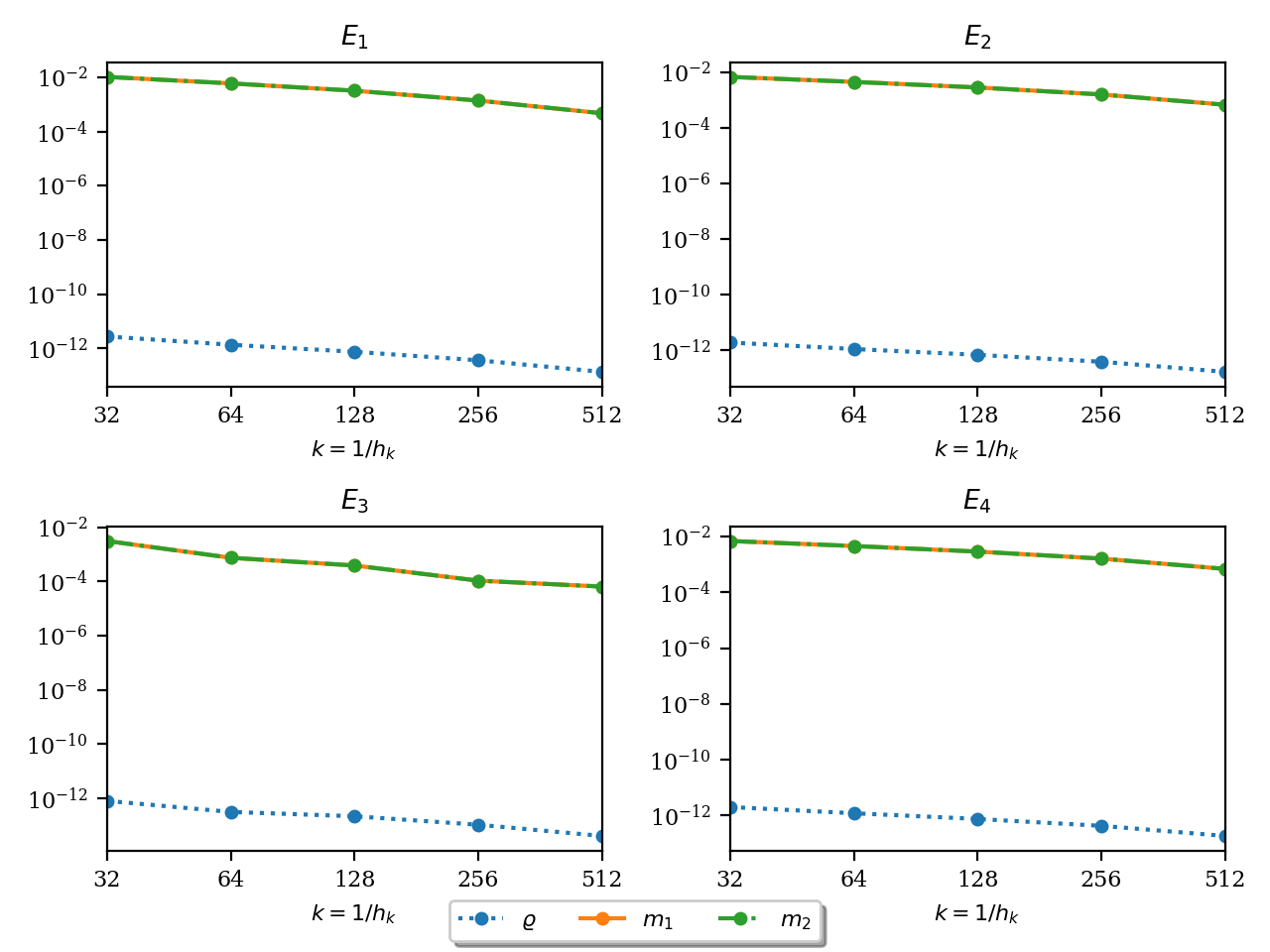}
        \caption{Error Profiles when $\veps = 10^{-4}$.}
    \label{fig:dmv_err_eps_10e-4}
    \end{subfigure}
    \caption{Error profiles for different values of $\veps$.}
    \label{fig:dmv_err_comp}
\end{figure}

\textbf{Part B :\ }Firstly, we observe from $\eqref{eqn:deg-tg-momx}-\eqref{eqn:deg-tg-momy}$ that the initial data for the incompressible Euler system will be given by 
\[
 v_1(0,x_1,x_2) = v_2(0,x_1,x_2) = \sin(2\pi(x_1 - x_2)).
\]
The setup is exactly the same as in the compressible case, the domain is $\Omega = \lbrack 0,1\rbrack\times\lbrack 0,1\rbrack$ with periodic boundaries and the final time is $T = 0.02$. First, we present the errors $E_1,\dots, E_4$ in Figure \ref{fig:dmv-incomp-err}. As before, we take the reference solution to be the solution computed on the $1024\times 1024$ mesh.
\begin{figure}
    \centering
    \includegraphics[height=0.3\textheight]{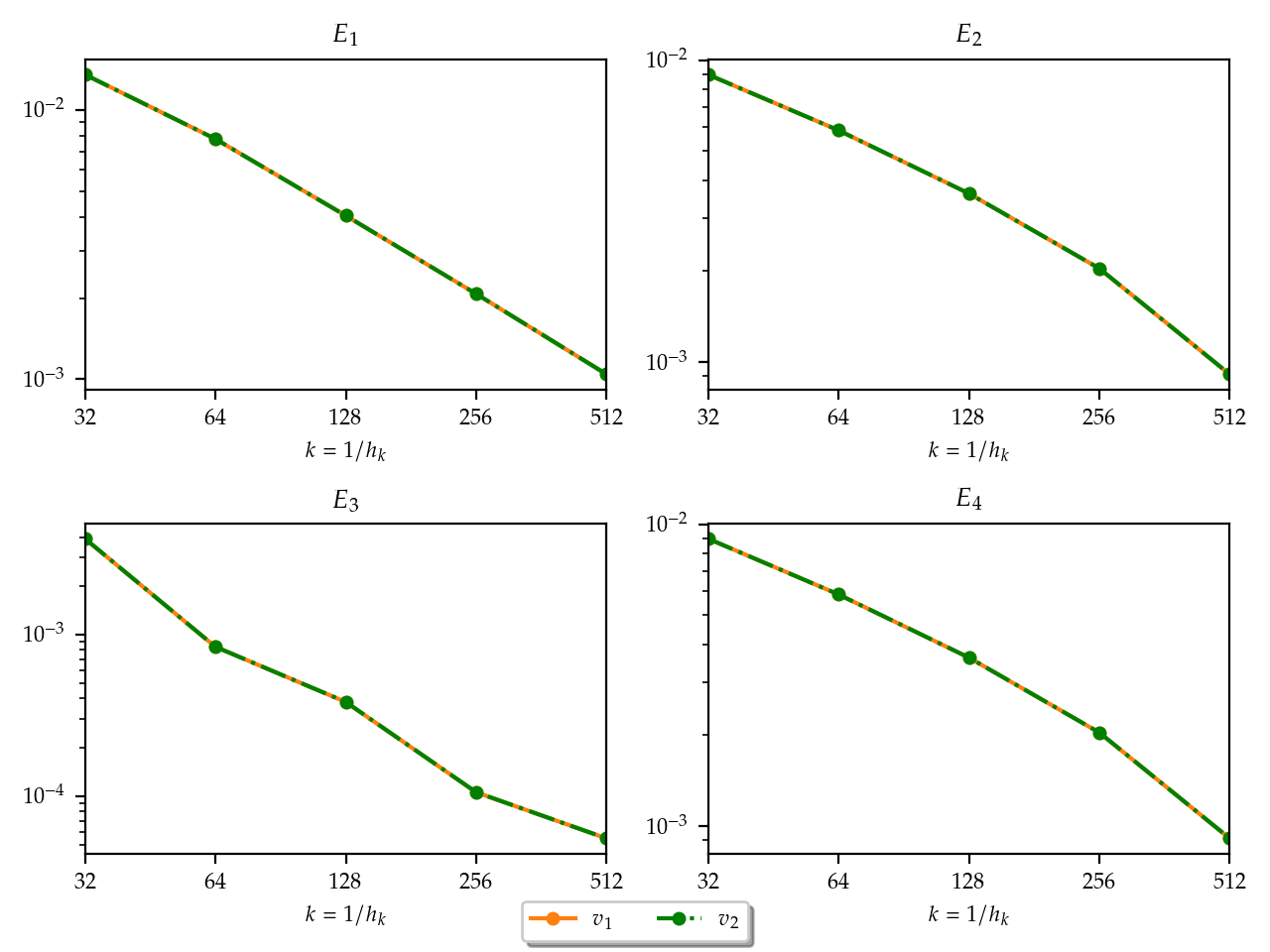}
    \caption{Error profiles in the incompressible case.}
    \label{fig:dmv-incomp-err}
\end{figure}
Clearly, we note from Figure \ref{fig:dmv-incomp-err} that we obtain convergence to a DMV solution as expected. Further, as the initial data is sufficiently regular, it satisfies the hypothesis of Theorem \ref{thm:incomp-soln-exis} and the existence of a classical solution globally in time is guaranteed. Hence, the expected classical convergence is also seen here, which is illustrated by the decreasing behaviour of the error $E_1$ upon mesh refinement. In order to further strengthen this claim, we compute the relative energy, which in the incompressible case is given by 
\[
E_{rel}(\bv\lvert\uu{V})(t) = \half\int_{\Omega}\abs{\bv - \uu{V}}^2\,\dx,
\]
where $\bv,\uu{V}$ are two solutions to the incompressible Euler system. We set $\uu{V}$ to be the reference solution i.e.\ the solution computed on the grid of size $1024 \times 1024$ and we take $\bv$ to be the solution computed on a grid of size $k\times k$, $k = 32\cdot 2^j$, $j = 0,\dots,4$. We compute the relative energy at the final time and present the results in Table \ref{tab:rel-ent-incomp}. Upon refinement, the relative energy is decreasing, which firmly points towards the convergence of the numerical solutions towards the classical solution of the incompressible Euler system, which is in line with the weak-strong uniqueness principle, cf.\ Theorem \ref{thm:wk-str-incomp}. Also, in order to determine the experimental order of convergence (EOC), we compute the $L^2$-norms of $\bv - \uu{V}$ and present them in Table \ref{tab:cg-rate}. As expected, the rate of convergence is around 1.
\begin{center}
\begin{table}[htpb]
\begin{tabular}{|c|c|c|c|c|c|}
    \hline
    $k$ & 32 & 64 & 128 & 256 & 512 \\
    \hline
    $E_{rel}(\bv\vert\uu{V})(T)$ & 2.46e-4 & 7.86e-5 & 2.07e-5 & 5.39e-6 & 1.35e-6 \\
    \hline
\end{tabular}
\caption{Relative energy values in the incompressible case.}
\label{tab:rel-ent-incomp}
\end{table}
\begin{table}[htpb]
\begin{tabular}{|c|c|c|c|c|c|}
    \hline
    $k$ & 32 & 64 & 128 & 256 & 512 \\
    \hline
    $\norm{\bv - \uu{V}}_{L^2}$ & 2.21e-2 & 1.25e-2 & 6.44e-3 & 3.28e-3 & 1.64e-3 \\
    \hline
    EOC & - & 0.8236 & 0.9611 & 0.9721 & 0.9955 \\
    \hline
\end{tabular}
\caption{Experimental order of convergence.}
\label{tab:cg-rate}
\end{table}
\end{center}
Now, as we have ascertained the convergence towards a classical solution in the incompressible case, we can finally verify the convergence claimed in Theorem \ref{thm:eps-h-lim}, which is given along the lines of Theorem \ref{thm:asymp-lim-dmv}. We compute the relative energy $E_{rel}(\vrho_\veps, \bm_\veps\vert r,\bU)$, cf.\ \eqref{eqn:rel-ent}, by setting $r = 1$ and $\bU$ as the solution of the incompressible system computed on the $1024\times 1024$ grid using the limit scheme \eqref{eqn:disc-incomp-div-free}-\eqref{eqn:disc-incomp-mom}. $\vrho_\veps$ and $\bm_\veps$ respectively are taken to be the numerical solutions to the compressible system computed using the scheme \eqref{eqn:disc-mss-bal}-\eqref{eqn:disc-mom-bal} on a grid of size $1024\times 1024$, when $\veps = 10^{-i},\,i=0,\dots,4$. We compute the relative energy at the final time and plot it against $\veps$ in Figure \ref{fig:rel-ent-dmv-lim}. As $\veps$ is decreasing, the relative energy values are also decreasing, and this verifies the claim. 

\begin{figure}
    \centering
    \includegraphics[height=0.3\textheight]{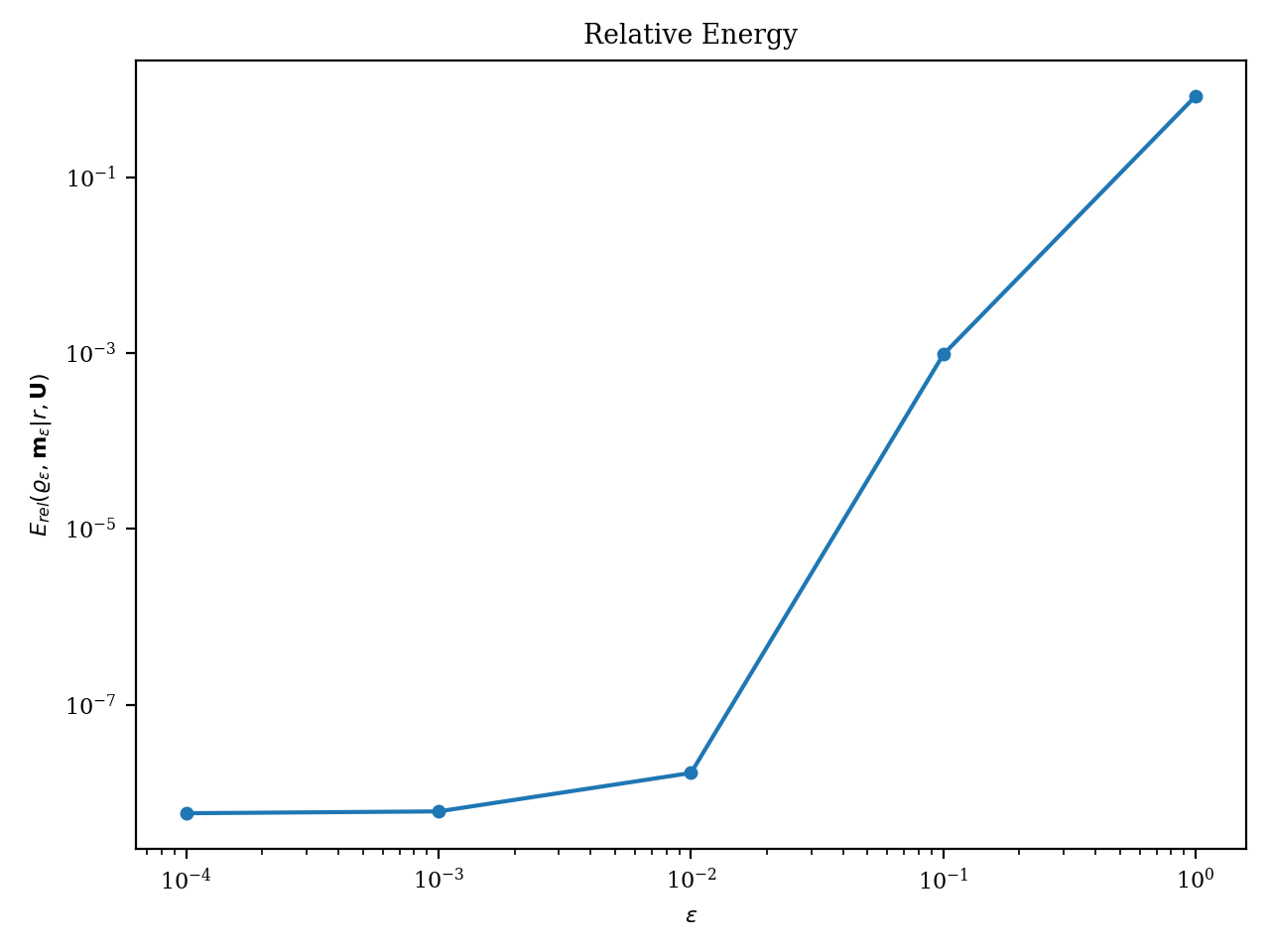}
    \caption{Relative Energy}
    \label{fig:rel-ent-dmv-lim}
\end{figure}

\section{Conclusion}
\label{sec:conc}

In the present work, we have proposed and analyzed a new AP scheme for simulating low Mach number inviscid barotropic flows, where the latter are governed by the compressible Euler equations. 

The proposed numerical method belongs to the class of collocation finite volume methods. It works with piecewise constant numerical solutions, uses a stabilized upwind numerical flux where the stabilization term is proportional to the stiff pressure gradient, and employs a semi-implicit time discretization of the convective terms. The scheme is proven to be conditionally energy stable and consistent; see Theorems \ref{thm:disc-loc-ent-ineq} - \ref{thm:cons-scheme}. Using the concept of DMV solutions, we rigorously illustrate the asymptotic preserving properties of the scheme for well-prepared initial data.

In particular, we have proved the following results. We have shown that the numerical solutions of the proposed numerical scheme will converge to a DMV of the compressible Euler system as the mesh size vanishes while the Mach number remains fixed; see Theorem \ref{thm:scheme-wk-con}. In addition, assuming that the initial data is well prepared, the relative energy between the limiting measure-valued solution and the classical solution of the incompressible Euler system vanishes as long as the Mach number vanishes; see Theorem \ref{thm:eps-h-lim}. On the other hand, performing first the low Mach number limit of the numerical solutions generated by the scheme on a fixed mesh yields a semi-implicit scheme, which is a consistent and energy stable approximation of the incompressible Euler equations; see Theorems \ref{thm:asymp-lim-scheme}, \ref{thm:disc-loc-ent-incomp}, \ref{thm:disc-glob-ent-incomp} and \ref{thm:cons-incomp-scheme}. Finally, we study the convergence of the numerical solutions of the limiting semi-implicit scheme and show that they converge weakly to a DMV of the incompressible Euler equations; see Theorem \ref{thm:scheme-wk-con-incomp}. These results are summarized in Theorem \ref{thm:comm-lim} that formulates the AP properties of our numerical scheme.

The theoretical results are illustrated by a series of numerical experiments in Section \ref{sec:num-exp}, which confirm the asymptotic preserving properties of the proposed scheme. Here, we have studied the convergence of individual numerical solutions as well as the convergence of their empirical means, deviations and empirical measures applying the Wasserstein distance.

\appendix
\renewcommand{\thesection}{\Alph{section}}
\renewcommand{\thesubsection}{A\arabic{subsection}}
\section{Energy Stability of the Scheme}

In this section, we give a sketch of the proofs of Theorem \ref{thm:disc-loc-ent-ineq} and Theorem \ref{thm:disc-glob-ent-est}. We first prove the discrete variants of the energy estimates given in Proposition \ref{prop:mod-enrg-est}. By $\llbracket a,b\rrbracket$, we denote the interval $\lbrack\min\lbrace a,b\rbrace, \max\lbrace a,b\rbrace\rbrack$.

\subsection{Discrete Positive Renormalization Identity}
\label{subsec:app-renorm-idt}
Multiply the discrete mass balance \eqref{eqn:disc-mss-bal} with $\psig^\prime(\vrho^{n+1}_K)$ to yield $\mathcal{A}_1 + \mathcal{A}_2 = 0$, where 
\begin{align*}
    &\mathcal{A}_1 = \frac{\psig^\prime(\rk^{n+1})}{\delt}(\rk^{n+1} - \rk^n), \\
    &\mathcal{A}_2 = \frac{1}{\absk}\sum_{\substack{\sink \\ \sigma = K\vert L}} \psig^\prime(\rk^{n+1})\flx(\vrho^{n+1}, \bw^n).
\end{align*}

To simplify $\mathcal{A}_1$, we employ a Taylor expansion in time. Next, using the definition of the mass flux \eqref{eqn:mss-flx}, the relation $r\psig^\prime(r) - \psig(r) = p(r)$ for $r = \rk^{n+1}$ and also employing a Taylor expansion in space yields the following simplified form for $\mathcal{A}_2$.

\begin{equation}
\label{eqn:app_3}
    \mathcal{A}_2 = (\divup(\psig(\vrho^{n+1}), \bw^n))_K + p^{n+1}_K(\divt\bw^n)_K + \frac{1}{2\absk}\sum_{\substack{\sink \\ \sigma = K\vert L}}\abssig(-w^n_{\sk})^{-}(\vrho^{n+1}_L -\vrho^{n+1}_K)^2\psig^{\prime\prime}(\tilde{\vrho}^{n+1}_\sigma),
\end{equation}
where $\tilde{\vrho}^{n+1}_\sigma \in \llbracket\rk^{n+1}, \vrho^{n+1}_L\rrbracket$. Finally, summing up the resulting equations for $\mathcal{A}_1$ and $\mathcal{A}_2$ gives us 
\begin{equation}
\label{eqn:disc-renorm-idt}
    \frac{1}{\delt}(\psig(\rk^{n+1})-\psig(\rk^n))+(\divup(\psig(\vrho^{n+1}),\bw^n))_K+p^{n+1}_K(\divt\bw^n)_K+R^{n+1}_{K,\delt}=0,
\end{equation}
where the remainder term $R^{n+1}_{K,\delt}$ is given by
\begin{equation}
\label{eqn:renorm-rem}
     R^{n+1}_{K,\delt}=\frac{1}{2\delt}(\rk^{n+1}-\rk^n)^2\psig^{\prime\prime}(\overline{\vrho}_{K}^{n+1/2})+\frac{1}{2\absk}\sum_{\substack{\sink \\ \sigma=K\vert L}}\abssig(-(\wsk^n)^{-})(\rk^{n+1}-\vrho_L^{n+1})^2\psig^{\prime\prime}(\Tilde{\vrho}^{n+1}_{\sigma}),
\end{equation}
with $\overline{\vrho}_K^{n+1/2}\in\llbracket \rk^n,\rk^{n+1}\rrbracket$ and $\Tilde{\vrho}^{n+1}_{\sigma}\in\llbracket\rk^{n+1},\vrho_L^{n+1}\rrbracket$. The convexity of $\psig$ allows us to conclude that $R^{n+1}_{K,\delt}\geq 0$ for each $K\in\T$. 

The defintion of $\pig$ along with the discrete mass balance \eqref{eqn:disc-mss-bal} yields the following discrete positive renormalization identity after some straightforward calculations.
\begin{equation}
\label{eqn:disc-pos-renorm-idt}
    \frac{1}{\delt}(\pig(\rk^{n+1})-\pig(\rk^n)) + (\divup(\psig(\vrho^{n+1}) - \psig^\prime(1)\vrho^{n+1},\bw^n))_K + p^{n+1}_K(\divt\bw^n)_K+R^{n+1}_{K,\delt}=0.
\end{equation}

\subsection{Discrete Kinetic Energy Identity}
\label{subsec:app-kin-idt}

Writing $\rk^{n+1}\uk^{n+1} - \rk^n\uk^n = \rk^{n+1}(\uk^{n+1} - \uk^n) + \uk^n(\rk^{n+1} - \rk^n)$ in the discrete momentum balance \eqref{eqn:disc-mom-bal} and using the discrete mass balance \eqref{eqn:disc-mss-bal} to simplify yields the following update for the velocity:
\begin{equation}
\label{eqn:disc-vel-upd}
    \frac{\rk^{n+1}}{\delt}(\uk^{n+1}-\uk^n)+\frac{1}{\absk}\sum_{\substack{\sink \\ \sigma=K\vert L}}\abssig(\bu^n_L-\uk^n)\flx^{n+1,-}+\frac{1}{\veps^2}(\gradt p^{n+1})_K=0,
\end{equation}
where $F^{n+1,-}_{\sk} = \vrho^{n+1}_L(\wsk^n)^{-}$ for $\sigma = K\vert L$.

We take the dot product of the above equation with $\uk^n$ and use the algebraic identity $(a-b)\cdot b = (\abs{a}^2 - \abs{b}^2 - \abs{a-b}^2)/2$ to simplify the first and second terms. Next, multiplying the discrete mass balance \eqref{eqn:disc-mss-bal} by $\abs{\uk^n}^2/2$, and summing it with the previously obtained expression and simplifying gives us the following kinetic energy identity:
\begin{equation}
\label{eqn:disc-kin}
    \begin{split}
        \frac{1}{\delt}\biggl(\half\rk^{n+1}\abs{\uk^{n+1}}^2-\half\rk^n\abs{\uk^n}^2\biggr)+\biggl(\divup\biggl(\half\vrho^{n+1}\abs{\bu^n}^2,\bw^n\biggr)\biggr)_K+&\frac{1}{\veps^2}(\gradt p^{n+1})_K\cdot\bv^n_K+S^{n+1}_{K,\delt} \\
        &=-\frac{\eta\delt}{\veps^4}\abs{(\gradt p^{n+1})_K}^2,
    \end{split}
\end{equation}
where, we have used $\uk^n = \bv^n_K + \eta\delt(\gradt p^{n+1})_K$ to rewrite the pressure gradient term. The remainder term $S^{n+1}_{K,\delt}$ is given by 
\begin{equation}
\label{eqn:ke-rem}
    S^{n+1}_{K,\delt}=-\frac{1}{2\delt}\rk^{n+1}\abs{\uk^{n+1}-\uk^{n}}^2+\frac{1}{2\absk}\sum_{\substack{\sink \\ \sigma=K\vert L}}\abssig\abs{\bu^n_L-\uk^n}^2(-\flx^{n+1,-}).
\end{equation}

\subsection{Proof of Theorem \ref{thm:disc-loc-ent-ineq}}
\label{subsec:app-proof-loc-ent}

Multiplying \eqref{eqn:disc-renorm-idt} by $\absk/\veps^2$ and \eqref{eqn:disc-kin} by $\absk$, adding the resulting expressions and then summing over all $K\in\T$ gives us
\begin{equation}
\label{eqn:a3-1}
    \begin{split}
        \sum_{K\in\T}\frac{\absk}{\delt}\biggl(\half\rk^{n+1}\abs{\uk^{n+1}}^2 + &\frac{1}{\veps^2}\psig(\rk^{n+1})\biggr) - \sum_{K\in\T}\frac{\absk}{\delt}\biggl(\half\rk^{n}\abs{\uk^{n}}^2 + \frac{1}{\veps^2}\psig(\rk^{n})\biggr) \\ 
        &+ \sum_{K\in\T}\absk\bigl(R^{n+1}_{K,\delt}+ S^{n+1}_{K,\delt}\bigr) = -\frac{\eta\delt}{\veps^4}\sum_{K\in\T}\absk\abs{(\gradt p^{n+1})_K}^2.
    \end{split}
\end{equation}
Here, we have used the conservativity of the flux and the discrete grad-div duality \eqref{eqn:disc-grad-div-primal}. Note that $\sum_{K\in\T}R^{n+1}_{K,\delt}\geq 0$. 

Using the identity $\abs{a-b}^2\leq 2(\abs{a}^2 + \abs{b}^2)$ and using the velocity update \eqref{eqn:disc-vel-upd} to estimate the first term of $S^{n+1}_{K,\delt}$ allows us to obtain the following estimate, see \cite[Theorem 1.14]{AA24} for further details.
\begin{equation}
\label{eqn:a3-2}
    \begin{split}
        -\sum_{K\in\T}\absk S^{n+1}_{K,\delt}&\leq\sum_{K\in\T}\absk\Biggl(\sum_{\substack{\sink \\ \sigma=K\vert L}}\frac{\abssig}{\absk}\abs{\bu^n_L-\uk^n}^2\flx^{n+1,-}\Biggr)\times\Biggl(\half +\frac{\delt}{\absk\rk^{n+1}}\sum_{\substack{\sink \\ \sigma=K\vert L}} \abssig\flx^{n+1,-}\Biggr) \\
        & + \frac{\delt}{\veps^4}\sum_{K\in\T}\frac{\absk}{\rk^{n+1}}\abs{(\gradt p^{n+1})_K}^2.
    \end{split}
\end{equation}

Combining \eqref{eqn:a3-1} and \eqref{eqn:a3-2} yields
\begin{equation}
\label{eqn:a3-3}
    \begin{split}
        &\sum_{K\in\T}\frac{\absk}{\delt}\biggl(\half\rk^{n+1}\abs{\uk^{n+1}}^2 + \frac{1}{\veps^2}\psig(\rk^{n+1})\biggr) - \sum_{K\in\T}\frac{\absk}{\delt}\biggl(\half\rk^{n}\abs{\uk^{n}}^2 + \frac{1}{\veps^2}\psig(\rk^{n})\biggr) \\ 
        &\leq\sum_{K\in\T}\absk\Biggl(\sum_{\substack{\sink \\ \sigma=K\vert L}}\frac{\abssig}{\absk}\abs{\bu^n_L-\uk^n}^2\flx^{n+1,-}\Biggr)\times\Biggl(\half +\frac{\delt}{\absk\rk^{n+1}}\sum_{\substack{\sink \\ \sigma=K\vert L}} \abssig\flx^{n+1,-}\Biggr) \\
        & + \frac{\delt}{\veps^4}\sum_{K\in\T}\absk\biggl(\frac{1}{\rk^{n+1}} - \eta\biggr)\abs{(\gradt p^{n+1})_K}^2.
    \end{split}
\end{equation}

Under the conditions 
\begin{equation}
\label{eqn:app-cfl-cond}
    \delt\leq\dfrac{\absk\rk^{n+1}}{2\displaystyle\sum_{\substack{\sink \\ \sigma=K\vert L}}\abssig\vrho^{n+1}_L (-(w^n_{\sk})^{-})}
\end{equation}
and $\eta>\frac{1}{\rk^{n+1}}$, we observe that the right-hand side of \eqref{eqn:a3-3} remains non-positive, which yields the desired energy condition \eqref{eqn:disc-loc-ent-ineq-1}. Analogously, if we multiply \eqref{eqn:disc-pos-renorm-idt} by $\absk/\veps^2$ and \eqref{eqn:disc-kin} by $\absk$, add the resulting expressions and sum over all $K\in\T$, we will obtain \eqref{eqn:disc-loc-ent-ineq} by carrying out the exact same analysis.

\subsection{Proof of Theorem \ref{thm:disc-glob-ent-est}}
\label{subsec:app-proof-glob-ent}

For each $r = 0,\dots, N-1$, \eqref{eqn:a3-3} under the timestep restriction \eqref{eqn:app-cfl-cond} guarantees that
\begin{equation}
\label{eqn:a4-1}
    \begin{split}
        \sum_{K\in\T}\frac{\absk}{\delt}\biggl(\half\rk^{r+1}\abs{\uk^{r+1}}^2 + \frac{1}{\veps^2}\psig(\rk^{r+1})\biggr) - &\sum_{K\in\T}\frac{\absk}{\delt}\biggl(\half\rk^{r}\abs{\uk^{r}}^2 + \frac{1}{\veps^2}\psig(\rk^{r})\biggr) \\ 
        &+ \frac{\delt}{\veps^4}\sum_{K\in\T}\absk\biggl(\eta - \frac{1}{\rk^{r+1}}\biggr)\abs{(\gradt p^{r+1})_K}^2 \leq 0.
    \end{split}
\end{equation}

Multiplying the above equation by $\delt$ and for $1\leq n\leq N$, summing over $r$ from $0$ to $n-1$ yields
\begin{equation}
\label{eqn:a4-2}
    \begin{split}
        \sum_{K\in\T}\absk\biggl(\half\rk^n\abs{\uk^n}^2+\frac{1}{\veps^2}\psig(\rk^n)\biggr) &+ \frac{1}{\veps^4}\sum_{r=0}^{n-1}\delt^2\sum_{K\in\T}\abs{K}\biggl(\eta - \frac{1}{\rk^{r+1}}\biggr)\abs{(\gradt p^{r+1})_K}^2 \\
        &\leq \sum_{K\in\T}\absk\biggl(\half\rk^0\abs{\uk^0}^2+\frac{1}{\veps^2}\psig(\rk^0)\biggr)\leq C,
    \end{split}
\end{equation}
where we have used \eqref{eqn:ini-enrg-bound}, which yields \eqref{eqn:disc-glob-ent-est-1}. Analogous calculations will yield \eqref{eqn:disc-glob-ent-est}.

\bibliographystyle{abbrv}
\bibliography{ref}
\end{document}